\documentclass[11pt]{article}

\usepackage{amsmath}

\usepackage{geometry}
\usepackage{indentfirst}
\usepackage{bm}
\usepackage{accents}
\usepackage{changepage}
\usepackage{color}
\usepackage{bigints}
\usepackage{yhmath}
\usepackage{dirtytalk}
\usepackage{tensor}
\usepackage{multirow}
\usepackage{enumerate}
\usepackage{enumitem}
\usepackage{mathtools}
\usepackage{latexsym,amsthm,amsfonts,amssymb,amsmath,amsxtra,mathrsfs,stmaryrd,bm}
\usepackage{figsize,subfigure,graphicx,ifthen,calc,psfrag}
\usepackage{upgreek}
\usepackage[greek,english]{babel}
\usepackage[numbers,sort&compress]{natbib}

\usepackage{comment}
\usepackage{hyperref}
\usepackage{caption}
\usepackage{float}
\usepackage{cleveref}
\usepackage{svg}

\newtheorem{prop}{Proposition}[section]
\newtheorem*{defi*}{Definition}
\newtheorem{defi}{Definition}[section]

\newtheorem{lem}{Lemma}[section]
\newtheorem{rem}{Remark}[section]
\newtheorem{thm}{Theorem}[section]
\newtheorem{coro}{Corollary}[section]
\newtheorem{ex}{Example}[section]

\makeatletter
\newcommand{\mylabel}[2]{#2\def\@currentlabel{#2}\label{#1}}
\makeatother

\numberwithin{equation}{section}

\allowdisplaybreaks

\usepackage{nomencl}
\makenomenclature

\usepackage{etoolbox}
\providetoggle{nomsort}
\settoggle{nomsort}{true} 

\makeatletter
\iftoggle{nomsort}{%
    \let\old@@@nomenclature=\@@@nomenclature        
        \newcounter{@nomcount} \setcounter{@nomcount}{0}%
        \renewcommand\the@nomcount{\two@digits{\value{@nomcount}}}
        \def\@@@nomenclature[#1]#2#3{
          \addtocounter{@nomcount}{1}%
        \def\@tempa{#2}\def\@tempb{#3}%
          \protected@write\@nomenclaturefile{}%
          {\string\nomenclatureentry{\the@nomcount\nom@verb\@tempa @[{\nom@verb\@tempa}]%
          \begingroup\nom@verb\@tempb\protect\nomeqref{\theequation}%
          |nompageref}{\thepage}}%
          \endgroup
          \@esphack}%
      }{}
\makeatother

\usepackage{multicol}
\def\div{\mathop{\operatorname{div}}\nolimits}

\usepackage{algorithm}
\usepackage[noend]{algpseudocode}

\Crefname{coro}{Corollary}{Corollaries}
\Crefname{thm}{Theorem}{Theorems}

\title{Numerical spectral analysis of Cauchy-type inverse problems: \\A probabilistic approach \\[0.5 cm]
}
\author{
Iulian C\^{i}mpean$^{1,3,}$\thanks{Corresponding author. E-mail: \texttt{iulian.cimpean@unibuc.ro;} \href{https://orcid.org/0000-0002-3239-6834}{ORCID ID: 0000-0002-3239-6834}},
Andreea Grecu$^{2,}$\thanks{E-mail: \texttt{andreea.grecu@ismma.ro;} \href{https://orcid.org/0000-0001-5829-7765}{ORCID ID: 0000-0001-5829-7765}},
Liviu Marin$^{1,2,}$\thanks{E-mail: \texttt{liviu.marin@fmi.unibuc.ro;} \href{https://orcid.org/0000-0003-4009-1181}{ORCID ID: 0000-0003-4009-1181}}
}

\date{\small
$^1$Department of Mathematics, Faculty of Mathematics and Computer Science, University of Bucharest, 
14~Academiei, 010014~Bucharest, Romania\\
$^2$``Gheorghe Mihoc~--~Caius Iacob" Institute of Mathematical Statistics and Applied Mathematics of the Romanian Academy, 13~Calea 13 Septembrie, 050711~Bucharest, Romania\\
$^3$``Simion Stoilow" Institute of Mathematics of the Romanian Academy, 21~Calea Grivi\c{t}ei,  010702~Bucharest, Romania 
}
\begin{document}

\maketitle
\begin{abstract}
We investigate the inverse Cauchy and data completion problems for elliptic partial differential equations in a bounded domain $D \subset \mathbb{R}^d$, $d \ge 2$, with a special emphasis on the steady-state heat conduction in anisotropic media. More precisely, boundary conditions are prescribed on an accessible part of the boundary $\varnothing \neq \Gamma_0 \subsetneqq \partial{D}$ and/or internal conditions are available inside the domain $D$ and the aim is to reconstruct the solution to these inverse problems in the domain and on the inaccessible remaining boundary $\Gamma_1 \coloneqq \partial{D} \setminus \Gamma_0$.
Although such severely ill-posed problems have been studied intensively in the past decades, deriving efficient methods for approximating their solution still remains challenging in the general setting, e.g., in high dimensions, for solutions and/or domains with singularities, in complex geometries, etc.
Herein, we derive a fundamental probabilistic framework for the stable reconstruction of the solution to the Cauchy and data completion problems in steady-state anisotropic heat conduction, as well as enhancing the knowledge on the impact of the geometry of the domain $D$ and the structure of the conductivity tensor $\mathbf{K}$ on the stability of these inverse problems.
This is achieved in three steps: ({\it i}) the spectrum of the direct problem is simulated using stochastic estimators; ({\it ii}) the singular value decomposition of the corresponding direct operator is performed; and ({\it iii}) for the prescribed measurements, a natural subspace of approximate solutions is constructed. 
This approach is based on elliptic measures, in conjunction with probabilistic representations and parallel Monte Carlo simulations. 
Thorough numerical simulations performed on GPU, for various two- and three-dimensional geometries, are also provided.
\end{abstract}




\newpage

\section{Introduction} 
\label{section:intro}
\subsection{Formulation of the inverse problem}
\label{ss:formulation}
Consider the numerical treatment of inverse Cauchy-type problems for elliptic partial differential equations in a bounded domain with a special emphasis on steady-state heat conduction in anisotropic media. Hence the main aim is to reconstruct numerically the solution to such a problem from data prescribed merely on some accessible portion of the boundary and/or in the domain. More precisely, we consider a material that occupies a bounded domain $D \subset \mathbb{R}^d$, $d \geq 2$, and is characterised by an inhomogeneous and anisotropic conductivity tensor $\mathbf{K} \coloneqq \big( K_{ij}(\mathbf{x}) \big)_{i,j=\overline{1,d}} \in \mathbb{R}^{d \times d}$, $\mathbf{x}\in D$, such that
\begin{enumerate}[label=({\bf H.\arabic*})]
\item \label{Hyp_D} $\partial{D} = \Gamma_0 \cup \Gamma_1$ with $\Gamma_1\cap\Gamma_0 = \varnothing$ ($\Gamma_0$ and $\Gamma_1$ should be regarded as the accessible and inaccessible parts of the boundary $\partial{D}$, respectively);
\item \label{Hyp_K} $\mathbf{K} \in \mathbb{R}^{d \times d}$ is a symmetric and strictly elliptic matrix of bounded and measurable coefficients.
\end{enumerate}
If $\Gamma_0 \subset \partial D$ is an open and Lipschitz set in the sense of \cite{Al09}, then the classical {\it inverse Cauchy problem}, see \Cref{Fig00a}, reads as: 
\begin{eqnarray}  \label{eq:ICP}
\left|
\begin{array}{l}
\textrm{Given } u_0 \in H^{1/2}(\Gamma_0) \textrm{ and } q_0 \in H^{-1/2}(\Gamma_0), \textrm{ find } u \in H^1(D) \textrm{ the weak solution to }\\[5pt]
\begin{array}{lll}
\qquad 
\div\big(\mathbf{K} \nabla{u}\big) = 0~\textrm{ in } D, \quad & 
u = u_0~\textrm{ on } \Gamma_0, \quad & 
\mathbf{n} \cdot \big(\mathbf{K} \nabla{u}\big) = q_0~\textrm{ on } \Gamma_0,
\end{array}\\[5pt]
\textrm{where } \mathbf{n} \textrm{ denotes the outward unit normal at } \Gamma_0. 
\end{array}
\right.
\end{eqnarray}

Note that recovering $u$ is essentially equivalent to finding the unknown Dirichlet data $u_1 \coloneqq u\big|_{\Gamma_1}$, assuming that the trace of $u$ on $\Gamma_1$ has a consistent meaning.
Then, based on the unique continuation property for elliptic operators, or, more precisely, by assuming $\mathbf{K}$ is Lipschitz for $d \geq 3$ and using \cite[Theorem 1.7]{Al09}, it follows there exists at most one weak solution $u$ to problem \eqref{eq:ICP} for the prescribed pair of data $\big( u_0, q_0 \big) \in H^{1/2}(\Gamma_0) \times H^{-1/2}(\Gamma_0)$.
Moreover, it is well-known that reconstructing $u_1$ from the prescribed measurements on $\Gamma_0$ is a severely unstable problem, see \cite{Ha23}.

The inverse problem \eqref{eq:ICP} only stands as a particular continuous version of the following discrete inverse problem of interest that is actually addressed herein, see \Cref{Fig00b}, namely:
\begin{eqnarray} \label{eq:insidemeasurements}
\left|
\begin{array}{l}
\textrm{Assuming that any } \mathbf{x} \in D \textrm{ is regular (see \Cref{section:formulation}), and given } M_0 \textrm{ boundary points } \\
\big( \mathbf{x}_i^0 \big)_{i= \overline{1,M_0}} \subset \Gamma_0 \textrm{ and } M_D \textrm{ internal points } \big( \mathbf{x}_i^D \big)_{i= \overline{1,M_D}} \subset D, \textrm{ find } u \in C(D) \cap H^1_{\rm loc}(D)\\ \textrm{such that}\\[5pt]
\begin{array}{lll}
\qquad 
\div\big(\mathbf{K} \nabla{u}\big) = 0~\textrm{ in } D, \quad & 
u\big( \mathbf{x}^0_i \big) = u^0_i,~~i = \overline{1, M_0}, \quad & 
u\big( \mathbf{x}^D_i \big) = u^D_i,~~i = \overline{1, M_D},
\end{array}\\[5pt]
\textrm{where the partial differential equation above is satisfied in a distributional sense (see }\\  
\textrm{\Cref{def:local solution}), } \mathbf{u}^0 \coloneqq \big( u^0_i \big)_{i = \overline{1, M_0}} \in \mathbb{R}^{M_0} \textrm{ are boundary measurements in the sense of }\\
\textrm{\Cref{def:local solution bvp}, relation \eqref{eq:limit regular point}, and } 
\mathbf{u}^D \coloneqq \big( u^D_i \big)_{i = \overline{1, M_D}} \in \mathbb{R}^{M_D}  
\textrm{are internal measurements.}  
\end{array}
\right.
\end{eqnarray}

Clearly, an unknown Dirichlet data $u_1 \in C(\Gamma_1)$ or $u_1 \in L^2(\Gamma_1)$ associated with problem \eqref{eq:insidemeasurements} cannot be uniquely determined as an element in the infinite dimensional space $C(\Gamma_1)$ or $L^2(\Gamma_1)$ merely from discrete measurements.
However, if the discrete measurement locations $\big( \mathbf{x}^0_i \big)_{i = \overline{1, M_0}}$ and $\big( \mathbf{x}^D_i \big)_{i = \overline{1, M_D}}$ tend to become dense on the boundary of some open subset of $D$, then by a unique continuation principle, the limit problem is expected to have a unique solution. 
A rigorous treatment of this issue, as well as other aspects related to the continuous data version of \eqref{eq:insidemeasurements} are provided in \Cref{section:formulation}, see, e.g., problem \eqref{eq:insidemeasurementsf0fh}.

Furthermore, on the one hand, system \eqref{eq:insidemeasurements} itself is a relevant model for data completion problems, when the measurements are available merely at some discrete locations on the accessible boundary curve/surface of a domain, as well as (or even only) at some discrete positions in the domain, e.g., very close to the boundary. 
On the other hand, one can easily and accurately transform the classical inverse Cauchy problem \eqref{eq:ICP} into problem \eqref{eq:insidemeasurements}, so that the latter can be seen as a (discrete) generalisation of the former.
To justify this, we firstly recall that by the interior H\"older elliptic regularity \cite[Theorem 8.24]{GiTr01}, it follows that any weak solution $u \in H^1(D)$ to problem \eqref{eq:ICP} is, in fact, H\"older continuous in each compact subset of $D$ and, in particular, has a continuous version on $D$ which we shall tacitly consider further. 
Motivated by the H\"older stability estimate for problem \eqref{eq:ICP} due to \cite[Theorem 1.7]{Al09}, it is reasonable to assume that under the Lipschitz regularity of the domain, locally and in the proximity of $\Gamma_0$, where the Cauchy measurements $\big( u_0, q_0 \big)$ are taken, e.g., on some portion $\Gamma_D \subset D$ such that ${\rm dist}(\mathbf{x}, \Gamma_0) \leq h$, $\mathbf{x} \in \Gamma_D$, for some small $h > 0$, one can accurately reconstruct by standard methods the restricted solution $u_D \coloneqq u\big|_{\Gamma_D}$. 
Hence one immediately arrives at the setting given by problem \eqref{eq:insidemeasurements} or its continuous data analogue \eqref{eq:insidemeasurementsf0fh} discussed in \Cref{section:formulation}. 

From a much simpler perspective, under some additional boundary regularity conditions, one can easily reconstruct $u\big|_{\Gamma_D}$ from the Cauchy (Dirichlet and Neumann) data pair $\big( u_0, q_0 \big)$ on $\Gamma_0$ with high accuracy by employing a Taylor expansion from the boundary to the interior of the domain. 
More precisely, consider 
\begin{eqnarray}
\mathbf{x}_0 \in \Gamma_0, \quad 
h > 0: \qquad 
\mathbf{x}_D \coloneqq \mathbf{x}_0 
- h \, \mathbf{K}\big( \mathbf{x}_0 \big) 
\, \mathbf{n}\big( \mathbf{x}_0 \big) \in D
\end{eqnarray}
and assume that both the first- and the second-order derivatives of the solution $u$ to problem \eqref{eq:ICP} exist at any point of the open segment $\big\{t \mathbf{x}_D + (1 - t) \mathbf{x}_0 \: \big| \: t\in (0,1) \big\}$ and can be continuously extended to its closure $\big\{t \mathbf{x}_D + (1 - t) \mathbf{x}_0 \: \big| \: t\in [0,1] \big\}$, and define 
\begin{equation*}
C \coloneqq \sup_{t \in (0,1)} 
\big\langle \mathbf{K}\big( \mathbf{x}_0 \big) 
D^2u\big( t \mathbf{x}_D + (1 - t) \mathbf{x}_0 \big) \mathbf{K}\big( \mathbf{x}_0 \big) 
\mathbf{n}\big( \mathbf{x}_0 \big), 
\mathbf{n}\big( \mathbf{x}_0 \big) \big\rangle.  
\end{equation*}
If the Neumann measurement at $\mathbf{x}_0$ is assumed to be given precisely by
\begin{equation*}
q_0\big( \mathbf{x}_0 \big) = 
\displaystyle \lim\limits_{t \to 0} 
\mathbf{n}\big( \mathbf{x}_0 \big) \cdot \big( \mathbf{K} \nabla{u} \big)\big(t \mathbf{x}_D + (1 - t) \mathbf{x}_0 \big), 
\end{equation*}
then by a straightforward second-order Taylor expansion of $u$ at $\mathbf{x}_0 \in \Gamma_0$, it follows that
\begin{equation*}
\big\vert u\big( \mathbf{x}_D \big) 
- u\big( \mathbf{x}_0 \big) 
- h q_0\big( \mathbf{x}_0 \big) \big\vert 
\leq \dfrac{C}{2} h^2,
\end{equation*}
i.e. one can reconstruct $u\big( \mathbf{x}_D \big)$ approximately as $u_0\big( \mathbf{x}_0 \big) - h \, q_0\big( \mathbf{x}_0 \big)$ with an $\mathcal{O}\big( h^2 \big)$ absolute error.
This procedure can be easily extended to extrapolate $u$ from the Cauchy data $\big( u_0, q_0 \big)$ given at some discrete locations on $\Gamma_0$ to some discrete positions in the domain $D$ located at a distance of maximum $h > 0$ from the aforementioned boundary points on $\Gamma_0$. 
Consequently, yet again, one can accurately transform problem \eqref{eq:ICP} into problem \eqref{eq:insidemeasurements} and hence its continuous data version \eqref{eq:insidemeasurementsf0fh}.

\subsection{Review of numerical methods for the inverse problem} 
\label{ss:review}
For the sake of brevity, we only refer the reader to a few monographs on inverse problems for partial differential equations and, in particular, related to problem \eqref{eq:ICP}, namely  \cite{al12,go02,gr93,is06,ki11,Le21,ta05}. 
The high interest in this type of inverse problems stems, on the one hand, from the fact that they have various applications related to, e.g., heat transfer \cite{al12} and the references therein, solid mechanics, engineering and thermostatics \cite{bo98,andrieux2005data}, crack detection \cite{alessandrini1993stable,AnAb96,BrHaPi01}, corrosion detection \cite{liu2008modified}, latent flatness defect detection during a rolling process \cite{We15}, seismology or geophysics \cite{al19} and the references therein, electroencephalography \cite{fr05,co08,ma20,clerc2007cortical,koshev2020fem}, and electrocardiography \cite{colli1985mathematical,FrMa79}, and on the other hand, from their usually severe ill-posedness which actually results in challenges related to their corresponding numerical approach, i.e. various regularisation techniques need to be employed, investigated and thoroughly analysed.
 
In contrast to problem \eqref{eq:insidemeasurements} which surprisingly has not been investigated so far and, at the same time, represents the main interest of the present study, the literature devoted to numerical methods for the approximate solution to problem \eqref{eq:ICP} is vast. 
Since the above problems \eqref{eq:ICP} and \eqref{eq:insidemeasurements} are numerically equivalent, we further provide the reader with a brief review of the existing approaches to the former. 
On a broader scale, the aforementioned numerical methods may be divided into direct, iterative and Bayesian methods, and these are further discussed for clarity and comparison purposes.

\paragraph{Direct methods.}{Direct methods are usually based on a convenient representation of a direct operator mapping the unknown solution to the prescribed data that is further discretised using, e.g., the finite-difference, boundary element and finite element methods or the method of fundamental solutions. 
Due to the instability of the problem, the resulting discrete system is then regularised, e.g., by the Tikhonov or the Lavrentiev regularisation method, the truncated singular value decomposition (TSVD), or a projection on a certain subspace of possible solutions. 
For some approaches available in the literature, we refer the reader to \cite{ben2012local,belgacem2018analysis} and \cite{ref1tikhonov,ref2engl} in case of the Lavrentiev and the Tikhonov regularisation methods, respectively, which provide a regularising framework for both direct and iterative methods, as well as \cite{berntsson2001numerical,JIN06a,azaiez2011finite,go23,jin2006method,ref4marin,ref5marin,ref7marin} and the references therein which are related to direct methods only. 
An alternative way to regularise the inverse problem \eqref{eq:ICP} is given by using either a certain discretisation scheme which has itself a stabilising effect, see 
\cite{reinhardt1999stability}, or the quasi-reversibility method \cite{Klibanov} which builds upon the ideas introduced by \cite{lattes1969method}.}

\paragraph{Iterative methods.}{In contrast to direct methods, iterative methods produce a sequence of approximate solutions to the inverse problem \eqref{eq:ICP}. This sequence is usually obtained by ({\it i})~the alternating procedure originally introduced in \cite{kozlov1991iterative} and its corresponding extensions and applications to (an)isotropic heat conduction \cite{ref8lesnic,mera2000boundary}, (an)isotropic elasticity \cite{ref9marin,ref11comino} and Helmholtz-type equations \cite{ref12berntsson,ref10marin}, and acceleration versions \cite{ref16Berntsson,ref17bucataru,ref13jourhmane,ref14marin,ref15marin}; 
({\it ii})~a variational approach that transforms the original inverse problem \eqref{eq:ICP} into an equivalent control one and yields either a conjugate gradient \cite{ref18hao,ref19marin,ref20marin}, a steepest descent (Landweber-Fridman) \cite{bucataru2022gradient,bucataru2023stable,ref22marin,ref23marin}, or a minimal error algorithm \cite{ref24johansson,ref25marin,ref26marin}; 
({\it iii})~minimising an energy error functional whose optima or critical points are solutions to the original inverse problem \eqref{eq:ICP} \cite{andrieux2006solving,ref28andrieux,ref29baranger,rischette2013regularization}; 
({\it iv})~reformulating the original inverse problem \eqref{eq:ICP} as a Steklov-Poincar\'{e} one via two direct problems that make use of the prescribed Dirichlet and Neumann data, respectively, \cite{Be05,AzBeEl06}; 
({\it v})~the iterated quasi-reversibility method that consists of determining iteratively the solutions of quasi-reversibility problems depending on the previous ones \cite{darde2015iterated}; 
({\it vi})~the fading regularisation method that reduces the original problem \eqref{eq:ICP} to a sequence of well-posed optimisation ones  \cite{ref32cimetiere,ref34delvare,ref35marin,voinea2021fading}; 
({\it vii})~Nash game strategies to recover the solution to the inverse problem \eqref{eq:ICP} as the Nash equilibrium \cite{habbal2013neumann}; 
({\it viii})~a fixed-point iteration associated with a suitable regularising operator \cite{johansson2004iterative,baravdish2018iterative}; 
and ({\it ix})~other methods \cite{ka95,belgacem2011extended,abouladdotiotach2008missing,caubet2020dual,chakib2006convergence,ben2022full}.
An important feature of these iterative methods is that the regularisation effect is an intrinsic part of the numerical scheme and no further regularisation is, in principle, required for exact data. 
More precisely, the stabilising effect of the iterative methods comes from three main sources: the initialisation, the smoothing effect of the operator that updates the current element of the sequence, and finally, the rule that stops the iterations after a suitable number of steps, provided that the data is noisy/perturbed.
Widely used stopping criteria are based on the discrepancy principle of Morozov, Hansen's $L-$curve method and the generalised cross-validation.}

\paragraph{Bayesian methods.}{The Bayesian approach consists of regarding the unknown solution to \eqref{eq:ICP} or an inverse problem in general as a latent random variable with values in some convenient function space, with a given prior distribution that is aimed to capture the {\it a priori} knowledge of the unknown data, e.g., its regularity or some known bounds. 
Then, the method focuses on deriving the posterior distribution given the available data/measurements with the latter viewed as noisy observations of the unknown solution. 
If required, Monte Carlo Methods may also be employed to sample the posterior distribution. 
Consequently, the Bayesian approach provides one with a a probabilistic framework that incorporates the direct operator inside the distribution of the observations conditioned on the unknown/latent data and restores a distribution of possible solutions. 
The stabilising effect of the Bayesian method on inverse problems is, in principle, entailed by the presence of noise in the observations and the prior distribution imposed on the unknown solution as parts of the model. However, additional regularisation may also be performed.
For some fundamental investigations related to this direction and some applications of this approach to inverse Cauchy-type problems \eqref{eq:ICP}, we refer the reader to \cite{st10,Vaart11,idier2013bayesian} and the references therein, and  \cite{jin2008bayesian,kaipio2011bayesian}, respectively. 

It should be mentioned that the methods developed herein, although probabilistic, are completely different from the aforementioned Bayesian method.
More precisely, the methods proposed in the present study rely upon the probabilistic description of the heat particle inside a domain in order to represent intrinsically the direct operator and are completely different from the Bayesian approach wherein the direct operator is part of an extended probabilistic model.}

\subsection{Challenges, strategy and aims of the present approach}
\label{ss:challenges}
\paragraph{Challenges.}{In general, the inverse Cauchy problem \eqref{eq:ICP} is severely ill-posed \cite{Ha23} and, therefore, this feature is also transmitted to its discrete reformulation \eqref{eq:insidemeasurements}. Under some {\it a priori} bounds on the unknown solution $u$ to the inverse problem \eqref{eq:ICP}, see \cite{Al09}, or by some regularisation procedures, one may expect to reconstruct it from the knowledge of the Cauchy data on $\Gamma_0$, more or less accurately, in a domain $D^\prime \subset D$ with ${\rm dist}(\partial{D^\prime}, \Gamma_1) > 0$ sufficiently large.
However, there are several fundamental issues that need to be addressed carefully when aiming at the numerical reconstruction of the unknown Dirichlet data on $\Gamma_1$: 
\begin{enumerate}[label={\rm (}{\it \alph*}{\rm )}]
\setlength\itemsep{1pt}
\item If the {\it a priori} knowledge on $u$ is imprecise or the regularisation procedure induces unrealistic constraints, then although a corresponding numerical algorithm may return a stable approximate solution, this may still significantly (and stably) deviate from the true solution. 

\item The instability of the inverse problem is not uniform on the inaccessible boundary $\Gamma_1$, but it highly depends on the geometry of the domain. Thus, a customized regularising procedure for the inverse problem should depend on the distribution of the instability on $\Gamma_1$.

\item The assumption that the Cauchy data $\big( u_0, q_0 \big)$ is fully available on the accessible boundary $\Gamma_0$ is rarely satisfied in practice. In fact, in real-life problems, $u_0$ and/or $q_0$ are often available only at some discrete locations on $\Gamma_0$. 
Most of the available numerical methods require an interpolation of the measurements on the entire $\Gamma_0$ and the interpolation error may go beyond a critical threshold value that destroys the often fragile stability obtained by imposing some {\it a priori} bounds or using a regularisation method.

\item Moreover, realistic discrete measurements are almost always affected by noise, which further contributes to the instability of the inverse problem. A critical situation occurs when the Neumann data (the normal heat flux) $q_0$ is not {\it a priori} given as part of the model or directly retrieved from measurements, but obtained, instead, from a finite-difference formula for two close noisy measurements. In this case, the measurement errors are magnified by noisy numerical differentiation and this may, clearly, further minimise the reliability of the reconstructed unknown solution.

\item The geometry of the domain, structure of the coefficients and possible singularities of the solutions also have a crucial impact upon the stability of the inverse problem. 
However, to the best of our knowledge, these profound aspects have remained fundamentally unexplored in the literature so far.
\end{enumerate}
In spite of all these crucial issues, the numerical methods for the inverse problem \eqref{eq:ICP}, developed and tested intensively in the past decades and briefly presented in \Cref{ss:review}, have shown a surprising success at least when applied to certain domains or benchmark solutions. 
As mentioned above, it has still remained open the question of whether and under which conditions these various numerical methods are indeed able to recover, in general, the true solution to problem \eqref{eq:ICP}. 
Moreover, as stated in \cite[Conclusion, pp. 578]{ben2022full}, it is an up-to-date challenge to derive both theoretical and numerical methods that can efficiently treat the inverse problem \eqref{eq:ICP} in a general setting, especially in high dimensions, when the solution or the domain exhibits singularities, and/or in complex geometries.}

\paragraph{Strategy.}{The strategy of the present approach to the inverse problem investigated herein focuses on the interplay between the shape of the domain $D$ and the structure of the diffusion coefficients, but not necessarily on their smoothness, which plays a fundamental role in the stability of the inverse problem \eqref{eq:ICP} and seems, to the best of our knowledge, to have been treated poorly or not at all in the literature. 
To clarify this perspective, we start off from the informal statement that the knowledge of the Cauchy data $\big( u_0, q_0 \big)$ on $\Gamma_0$ is essentially equivalent to that of the solution $u$ in an arbitrary thin shell of $\Gamma_0$, for example, $\Gamma_D \coloneqq \big\{ \mathbf{x} \in D \: \big| \: {\rm dist}(\mathbf{x}, \Gamma_0) \leq \varepsilon \big\} \subset D$ for some $0< \varepsilon \ll 1$ or even on its boundary  $\partial{\Gamma_D}$. 
We further let $u_1 = u\big|_{\Gamma_1}$ be the unknown Dirichlet data, so that $u\big|_{\Gamma_D}$ is represented as the integral of $u\big|_{\partial D} = \big(u_0, u_1 \big)$ on $\partial D$ with respect to the elliptic (harmonic) measure $\mu_\mathbf{x}$, $\mathbf{x} \in \Gamma_D$ (for details, see \Cref{section:formulation}), namely
\begin{equation*}
u(\mathbf{x}) 
\coloneqq \int_{\Gamma_0} u_0(\mathbf{y}) \, \mathrm{d} \mu_\mathbf{x}(\mathbf{y}) 
+ \int_{\Gamma_1} u_1(\mathbf{y}) \, \mathrm{d} \mu_\mathbf{x}(\mathbf{y}), 
\quad \mathbf{x} \in \Gamma_D.
\end{equation*}
Consequently, the contribution of the unknown boundary data $u_1$ to the values of $u$ in $\Gamma_D$ and hence to the values of the Cauchy data $\big( u_0, q_0 \big)$ is transferred through the elliptic measure restricted to $\Gamma_1$ with the poles at $\mathbf{x} \in \Gamma_D$ close to $\Gamma_0$. 
Therefore, the stability of the inverse problem \eqref{eq:ICP} is determined by how $\Gamma_1$ is charged by the elliptic measures $\mu_{\mathbf{x}}$ with the poles in the proximity of $\Gamma_0$.
In probabilistic terms, the stability issue is fundamentally characterised by the {\it distribution} on $\Gamma_1$ of the underlying diffusion process started in the proximity of $\Gamma_0$ at its first exit from $D$.

We emphasize that the elliptic measure $\mu_{\mathbf{x}}$ stands as a very fine tool used in the analysis of elliptic PDEs and Potential Theory, yet there is no explicit formula describing it in general, except for some very particular geometries and diffusion coefficients. 
For example, it is well known that if $D=B(0,1)$ is the unit ball in $\mathbb{R}^d$, $d\geq 2$, and $\mathbf{K}$ is the identity matrix, then $\mu_{\mathbf{x}}$, $\mathbf{x}\in D$, is absolutely continuous with respect to the surface measure $\sigma$ and
\begin{equation*}
    \frac{d\mu_{\mathbf{x}}}{d\sigma}(\mathbf{y})=\sigma(\partial B(0,1))^{-1}\frac{1-\|\mathbf{x}\|^2}{\|\mathbf{x}-\mathbf{y}\|^d}, \quad \mathbf{y}\in \partial B(0,1).
\end{equation*}
However, such a simple formula is not available for general domains; even for an annulus it is well known that the harmonic measure can be represented only as an infinite sum whose terms involve {\it homogeneous harmonic polynomials}. 
Consequently, a key part of our methods concerns the problem of efficiently representing and simulating the elliptic measure for more general domains and possibly anisotropic conductivity coefficients.
}

\paragraph{Numerical example.}{The above strategy is illustrated by considering a two-dimensional isotropic material, i.e. $\mathbf{K} = \mathbf{I}_2$, that occupies a disc perforated by five smaller discs, see \Cref{Fig001}. 
Here, the accessible boundary $\Gamma_0$ is the outer one, i.e. the large circle, hence the inaccessible boundary $\Gamma_1$ consists of the inner ones, namely the five small circles. 
The prescribed Cauchy measurements (Dirichlet and Neumann data) on $\Gamma_0$ may be transformed with a little loss by keeping the Dirichlet data on $\Gamma_0$ and replacing the Neumann data on $\Gamma_0$ with discrete Dirichlet measurements on the dotted circle showed in \Cref{Fig001}. 
The harmonic measure with the pole at a generic location $\mathbf{x}^D$ on the dotted circle can be simulated numerically, see \Cref{s:3} for technical details, and its distribution restricted to $\Gamma_1$ is displayed in \Cref{Fig001a}. Moreover, one can also compute $\mu_{\mathbf{x}^D}(\Gamma_1) \approx 0.0961$ and hence obtain $\mu_{\mathbf{x}^D}(\Gamma_0) \approx 0.9039$. 
The restriction to $\Gamma_1$ of the averaged harmonic measures with the poles at all points represented by the dotted circle is illustrated in \Cref{Fig001b}, whilst the average harmonic measure of $\Gamma_1$ can, yet again, be computed to be approximately $0.1126$ and thus the average harmonic measure of $\Gamma_0$ is obtained to be approximately $0.8874$. 
This example and the corresponding computations illustrate that the contribution of the unknown Dirichlet data on $\Gamma_1$ to the measurements taken on the dotted circle is, in a broad sense, proportional to $0.1126$. 
In other words, a variation (in the sup-norm) of the intensity $\delta$ from the true unknown Dirichlet data on $\Gamma_1$ is reduced (smoothed out) to a variation of at most $\delta \times 0.1126$ of the measurements taken on the dotted circle. Clearly, the smaller the contribution of $\Gamma_1$ to the measurements taken, the higher the instability of the inverse problem. 
It can also be clearly noticed from \Cref{Fig001b} that the contribution of $\Gamma_1$ is not uniform, in the sense that, on average, the portions of the small circles that are closer to the location of the measurements contribute more to the latter than their corresponding remaining parts depicted in \Cref{Fig001b}. This amounts to saying that the reconstruction of the unknown Dirichlet data on $\Gamma_1$ is not uniformly unstable on $\Gamma_1$, but rather less unstable in the red regions and more unstable in the blue ones. 
Note that the red and blue regions in this example may have a totally different configuration, provided that $\mathbf{K}$ is either non-homogeneous or anisotropic. 
{\it Consequently, the interplay between the geometry of the domain and the structure of the conductivity tensor $\mathbf{K}$ has a crucial impact on the instability of the inverse problem investigated herein and this is the main feature the methods proposed in this study aim to capture this accurately.}} 

\paragraph{Aims.}{The present study goes beyond the mere reconstruction of an approximate solution to the inverse problem $\eqref{eq:ICP}$ or, more precisely, to its corresponding discrete version \eqref{eq:insidemeasurements}. Its main aims refer to domains $D \subset \mathbb{R}^d$, $d \geq 2$, subject to mild regularity assumptions, and read as follows:
\begin{enumerate}[label={\rm (}{\it \alph*}{\rm )}]
\setlength\itemsep{1pt}
\item Estimate numerically the density of the elliptic measure on $\Gamma_1$, with prescribed poles inside $D$, e.g., in the proximity of $\Gamma_0$, which would give a valuable insight into the location of the measurements with respect to the inaccessible portion of the boundary $\Gamma_1$.
\item Reconstruct the spectrum of the symmetrised direct operator associated with \eqref{eq:insidemeasurements}. 
\item Quantify numerically the ill-posedness of \eqref{eq:insidemeasurements} by means of the elliptic measure and the spectrum of the aforementioned operator.
\item Reconstruct a spectrum of solutions to problem \eqref{eq:insidemeasurements} by employing the singular value decomposition (SVD) of the corresponding direct operator.
\item Derive a memory efficient meshfree Monte Carlo method that results in a parallel numerical algorithm easily implemented on GPU and is aimed at representing/learning efficiently the direct operator and its spectrally truncated inverse. 
Once the direct/inverse operator is represented, it can be used to solve the corresponding inverse problem for any given set of appropriate measurements.
\end{enumerate}}

\subsection{Structure of the paper} 
\label{ss:structure}
The rest of the paper is structured as follows. 
Section~\ref{section:formulation} is devoted to the rigorous introduction of various notions of solutions to the inverse problem with continuous data \eqref{eq:insidemeasurementsf0fh}, regarded as the continuous counterpart of the original inverse problem with discrete boundary and internal data \eqref{eq:insidemeasurements}, as well as their representation by means of elliptic measures. The uniqueness of the solution to problem \eqref{eq:insidemeasurementsf0fh} is proved (\Cref{prop:uniqueness IP}) and the compactness of the corresponding direct operator defined by \eqref{eq:compact operator} is also shown (\Cref{coro:compact}). 

In \Cref{ss:convergence_spectrum} the spectral structure of the direct operator \eqref{eq:compact operator} is investigated, namely its restriction to a finite set of locations in the domain where the discrete interior measurements are available, see \eqref{eq:Th}. The spectrum of the symmetrised direct operator \eqref{eq:B Th*Th} is analysed (\Cref{coro:B-Lambda}) and a representation for the TSVD solutions to the corresponding operator problem \eqref{eq:TTh} is derived in terms of the elliptic densities (\Cref{coro:representation_u(r)}). \Cref{ss:discret} is concerned with a brief description of the three steps required for approximating the elliptic densities and the related matrix operators. The first step is performed therein by introducing a proper discretisation of the inaccessible boundary $\Gamma_1$ that yields matrix $\boldsymbol{\Lambda}^{\boldsymbol{\nu}}_{\omega^1}$, see \eqref{eq:Lambda_hat_nu}, i.e. a numerical approximation of matrix $\boldsymbol{\Lambda}^{\boldsymbol{\nu}}$ associated with the symmetrised direct operator \eqref{eq:B Th*Th}, whilst the numerical approximation of the $r-$TSVD solution $u^{(r)}$ to the operator equation \eqref{eq:TTh} given by \eqref{eq:u_r} reduces to the $(\omega^1,r)-$TSVD approximate solution $u_{\omega^1}^{(r)}$ given by \eqref{eq:u_omega_M1_r} or \eqref{eq:u_omega_matrix_form}. 
Moreover, both these matrix and solution approximations are constructed from the unknown matrix $\mathbf{A}_{\omega_1} = \left(\mu_{\mathbf{x}_i^D}(\omega^1_j)\right)_{\substack{i=\overline{1,M_D}\\  j=\overline{1,M_1}}} \in \mathbb{R}^{M_D \times M_1}$ defined by relation \eqref{eq:AM matrix components}.

The main aim of \Cref{s:3} is to approach the remaining two steps mentioned in \Cref{ss:discret} and approximate numerically the entries of $\mathbf{A}_{\omega_1}$ based upon the probabilistic representation of the elliptic measure, the approximation via random walk-on-ellipsoids, and the corresponding Monte Carlo estimators for the elliptic densities and the related matrix operators. Such estimators depend fundamentally on two parameters, namely the thickness $\varepsilon$ of the shell of the boundary at which the random walk-on-ellipsoids is stopped and the number $N$ of independent and identically distributed (i.i.d.) samples used to simulate the random functions or random matrices that estimate the quantities of interest, respectively. 
The following {\it main estimators} are derived in \Cref{s:3}: ({\it i})~the random matrix $\boldsymbol{\Lambda}^{\boldsymbol{\nu}}_{\omega^1, \varepsilon, N}$, whose spectrum approximates that of the symmetrised direct operator \eqref{eq:B Th*Th} associated with problem \eqref{eq:insidemeasurements}; ({\it ii})~the random functions $\rho_{\mathbf{x}_i^D, \omega^1, \varepsilon, N}$ that estimates the elliptic density with poles at the measurement locations $\big( \mathbf{x}_i^D \big)_{i = \overline{1,M_D}}$; and ({\it iii})~the random functions $u_{\omega^1, \varepsilon, N}^{(r)}$ that approximate the $r-$TSVD solution to problem \eqref{eq:u_r} and hence to the original problem \eqref{eq:insidemeasurements}.

As a direct consequence of Sections~\ref{ss:convergence_spectrum} and \ref{s:3}, the resulting algorithm is described in \Cref{ss:algorithm} and its accuracy, convergence, stability and efficiency when solving inverse Cauchy problems for isotropic and anisotropic materials in two- and three-dimensional multiply-connected domains with (piecewise) smooth boundaries are thoroughly investigated and tested in Sections~\ref{ss:2Dexamples} and \ref{ss:3Dexamples}, respectively. 

Concluding remarks and further developments of this study are presented in the final \Cref{section:conclusions}, whilst several choices for the boundary weights and the interpolation employed for discretising the inaccessible boundary $\Gamma_1$, see \Cref{ss:discret}, are presented in \Cref{appendix}.


We conclude this section by mentioning that the probabilistic methods introduced and tested in this study are substantiated by a comprehensive convergence and stability analysis systematically developed in \cite{CiGrMaII}. For a unified version of the present paper and the aforementioned one, we also refer the reader to \cite{CiGrMaExtended}.


\section{Inverse Cauchy problem with discrete measurements: Solutions, continuous data reformulation, uniqueness and ill-posedness} 
\label{section:formulation}
This section is devoted to the introduction, in a rigorous manner, of various notions of solutions to the inverse problem with continuous data \eqref{eq:insidemeasurementsf0fh}, which is regarded as the continuous counterpart of the original inverse problem with discrete boundary and internal data \eqref{eq:insidemeasurements}, as well as its analysis.

Let $D \subset \mathbb{R}^d$ be a bounded domain and $\mathbf{K} : D \longrightarrow \mathbb{R}^{d \times d}$ be a bounded measurable, symmetric and strictly elliptic (i.e. positive definite) matrix-valued function. 

\begin{defi}\label{def:local solution}
A function $u \in H^1_{loc}(D)$ is called a local solution of $\div \big(\mathbf{K} \nabla \big)$ in $D$ if 
\begin{equation*}
\int_D \nabla{u}(\mathbf{x}) \cdot \big(\mathbf{K}(\mathbf{x}) \nabla{\varphi}(\mathbf{x})\big) \, \mathrm{d}\mathbf{x} = 0, \quad \forall~\varphi \in C^1_c(D),
\end{equation*}
where $C^1_c(D)$ denotes the set of continuously differentiable functions with a compact support in $D$.
\end{defi}

\begin{thm}[\cite{DeGiorgi57,Nash58}] \label{thm:DeGiorgi Nash}
Any local solution of $\div \big(\mathbf{K} \nabla \big)$ in $D$ admits a H\"older continuous version in any compact subdomain of $D$.
\end{thm}

Note that a simpler proof of this result due to De Giorgi \cite{DeGiorgi57} and Nash \cite{Nash58} can be found in Moser \cite{Moser60}. Consequently, according to Theorem~\ref{thm:DeGiorgi Nash}, any local solution $u$ of $\div \big(\mathbf{K} \nabla \big)$ in $D$ will automatically be considered continuous on $D$.

\begin{defi}\label{def:local solution bvp}
A function $u \in H^1_{loc}(D)$ is called a solution of the Dirichlet problem
\begin{equation}\label{eq:dirichlet bvp}
    \div \big(\mathbf{K} \nabla{u}\big) = 0  \mbox{ in } D, \quad
    u = f  \mbox{ on } \partial D,
\end{equation}
where $f \in C(\partial D)$, if $u$ is a local solution of $\div \big(\mathbf{K} \nabla \big)$ in $D$ and 
\begin{equation}\label{eq:limit regular point}
\lim_{\substack{\mathbf{x} \to \mathbf{y} \\ \mathbf{x} \in D}} u (\mathbf{x}) = f(\mathbf{y}), 
\quad \forall\; \mathbf{y}\in \partial D \mbox{ regular point}.
\end{equation}
\end{defi}

For a precise definition of a regular boundary point $\mathbf{y} \in \partial D$, for a strictly elliptic operator in divergence form in $D$, with merely bounded measurable coefficients and a characterisation of a regular boundary point via barrier functions, we refer the reader to Stampacchia \cite[Definition 3.2 and Lemma 3.2]{Stampacchia63}, respectively. Littman et al. \cite[Corollary 9.1]{Stampacchia63} state that a boundary point $\mathbf{y} \in \partial D$ is regular for a strictly elliptic operator (in divergence form) in $D$ with merely bounded measurable coefficients iff $\mathbf{y} \in \partial D$ is regular for the Laplace operator previously characterised entirely by Wiener \cite{Wiener24}. Wiener's criterion is generally difficult to check in practice since it requires the computation of the capacity of sets. However, geometric assumptions on $\partial D$ that ensure that every boundary point is regular (e.g., for any domain $D$ satisfying the exterior cone condition) are provided by Gilbarg and Trudinger \cite[Sections 2.8 and 2.9]{GiTr01}.


\begin{thm} \cite[Theorem 1.1]{ChenZhao95} \label{thm:uniquenss local solution bvp}
Suppose that $D \subset \mathbb{R}^d$ is a bounded domain and $f \in C(\partial D)$. 
Then there exists a unique solution to the Dirichlet problem \eqref{eq:dirichlet bvp} in the sense of \Cref{def:local solution bvp}. 
Moreover, for any $\mathbf{x} \in D$, there exists a (unique) probability measure $\mu_\mathbf{x}$ on $\partial D$ such that 
\begin{equation}\label{eq:representation_C}
    u(\mathbf{x}) = 
    \int_{\partial D} f(\mathbf{y}) \, \mathrm{d} \mu_{\mathbf{x}}(\mathbf{y}), 
    \quad \forall~f \in C(\partial D),
\end{equation}
where $u\in H^1_{loc}(D) \cap C(D)$ is the corresponding solution to the Dirichlet problem \eqref{eq:dirichlet bvp}.
\end{thm}

\begin{defi}\label{def:elliptic_measure}
In the framework of \Cref{thm:uniquenss local solution bvp}, for each $\mathbf{x} \in D$, the probability measure $\mu_\mathbf{x}$ is called the elliptic measure associated with $\div \big(\mathbf{K} \nabla \big)$ in $D$ with the pole $\mathbf{x} \in D$. 
In the particular case of the Laplace operator $\Delta$, $\mu_\mathbf{x}$ is referred to as the harmonic measure.
\end{defi}

\begin{rem}\label{rem: other construction}
We emphasise that the existence of the elliptic measure $\mu_\mathbf{x}$ with the pole $\mathbf{x} \in D$ is actually entailed by the probabilistic representation of solutions obtained in \cite[Theorem 1.1]{ChenZhao95}, where $\mu_\mathbf{x}$ is given by the hitting distribution on $\partial D$ of the underlying diffusion process started from $\mathbf{x} \in D$.

There is another classical construction of the elliptic measure, see \cite{Dahlberg86}, usually obtained by assuming that the Dirichlet problem for $\div \big(\mathbf{K} \nabla \big)$ with continuous boundary data is solvable in $D$, in the sense that for any $f \in C(\partial D)$, there exists a local solution $u \in H^1_{loc}(D)$ of $\div \big(\mathbf{K} \nabla{u}\big) = 0$ in $D$ that is continuous in $\overline{D}$ and satisfies $u|_{\partial D} = f$.
In this case, the existence of the elliptic measure can be obtained by the maximum principle and the Riesz representation theorem. 
\end{rem}

In fact, the right-hand side of relation \eqref{eq:representation_C} can be used to define a larger class of solutions even if $\partial D$ contains non-regular points, see \Cref{prop:local solution Lq'} below.

\begin{prop}\label{prop:local solution Lq'}
The following assertions hold.
\begin{enumerate}[label={\rm (}\roman*{\rm )}]
    \item Any two measures $\mu_\mathbf{x}$ and $\mu_\mathbf{y}$, where $\mathbf{x}, \mathbf{y} \in D$ and $\mathbf{x} \neq \mathbf{y}$, are mutually absolutely continuous.
    \item Assume that there exists $\mathbf{x}_0 \in D$ such that $f\in L^1(\mu_{\mathbf{x}_0})$. Then $f\in L^1(\mu_\mathbf{x})$, $\forall~\mathbf{x} \in D$, and
    \begin{equation}\label{eq:L1-representation}
    u(\mathbf{x}) \coloneqq \int_{\partial D} f(\mathbf{y}) \, \mathrm{d} \mu_\mathbf{x}(\mathbf{y}), 
    \quad \mathbf{x} \in D,
    \end{equation}
    is a local solution of $\div \big(\mathbf{K} \nabla{u}\big) = 0$ in $D$ and, at the same time, is locally H\"older continuous.
\end{enumerate}
\end{prop}
\begin{proof}~\linebreak
\noindent ({\it i})~Let $\mathbf{x}, \mathbf{y} \in D$, $\mathbf{x} \neq \mathbf{y}$ and $A \subset \partial D$ be measurable such that $\mu_\mathbf{x}(A) = 0$. Since $\mu_\mathbf{x}$ is a Borel measure, there exists a decreasing sequence of open sets $A \subset A_n \subset \partial D$, $n \ge 1$, such that $\mu_\mathbf{x}(A_n) \leq 1/n$, $n \geq 1$. 
For any $n \geq 1$, there exists an increasing sequence of functions $\big(h_{n,k}\big)_{k \geq 1} \subset C(\partial D)$ such that $1_{A_n} = \sup\limits_{k \geq 1} h_{n,k}$.
Then, by \Cref{thm:uniquenss local solution bvp}, it follows that
\begin{equation*}
u_{n,k}(\mathbf{z}) \coloneqq \int_{\partial D} h_{n,k}(\mathbf{y}) \, \mathrm{d} \mu_\mathbf{z}(\mathbf{y}), 
\quad \mathbf{z} \in D, \quad n, k \ge 1,
\end{equation*}
is a solution to the Dirichlet problem \eqref{eq:dirichlet bvp} with $f = h_{n,k}$, $n, k \geq 1$. 
Consequently, by the Harnack inequality \cite[Corollary~8.21]{GiTr01}, 
\begin{equation*} 
    \exists~C > 0,~\textrm{independent of $n$ and $k$:} \quad 
    u_{n,k}(y) \leq C u_{n,k}(x), \quad n, k \geq 1.
\end{equation*}
By letting $k \to \infty$, it follows that $\mu_\mathbf{y}(A_n) \leq 1/n$, $n \geq 1$, and hence $\mu_\mathbf{y}(A) = 0$.

\bigskip

\noindent ({\it ii})~By linearity, one may assume that $f \geq 0$. 
We further set $f_n \coloneqq f \wedge n$ so that $f_n\in L^1(\mu_\mathbf{x})$, for all $\mathbf{x} \in D$ and $n\geq 1$. 
For each $n\geq 1$, let $\big(f_{n,k}\big)_{k\geq 1} \subset C(\partial D)$ be uniformly bounded such that $\lim\limits_{k \to \infty} f_{n,k} = f_n$ a.e. with respect to $\mu_\mathbf{x}$, for a given $\mathbf{x} \in D$ (and hence for all $\mathbf{x} \in D$), and set 
\begin{equation*}
u_{n,k} (\mathbf{x}) \coloneqq  \int_{\partial D} f_{n,k}(\mathbf{y}) \, \mathrm{d} \mu_\mathbf{x}(\mathbf{y}), 
\quad \mathbf{x} \in D, \quad n, k \geq 1.
\end{equation*}
Yet again, by \Cref{thm:uniquenss local solution bvp}, it follows that $u_{n,k}$ is a solution to the Dirichlet problem \eqref{eq:dirichlet bvp} with $f = f_{n,k}$, $n, k \geq 1$. 
Employing the Harnack inequality \cite[Corollary~8.21]{GiTr01} yields that for any compact set $D^\prime \subset D$ with $\mathbf{x}_0 \in D^\prime$, there exists a constant $C > 0$ such that
\begin{equation}\label{eq:sup_u_nk}
\sup\limits_{\mathbf{x} \in D^\prime} u_{n,k}(\mathbf{x}) \leq C u_{n,k}(\mathbf{x}_0), 
\quad n, k \geq 1.
\end{equation}
Consequently,
\begin{equation*}
u(\mathbf{x}) 
= \sup\limits_{n \ge 1} \int_{\partial D} f_n(\mathbf{y}) \, \mathrm{d} \mu_\mathbf{x}(\mathbf{y}) 
= \sup\limits_{n \ge 1} \lim\limits_{k \to \infty} \int_{\partial D} f_{n,k}(\mathbf{y}) \, \mathrm{d} \mu_\mathbf{x}(\mathbf{y}) 
\leq C u(\mathbf{x}_0), 
\quad \mathbf{x} \in D^\prime.
\end{equation*}
Since $D^\prime$ was arbitrarily chosen, it follows that $f \in L^1(\mu_\mathbf{x})$, $\forall~\mathbf{x} \in D$.

Moreover, by setting $u_{n,k}(\mathbf{x}) \coloneqq \displaystyle \int_{\partial D} f_{n,k}(\mathbf{y}) \, \mathrm{d} \mu_\mathbf{x}(\mathbf{y})$, $\mathbf{x} \in D$, $n, k \ge 1$, one obtains that the sequence of solutions $\big(u_{n,k}\big)_{n, k \ge 1}$ is locally uniformly bounded in $D$, hence by \cite[Lemma]{GiTr01}, it is uniformly bounded in $H^1(D^\prime)$ for any compact set $D^\prime \subset D$. 
By passing to a subsequence which is weakly convergent in $H^1(D^\prime)$ and using the dominated and monotone convergence, one eventually obtains that $u$ is a local solution in the sense of \Cref{def:local solution}. 

We finally prove that $u$ is H\"older continuous. 
From relation \eqref{eq:sup_u_nk} and \cite[Theorem 8.24]{GiTr01}, it follows that $u_{n,k}$ is H\"older continuous locally in $D$, with the H\"older constant uniformly bounded with respect to $n, k \geq 1$. The claim now follows since $\lim\limits_{n \to \infty} \lim\limits_{k \to \infty} u_{n,k} = u$ pointwise in $D$. 

\end{proof}

Consider now a compact set $\Gamma_D \subset D$, where the internal measurements are available, such that 
\begin{enumerate}[label=({\bf H.3})]
\item \label{eq:I} $\exists~\varnothing \neq I \subset D$ an open set with $\partial I \subset \Gamma_D \cup \Gamma_0$.   
\end{enumerate}
We are now in the position to state and prove the uniqueness result concerning the solutions from $H^1_{loc}(D) \cap C(D)$ to the following inverse problem
\begin{equation} \label{eq:insidemeasurementsf0fh}
    \div \big(\mathbf{K} \nabla{u}\big) = 0 \mbox{ in } D, \quad 
    u = u_0 \mbox{ on } \Gamma_0 \subset \partial D, \quad
    u = u_D \mbox{ on } \Gamma_D \subset D,
\end{equation}
where $u_0 \in C(\Gamma_0)$ and $u_D \in C(\Gamma_D)$ are prescribed, whilst the Dirichlet condition on $\Gamma_0$ is understood in the sense of \Cref{def:local solution bvp}, see relation \eqref{eq:limit regular point}.

\begin{prop}[Uniqueness of solutions to the inverse problem] \label{prop:uniqueness IP} 
Suppose that $D \subset \mathbb{R}^d$ is a bounded domain and $\mathbf{K} : D \longrightarrow \mathbb{R}^{d \times d}$ is a bounded measurable, symmetric and strictly elliptic (positive definite) matrix-valued function. In addition, if $d \geq 3$, assume that 
\begin{equation*}
\exists~L >0: \quad 
\vert K_{ij}(\mathbf{x}) - K_{ij}(\mathbf{y}) \vert \leq L \vert \mathbf{x} - \mathbf{y} \vert, 
\quad \forall~\mathbf{x}, \mathbf{y} \in D, 
\quad i, j = \overline{1,d}.
\end{equation*}
Let $\Gamma_D \subset D$ be a compact set, $\varnothing \neq I \subset D$ be an open set such that hypothesis {\rm \ref{eq:I}} is satisfied, $u_0 \in C(\Gamma_0)$ and $u_D \in C(\Gamma_D)$.\\
Then there exists at most one solution $u \in H^1_{loc}(D) \cap C(D)$ to the inverse problem \eqref{eq:insidemeasurementsf0fh}.
\end{prop}
\begin{proof}
We argue by contradiction. Let $u, \widetilde{u} \in H^1_{loc}(D) \cap C(D)$ be two solutions to the inverse problem \eqref{eq:insidemeasurementsf0fh} and consider $w \coloneqq u - \widetilde{u}$. Then $w \in H^1_{loc}(D) \cap C(D)$ solves the following inverse problem
\begin{equation*} 
\div \big(\mathbf{K} \nabla{w}\big) = 0 \mbox{ in } D, \quad 
w = 0
\mbox{ on } \Gamma_0~\mbox{(in the sense of \eqref{eq:limit regular point})}, \quad
w = 0 \mbox{ on } \Gamma_D, 
\end{equation*}
and hence $w \in H^1_{loc}(I) \cap C(I)$ is a solution to the following problem
\begin{equation*} 
    \div \big(\mathbf{K} \nabla{w}\big) = 0 \mbox{ in } I, \quad
    w = 0 
    \mbox{ on } \partial I~\mbox{(in the sense of \eqref{eq:limit regular point})}.
\end{equation*}
Hypothesis \ref{eq:I} and Theorem~\ref{thm:uniquenss local solution bvp} imply that $w \equiv 0$ in $I$. The unique continuation property for elliptic operators, see Carleman~\cite{Car39} for $d = 2$ and Garofalo and Lin~\cite[Theorem 1.2 (i)]{GaLi86} for $d \geq 3$, or Alessandrini et al.~\cite[Theorem 5.1]{Al09}, implies that $w \equiv 0$ in $D$ and this concludes the proof.

\end{proof}

Assume now that $\partial D = \Gamma_0 \cup \Gamma_1$, where $\varnothing \neq \Gamma_1 \subset \partial{D}$. 
Let $u_0 \in L^1(\Gamma_0;\mu_{\mathbf{z}})$ and $u_1 \in L^1(\Gamma_1;\mu_{\mathbf{z}})$ for a given $\mathbf{z} \in D$ (and hence for all $\mathbf{z} \in D$). 
By \Cref{prop:local solution Lq'}, it follows that
\begin{equation}
\label{eq:local solution elliptic measure}
u(\mathbf{x}) 
\coloneqq \int_{\Gamma_0} u_0(\mathbf{y}) \, \mathrm{d} \mu_\mathbf{x}(\mathbf{y}) 
+ \int_{\Gamma_1} u_1(\mathbf{y}) \, \mathrm{d} \mu_\mathbf{x}(\mathbf{y}), 
\quad \mathbf{x} \in D,
\end{equation}
is a local solution of $\div \big(\mathbf{K} \nabla{u}\big)$ in $D$. 
Let $\Gamma_D \subset D$ be a compact set and assume $u_0 \equiv 0$. Based on \Cref{thm:DeGiorgi Nash}, we define the following linear operator
\begin{equation}\label{eq:compact operator}
\mathcal{T} : L^1(\Gamma_1;\mu_{\mathbf{z}}) \longrightarrow C(\Gamma_D), \quad 
(\mathcal{T} u_1)(\mathbf{x}) 
= \int_{\Gamma_1} u_1(\mathbf{y}) \, \mathrm{d} \mu_\mathbf{x}(\mathbf{y}), 
~\forall~\mathbf{x} \in \Gamma_D, 
~\forall~u_1 \in L^1(\Gamma_1;\mu_{\mathbf{z}}).
\end{equation}
We also note that definition \eqref{eq:compact operator} is independent of $\mathbf{z} \in D$.

\begin{coro}\label{coro:compact} 
The linear operator $\mathcal{T}$ defined by \eqref{eq:compact operator} is injective and compact. 
In particular, $\mathcal{T}$ does not have a bounded inverse.
\end{coro}
\begin{proof}
Let $B \subset L^1(\Gamma_1; \mu_{\mathbf{z}})$ be bounded. Analogously to the proof of \Cref{prop:local solution Lq'}, by the Harnack inequality \cite[Corollary 8.21]{GiTr01} and Gilbarg and Trudinger \cite[Theorem 8.24]{GiTr01}, it follows that $\mathcal{T}(B)$ is a family of pointwise bounded and H\"older continuous functions with the H\"older constants uniformly bounded. Therefore, $\mathcal{T}(B)$ is equicontinuous. 
From the Arzela-Ascoli theorem, it follows that $\mathcal{T}(B)$ is relatively compact in $C(\Gamma_D)$ endowed with the supremum norm.
\end{proof}

\begin{rem}
The compactness of operator $\mathcal{T}$ is actually not surprising due to the ill-conditioning of the classical inverse Cauchy problem. However, such a result shows that even if the regularity of the input data from the natural Sobolev space $H^{1/2}(\Gamma_1)$ is relaxed to the weakest possible function space, namely the space of integrable data with respect to the elliptic measure, and the regularity of the output solution is further increased from $L^2(\Gamma_D)$ to $C(\Gamma_D)$, the solution operator is still compact and, consequently, cannot be inverted continuously. 
\end{rem}


\section{Spectral analysis}
\label{ss:convergence_spectrum}
To derive the main stochastic estimators, a systematic spectral analysis of problem \eqref{eq:insidemeasurements} is required. Hence the spectral structure of operator $\mathcal{T}$ defined by \eqref{eq:compact operator} is investigated herein, namely its restriction to a finite set of locations in the domain $D$, where the discrete interior measurements are available. 

We further use the following notations: $\mathcal{H}^{d-1}$ is the $(d-1)-$dimensional Hausdorff measure, $\sigma \coloneqq \mathcal{H}^{d-1}|_{\partial D}$ is the surface measure on $\partial D$ and $L^q(\partial D) \coloneqq L^q(\partial D, \mathrm{d}\sigma)$. Also, assume the following fundamental condition that essentially provides an appropriate $L^2-$framework for analysing the spectral properties of $\mathcal{T}$:
\begin{enumerate}[label=({\bf H.4})]
\item \label{eq:Hq} $\forall~\mathbf{x} \in D$, the elliptic measure $\mu_\mathbf{x}$ is absolutely continuous with respect to $\sigma $ and the Radon-Nikodym derivative $\rho_{\mathbf{x}} \coloneqq \dfrac{\mathrm{d} \mu_\mathbf{x}}{\mathrm{d} \sigma} \in L^2(\partial D)$. 
\end{enumerate}

\begin{rem} Condition {\rm \ref{eq:Hq}} is strongly related to the regularity of $D$ and $\mathbf{K}$, and is usually satisfied, see, e.g., 
\cite[Theorem 1]{FaJeKe84}, \cite{Dahlberg86}, \cite{Fefferman89}, \cite[Theorem 3.5~(ii)]{GrWi82}), \cite[Theorem 3.1]{JeKe81}. 
\end{rem}

Consider $\big(\mathbf{x}^D_i\big)_{i = \overline{1,M_D}} \subset D$, $M_D \geq 1$, and $\boldsymbol{\nu} \coloneqq \big( \nu_i \big)_{i = \overline{1,M_D}} \in \mathbb{R}^{M_D}$ such that 
\begin{equation}\label{eq:nu_sum}
    \nu_i \in (0,1), \quad i = \overline{1, M_D}; \qquad \sum_{i=1}^{M_D} \nu_i^2=1.
\end{equation}

\begin{rem}
Given an accessible set $\Gamma_D$ such that $\big(\mathbf{x}^D_i\big)_{i = \overline{1,M_D}} \subset \Gamma_D \subset D$, $\boldsymbol{\nu} \in \mathbb{R}^{M_D}$ induces a probability measure $m \coloneqq \displaystyle \sum \limits_{i=1}^{M_D} \nu_i^2 \delta_{\mathbf{x}^D_i}$ on $\Gamma_D$. 
If we endow $\mathbb{R}^{M_D}$ with the inner product $\langle \mathbf{x}, \mathbf{y} \rangle_{\nu} \coloneqq \sum\limits_{i=1}^{M_D}\nu_i^2x_i y_i$, $\mathbf{x}, \mathbf{y} \in \mathbb{R}^{M_D}$, then $\mathbb{R}^{M_D}$ is identified isometrically with the space $L^2(\Gamma_{D}, m)$. 
The latter space should be regarded as a discretisation of a limit space $L^2(\Gamma_D, m)$ and $m$ as an {\it importance} distribution on $\Gamma_D$. For example, $\Gamma_D$ could be a Lipschitz curve or surface in $D$, whilst $m$ could be the usual (normalised) surface measure. 
Clearly, other function spaces on $\Gamma_D$ may be considered to obtain their corresponding discretisations. 
For convenience, instead of endowing $\mathbb{R}^{M_D}$ with an inner product motivated by a certain discretisation, in the following, we consider the Euclidean inner product on $\mathbb{R}^{M_D}$ and rescale appropriately the operators with the range $\mathbb{R}^{M_D}$.
\end{rem}

Assume hypothesis \ref{eq:Hq} is satisfied and consider $M_D \geq 1$ internal points 
$\big(\mathbf{x}^D_i\big)_{i = \overline{1,M_D}} \subset D$ and the vector $\boldsymbol{\nu} \coloneqq \big( \nu_i \big)_{i = \overline{1,M_D}} \in \mathbb{R}^{M_D}$ such that relation \eqref{eq:nu_sum} holds. Define the following operator
\begin{eqnarray} \label{eq:Th}
\left.
\begin{array}{l}
T_{\boldsymbol{\nu}} : \big(L^2(\Gamma_1), \Vert \cdot \Vert_{L^2(\Gamma_1)}\big) \longrightarrow \big(\mathbb{R}^{M_D}, \Vert \cdot \Vert\big), \\[4pt]
u \longmapsto T_{\boldsymbol{\nu}} u \in L^2(\Gamma_1) 
= \big(\nu_i \mu_{\mathbf{x}^D_i}(u) \big)_{i = \overline{1,M_D}}
= \displaystyle \bigg( \nu_i \int_{\Gamma_1} u \rho_{\mathbf{x}^D_i} \, \mathrm{d}\sigma \bigg)_{i = \overline{1,M_D}} \in \mathbb{R}^{M_D}.
\end{array}
\right. 
\end{eqnarray}
Consequently, it is easy to see that the adjoint operator to $T_\nu$, denoted by $T_\nu^\ast$, is given by
\begin{eqnarray} \label{eq:Th*}
\left.
\begin{array}{l}
T_{\boldsymbol{\nu}}^\ast : \big(\mathbb{R}^{M_D}, \Vert \cdot \Vert \big) \longrightarrow \big(L^2(\Gamma_1), \Vert \cdot \Vert_{L^2(\Gamma_1)}\big), \\[4pt]
\mathbf{v} = \big( v_i \big)_{i = \overline{1,M_D}} \in \mathbb{R}^{M_D} \longmapsto T_{\boldsymbol{\nu}}^\ast \mathbf{v} 
= \displaystyle \sum_{i=1}^{M_D} \nu_i v_i \rho_{\mathbf{x}^D_i} \in L^2(\Gamma_1).
\end{array}
\right.
\end{eqnarray}
In view of relations \eqref{eq:Th} and \eqref{eq:Th*}, consider the well defined and bounded operator
\begin{equation}\label{eq:B Th*Th}
    T_{\boldsymbol{\nu}}^\ast T_{\boldsymbol{\nu}} : L^2(\Gamma_1) \longrightarrow L^2(\Gamma_1) 
\end{equation}
and introduce the following symmetric and non-negative definite matrices
\begin{align} 
\label{eq:Lambda}
&\boldsymbol{\Lambda} = \big( \Lambda_{ij} \big)_{i,j = \overline{1,M_D}} \in \mathbb{R}^{M_D\times M_D}, \qquad 
\Lambda_{ij} \coloneqq \displaystyle \int_{\Gamma_1} \rho_{\mathbf{x}_i^D} \rho_{\mathbf{x}_j^D} \, \mathrm{d}\sigma, \quad i,j = \overline{1,M_D},\\[4pt]
\label{eq:Lambda_nu}
&\boldsymbol{\Lambda}^{\boldsymbol{\nu}} \coloneqq {\rm diag}(\boldsymbol{\nu}) \; \boldsymbol{\Lambda} \; {\rm diag}(\boldsymbol{\nu}) \in \mathbb{R}^{M_D\times M_D}.
\end{align}

The results listed in \Cref{coro:B-Lambda} reveal that the SVD of operator $T_{\boldsymbol{\nu}}^\ast T_{\boldsymbol{\nu}}$ reduces to the corresponding SVD of matrix $\boldsymbol{\Lambda}^{\boldsymbol{\nu}}$ through an explicit correspondence, whilst its proof is based upon the following lemma. 

\begin{lem} \label{lem:TT*}
Let $U$ and $V$ be Hilbert spaces and $T : U \longrightarrow V$ be a linear bounded operator. Then:
\begin{enumerate}[label={\rm (}\roman*{\rm )}]
\item $T^\ast T$ and $T T^\ast$ have the same eigenvalues.
\item If $\lambda$ is an eigenvalue of $T^\ast T$ (and hence of $T T^\ast$), $U_\lambda \coloneqq \big\{ u\in U \: \big| \: T^\ast T u = \lambda u \big\}$ and $V_\lambda \coloneqq \big\{v \in V \: \big| \: T T^\ast v = \lambda v \big\}$, then $V_\lambda = T (U_\lambda)$ and 
$U_\lambda = T^\ast(V_\lambda)$.
\item If $u \in U$ and $\tilde{u} \in U_\lambda$ are such that $\langle u, \Tilde{u} \rangle_U = 0$, then $\langle T u, T \widetilde{u} \rangle_V = 0$ and a similar statement holds for the corresponding adjoint operators. 
\end{enumerate}
\end{lem}
\begin{proof}
We firstly prove (\textit{i}) and (\textit{ii}) simultaneously. Let $\lambda \in (0,\infty)$ be an eigenvalue of $T^\ast T$ and $u \in U_\lambda$. Then, clearly, $T T^\ast T u= \lambda T u$, hence $\lambda$ is an eigenvalue of $T T^\ast$ and $T(U_\lambda) \subset V_\lambda$.

Now, if $\lambda \in (0,\infty)$ is an eigenvalue of $T T^\ast$ and $v \in V_\lambda$, then $v = T u_v$, where $u_v \coloneqq T^\ast v/\lambda$ and since $T^\ast T u_v = T^\ast v = \lambda u_v$, it follows that $u_v\in U_\lambda$, as well as $v \in T^\ast(U_\lambda)$. Consequently, $V_\lambda \subset T(U_\lambda)$ and hence $V_\lambda = T (U_\lambda)$. Relation $U_\lambda = T^\ast(V_\lambda)$ follows by simply replacing $T$ with $T^\ast$.

\medskip

To prove (\textit{iii}), we simply note that \(\langle Tu, T\tilde{u} \rangle_V = \langle u,  T^\ast T \tilde{u} \rangle_U = \lambda \langle u, \tilde{u} \rangle_U = 0.\)

\end{proof}

\begin{prop}\label{coro:B-Lambda}
The operators $T_{\boldsymbol{\nu}}$ and $T_{\boldsymbol{\nu}}^\ast$, and the matrix $\boldsymbol{\Lambda}^{\boldsymbol{\nu}}$ given by \eqref{eq:Th}, \eqref{eq:Th*} and \eqref{eq:Lambda_nu}, respectively, satisfy the following relation
\begin{equation*}
T_{\boldsymbol{\nu}} T_{\boldsymbol{\nu}}^\ast = \boldsymbol{\Lambda}^{\boldsymbol{\nu}}.
\end{equation*}
In particular, the following spectral correspondence between $T_{\boldsymbol{\nu}}^\ast T_{\boldsymbol{\nu}}$ and $\boldsymbol{\Lambda}^{\boldsymbol{\nu}}$ holds:
\begin{enumerate}[label={\rm (}\roman*{\rm )}]
\item $\boldsymbol{\Lambda}^{\boldsymbol{\nu}}$ and $T_{\boldsymbol{\nu}}^\ast T_{\boldsymbol{\nu}}$ have the same eigenvalues.
\item $u \in L^2(\Gamma_1)$ is an eigenfunction of operator $T_{\boldsymbol{\nu}}^\ast T_{\boldsymbol{\nu}}$ corresponding to the eigenvalue $\lambda \neq 0$ if and only if $\mathbf{u} \in \mathbb{R}^{M_D}$ is an eigenvector of matrix $\boldsymbol{\Lambda}^{\boldsymbol{\nu}}$ corresponding to the eigenvalue $\lambda \neq 0$ and 
\begin{equation*}
u = T_{\boldsymbol{\nu}}^\ast \mathbf{u} 
= \displaystyle \sum \limits_{i=1}^{M_D} \nu_i [\mathbf{u}]_i \rho_{\mathbf{x}_i^D}.
\end{equation*}
\item Let $k \coloneqq {\rm dim}~{\rm Ker}~(\boldsymbol{\Lambda}^{\boldsymbol{\nu}})^{\perp} \leq M_D$ and $\{ \mathbf{u}_1, \mathbf{u}_2, \ldots, \mathbf{u}_k \} \subset \mathbb{R}^{M_D}$ be an orthonormal basis of ${\rm Ker}~(\boldsymbol{\Lambda}^{\boldsymbol{\nu}})^{\perp} \subset \mathbb{R}^{M_D}$ that consists of eigenvectors of $\boldsymbol{\Lambda}^{\boldsymbol{\nu}}$ with the corresponding eigenvalues $\lambda_1 \geq \lambda_2 \geq \ldots \geq \lambda_k > 0$. Define 
\begin{equation}\label{eq:basisB}
u_j 
\coloneqq \lambda_j^{-1/2} T_{\boldsymbol{\nu}}^\ast \mathbf{u}_j 
= \lambda_j^{-1/2} \displaystyle \sum \limits_{i=1}^{M_D} \nu_i [\mathbf{u}_j]_{i} \rho_{\mathbf{x}_i^D}, 
\quad j = \overline{1, k},
\end{equation}
Then $\{ u_1, u_2, \ldots, u_k \}$ is an orthonormal basis of ${\rm Ker}~(T_{\boldsymbol{\nu}}^\ast T_{\boldsymbol{\nu}})^{\perp} \subset L^2(\Gamma_1)$ that consists of eigenfunctions of the operator $T_{\boldsymbol{\nu}}^\ast T_{\boldsymbol{\nu}}$. In particular, ${\rm dim}~{\rm Ker}~(\boldsymbol{\Lambda}^{\boldsymbol{\nu}})^{\perp} = {\rm dim}~{\rm Ker}~(T_{\boldsymbol{\nu}}^\ast T_{\boldsymbol{\nu}})^{\perp}$.
\end{enumerate}
\end{prop}
\begin{proof}~\linebreak\\
\noindent ({\it i})~We apply \Cref{lem:TT*} for $T = T_{\boldsymbol{\nu}}$ given by \eqref{eq:Th}, $U = L^2(\Gamma_1)$ and $V = \mathbb{R}^{M_D}$, noting that
\begin{equation*}\label{eq:TT*}
T_{\boldsymbol{\nu}} T_{\boldsymbol{\nu}}^\ast : \mathbb{R}^{M_D} \rightarrow \mathbb{R}^{M_D}, \quad 
\mathbf{u} \in \mathbb{R}^{M_D} \longmapsto T_{\boldsymbol{\nu}} T_{\boldsymbol{\nu}}^\ast \mathbf{u} = \boldsymbol{\Lambda}^{\boldsymbol{\nu}} \mathbf{u} \in \mathbb{R}^{M_D}.
\end{equation*}
Hence the assertion follows from \Cref{lem:TT*}~({\it i}).

\medskip

\noindent ({\it ii})~According to \Cref{lem:TT*}~({\it ii}), $u \in L^2(\Gamma_1)$ is an eigenfunction of operator $T_{\boldsymbol{\nu}}^\ast T_{\boldsymbol{\nu}}$ corresponding to the eigenvalue $\lambda \neq 0$ if and only if
\begin{equation*}
\exists~\mathbf{u} \in \mathbb{R}^{M_D}: \quad 
\boldsymbol{\Lambda}^{\boldsymbol{\nu}} \mathbf{u} = \lambda \mathbf{u} 
\quad \mbox{ and } \quad 
u = T_{\boldsymbol{\nu}}^\ast \mathbf{u}.
\end{equation*}
Hence the claim follows from expression \eqref{eq:Th*} for $T_{\boldsymbol{\nu}}^\ast$.

\medskip

\noindent ({\it iii})~This assertion is a direct consequence of ({\it ii}) and \Cref{lem:TT*}~({\it iii}). 

\end{proof}

\begin{rem}\label{rem:inherit}
In particular, \Cref{coro:B-Lambda}~{\rm (}iii{\rm )} asserts that the eigenfunctions of operator $T_{\boldsymbol{\nu}}^\ast T_{\boldsymbol{\nu}}$ associated with nonzero eigenvalues inherit the regularity of the densities of the elliptic measure and they are a finite linear combination of the latter.
\end{rem}

Given $\mathbf{b} \in \mathbb{R}^{M_D}$ and $\mathbf{b}^{\boldsymbol{\nu}} \coloneqq {\rm diag}(\boldsymbol{\nu}) \; \mathbf{b}$, we further aim at finding $u \in L^2(\Gamma_1)$ that satisfies the following operator equation
\begin{equation} \label{eq:TTh}
    T_{\boldsymbol{\nu}} u = \mathbf{b}^{\boldsymbol{\nu}}
\end{equation}
by employing the SVD analysis.

\begin{rem}\label{rem:b and u0}
\begin{enumerate}[label={\rm (}\roman*{\rm )}]
\item A natural choice of the scaling vector is $\boldsymbol{\nu} \coloneqq \big( 1/\sqrt{i} \big)_{i = \overline{1, M_D}}$ which corresponds to a uniform relevance of the internal measurements taken at $\big( \mathbf{x}_i^D \big)_{i = \overline{1, M_D}}$.
However, $\boldsymbol{\nu}$ may be chosen in other ways.
In any case, it should be noted that the SVD solution of the operator equation \eqref{eq:TTh} depends heavily upon the choice of the scaling vector $\boldsymbol{\nu}$. 
    
\item In the framework considered herein with $u_0 = 0$ and according to relations \eqref{eq:local solution elliptic measure} and \eqref{eq:Th}, it follows that $\mathbf{b} = \mathbf{u}^D$, where $\mathbf{u}^D = \big( u^D_i \big)_{i=\overline{1, M_D}}$ are the prescribed internal measurements at $\big( \mathbf{x}^D_i \big)_{i=\overline{1, M_D}} \subset D$. 
If $u_0 \neq 0$ then $\mathbf{b} = \mathbf{u}^D - \left(\mu_{\mathbf{x}^D_i}(u_0)\right)_{i=\overline{1,M_D}}$, where $u_0 \in L^1(\Gamma_0,d\mu_\mathbf{z})$ for some $\mathbf{z} \in D$ (hence for any $\mathbf{z} \in D$).
\end{enumerate}
\end{rem} 

Using the notations from \Cref{coro:B-Lambda}, namely $k \coloneqq {\rm dim}~{\rm Ker}~(T_{\boldsymbol{\nu}}^\ast T_{\boldsymbol{\nu}})^\perp$ and $\lambda_1 \geq \lambda_2 \geq \ldots \geq \lambda_k > 0$ are the nonzero eigenvalues of operator $T_{\boldsymbol{\nu}}^\ast T_{\boldsymbol{\nu}}$ in decreasing order (including their multiplicities), we let $r \in \overline{1,k}$ be a prescribed truncation parameter and $P_r : L^2(\Gamma_1) \longrightarrow \mathrm{span}~\big\{v \in L^2(\Gamma_1) \, \big| \, T_{\boldsymbol{\nu}}^\ast T_{\boldsymbol{\nu}} v = \lambda_i v, \ i = \overline{1,r} \big\}$ be the projection operator onto the linear space generated by the eigenfunctions of $T_{\boldsymbol{\nu}}^\ast T_{\boldsymbol{\nu}}$ corresponding to the largest $r$ nonzero eigenvalues $\lambda_1 \geq \lambda_2 \geq \ldots \geq \lambda_r > 0$. The $r-${\it TSVD solution} to the operator equation \eqref{eq:TTh} is defined as 
\begin{equation} \label{eq:u_r}
u^{(r)} \coloneqq P_r u \in \mathrm{span}~\big\{v \in L^2(\Gamma_1) \: \big| \: T_{\boldsymbol{\nu}}^\ast T_{\boldsymbol{\nu}} v = \lambda_i v, \ i = \overline{1,r} \big\} \subset L^2(\Gamma_1),
\end{equation} 
where $u \in  L^2(\Gamma_1)$ is any solution to the following operator equation $T_{\boldsymbol{\nu}}^\ast T_{\boldsymbol{\nu}} u = T_{\boldsymbol{\nu}}^\ast \mathbf{b}^\nu$.

For any matrix $\mathbf{B} \in \mathbb{R}^{n\times m}$ such that $\mathrm{rank}(\mathbf{B}) > 0$ and any prescribed vector $\mathbf{q} \in \mathbb{R}^n$, we consider the following linear system
\begin{equation} \label{eq:LinSyst}
    \mathbf{B} \mathbf{v} = \mathbf{q},
\end{equation}
whose solution $\mathbf{v} \in \mathbb{R}^m$ is sought. We further consider the SVD of $\mathbf{B} \in \mathbb{R}^{n\times m}$, namely
\begin{eqnarray} \label{eq:SVD}
\left.
\begin{array}{ll}
\mathbf{B} = 
\mathbf{U} \mathbf{\Sigma} \mathbf{V}^{\sf T}, 
& \quad \mathbf{U} \in \mathbb{R}^{n \times n}: \ \mathbf{U}^\mathsf{T} \mathbf{U} = \mathbf{I}_n,\\
& \quad \mathbf{V} \in \mathbb{R}^{m \times m}: \ \mathbf{V}^\mathsf{T} \mathbf{V} = \mathbf{I}_m,\\
& \quad \mathbf{\Sigma} \coloneqq {\rm diag}\big(\sigma_1(\mathbf{B}), \sigma_2(\mathbf{B}), \ldots, \sigma_{\mathrm{rank}(\mathbf{B})}(\mathbf{B}) \big) \in \mathbb{R}^{n \times m},
\end{array}
\right.
\end{eqnarray}
where $\sigma_1(\mathbf{B}) \ge \sigma_2(\mathbf{B}) \ge \ldots \ge \sigma_{\mathrm{rank}(\mathbf{B})}(\mathbf{B}) > 0$ are the positive singular values of $\mathbf{B}$. 
Then, for any prescribed regularisation/truncation parameter $r \in \overline{1, \mathrm{rank}(\mathbf{B})}$, we define the corresponding $r-${\it TSVD solution} to the linear system \eqref{eq:LinSyst} by
\begin{equation} \label{eq:u1r}
    \mathbf{v}^{(r)} \coloneqq 
    \mathbf{B}^\dag_r \mathbf{q}, 
    \quad \textrm{ with } \quad 
    \mathbf{B}^\dag_r \coloneqq \mathbf{V} \mathbf{\Sigma}^{-1}_r \mathbf{U}^{\sf T}, 
\end{equation}
where $\mathbf{\Sigma}^{-1}_r \coloneqq {\rm diag}\big( \sigma_1(\mathbf{B})^{-1}, \sigma_2(\mathbf{B})^{-1}, \ldots, \sigma_r(\mathbf{B})^{-1} \big) \in \mathbb{R}^{m \times n}$.

With the above notations \eqref{eq:LinSyst}--\eqref{eq:u1r}, we consider the finite-dimensional linear system
\begin{equation} \label{eq:syst_uu_r}
    \boldsymbol{\Lambda}^\nu \mathbf{u} = \mathbf{b}^\nu 
\end{equation} 
and its corresponding $r-$TSVD solution \begin{equation}\label{eq:uu_r}
\mathbf{u}^{(r)} \coloneqq \left(\boldsymbol{\Lambda}^\nu\right)^\dag_r \mathbf{b}^\nu. 
\end{equation}
Then the following representation holds.

\begin{prop}\label{coro:representation_u(r)}
The $r-$TSVD solution of the operator equation \eqref{eq:TTh} given by \eqref{eq:u_r} and the $r-$TSVD solution of the linear system \eqref{eq:syst_uu_r} given by \eqref{eq:uu_r} satisfy the following relation
\begin{equation}
\label{eq: truncated solutions relation}
    u^{(r)} = \displaystyle \sum \limits_{j=1}^{M_D} \nu_j [\mathbf{u}^{(r)}]_j \rho_{\mathbf{x}^D_j}.
\end{equation}
\end{prop}
\begin{proof}
Let $\{ {\bf u}_1, {\bf u}_2, \ldots, {\bf u}_k \}$ be an orthonormal basis of ${\rm Ker}(\boldsymbol{\Lambda}^\nu)^\perp$ that consists of eigenvectors of $\boldsymbol{\Lambda}^{\boldsymbol{\nu}}$ with the corresponding eigenvalues $\lambda_1 \geq \lambda_2 \geq \ldots \geq \lambda_k > 0$ and $\{ u_1, u_2, \ldots, u_k \} \subset L^2(\Gamma_1)$ be the orthonormal basis of ${\rm Ker}(T_{\boldsymbol{\nu}}^\ast T_{\boldsymbol{\nu}})^{\perp} \subset L^2(\Gamma_1)$ given by \eqref{eq:basisB} that consists of eigenfunctions of the operator $T_{\boldsymbol{\nu}}^\ast T_{\boldsymbol{\nu}}$. Clearly, $u^{(r)}$ solves uniquely in ${\rm Range}~P_r \subset L^2(\Gamma_1)$ the following operator equation
\begin{equation} \label{eq:uu}
T_{\boldsymbol{\nu}}^\ast T_{\boldsymbol{\nu}} u^{(r)} = P_r T_{\boldsymbol{\nu}}^\ast \mathbf{b}^{\boldsymbol{\nu}}.
\end{equation}
If we consider that $u^{(r)} = \displaystyle \sum \limits_{i = 1}^r \langle u^{(r)}, u_i \rangle_{L^2(\Gamma_1)} u_i$, then according to \eqref{eq:Th*}, relation \eqref{eq:uu} yields
\begin{align*}
\sum \limits_{i = 1}^r \lambda_i \langle u^{(r)},  u_i \rangle_{L^2(\Gamma_1)} u_i 
&= \sum \limits_{j = 1}^{M_D} \nu_j b^\nu_j P_r\big( \rho_{\mathbf{x}_j^D} \big) 
= \sum \limits_{j = 1}^{M_D} \nu_j b^\nu_j 
\sum \limits_{i = 1}^r \langle \rho_{\mathbf{x}_j^D}, u_i \rangle_{L^2(\Gamma_1)} u_i 
= \sum \limits_{i = 1}^r 
\Bigg( \sum \limits_{j = 1}^{M_D} \nu_j b^\nu_j \langle \rho_{\mathbf{x}_j^D}, u_i \rangle_{L^2(\Gamma_1)} \Bigg) u_i.
\end{align*}
From relations \eqref{eq:basisB} and \eqref{eq:Lambda}, it follows that
\begin{align*}
\langle u^{(r)}, u_i \rangle_{L^2(\Gamma_1)} 
&= \dfrac{1}{\lambda_i} 
\sum \limits_{j = 1}^{M_D} \nu_j b^\nu_j \langle \rho_{\mathbf{x}_j^D}, u_i \rangle_{L^2(\Gamma_1)} 
= \dfrac{1}{\lambda_i} 
\sum \limits_{j = 1}^{M_D} \nu_j b^\nu_j \dfrac{1}{\sqrt{\lambda_i}} 
\sum \limits_{\ell = 1}^{M_D} \nu_\ell [\mathbf{u}_{i}]_\ell \langle \rho_{\mathbf{x}_j^D}, \rho_{\mathbf{x}_\ell^D} \rangle_{L^2(\Gamma_1)} \\[4pt]
&= \dfrac{1}{\lambda_i} 
\sum \limits_{j = 1}^{M_D} b^\nu_j \dfrac{1}{\sqrt{\lambda_i}} \left(\boldsymbol{\Lambda}^\nu \mathbf{u}_i \right)_j 
= \dfrac{1}{\lambda_i} 
\sum \limits_{j = 1}^{M_D} b^\nu_j \dfrac{\lambda_i}{\sqrt{\lambda_i}} [\mathbf{u}_{i}]_j 
= \dfrac{1}{\sqrt{\lambda_i}} \langle \mathbf{b}^\nu, \mathbf{u}_i \rangle, 
\quad i = \overline{1,r}.
\end{align*}
Consequently, in view of \eqref{eq:basisB} and \eqref{eq:uu_r}, one obtains
\begin{align*}
u^{(r)} 
&= \sum\limits_{i = 1}^r \dfrac{1}{\sqrt{\lambda_i}} \langle \mathbf{b}^\nu, \mathbf{u}_i \rangle u_i
= \sum\limits_{i = 1}^r \dfrac{1}{\lambda_i} \langle \mathbf{b}^\nu, \mathbf{u}_i \rangle
\sum\limits_{j = 1}^{M_D} \nu_j [\mathbf{u}_{i}]_j \rho_{\mathbf{x}^D_j} \\[4pt]
&= \sum\limits_{j = 1}^{M_D} \nu_j 
\Bigg( \sum \limits_{i = 1}^r \dfrac{1}{\lambda_i} \langle \mathbf{b}^\nu, \mathbf{u}_i \rangle [\mathbf{u}_{i}]_j \Bigg) \rho_{\mathbf{x}^D_j} 
= \sum\limits_{j = 1}^{M_D} \nu_j [\mathbf{u}^{(r)}]_j \rho_{\mathbf{x}^D_j}.
\end{align*}

\end{proof}

\subsection{Discretisation of the inaccessible boundary} \label{ss:discret}
The main aim herein and in \Cref{s:3} is to develop, based on \eqref{eq: truncated solutions relation}, an appropriate method for approximating the elliptic densities $\rho_{\mathbf{x}^D_j}$, $j = \overline{1,M_D}$, and hence the integrals $\displaystyle \int_{\Gamma_1} \rho_{\mathbf{x}_i^D} \rho_{\mathbf{x}_j^D} \, \mathrm{d}\sigma$, $i,j = \overline{1,M_D}$, which ultimately reduces to the approximation of $\boldsymbol{\Lambda}^{\boldsymbol{\nu}}$. More precisely, this is achieved using the following steps:
\begin{enumerate}[label={\bf Step~\arabic*.}]
\item Approximate $\boldsymbol{\Lambda}^{\boldsymbol{\nu}}$ by some $\boldsymbol{\Lambda}^{\boldsymbol{\nu}}_{\omega^1}$ arising from a deterministic discretisation of the inaccessible boundary $\Gamma_1$, see \eqref{eq:Lambda_hat_nu}.
\item Approximate $\boldsymbol{\Lambda}^{\boldsymbol{\nu}}_{\omega^1}$ by $\boldsymbol{\Lambda}^{\boldsymbol{\nu}}_{\omega^1, \varepsilon}$ via some probabilistic representations, see \eqref{eq:Lambda_hat_eps_nu}.
\item Approximate $\boldsymbol{\Lambda}^{\boldsymbol{\nu}}_{\omega^1, \varepsilon}$ by some random matrices $\boldsymbol{\Lambda}^{\boldsymbol{\nu}}_{\omega^1, \varepsilon, N}$ obtained by the Monte Carlo method, see \eqref{eq:Lambda_tilde_nu}.
\end{enumerate}
Hence in the following, we develop systematically the path for the successive approximations of $\boldsymbol{\Lambda}^{\boldsymbol{\nu}}$
\begin{equation}\label{eq:stages}
\boldsymbol{\Lambda}^{\boldsymbol{\nu}} \approx \boldsymbol{\Lambda}^{\boldsymbol{\nu}}_{\omega^1} \approx \boldsymbol{\Lambda}^{\boldsymbol{\nu}}_{\omega^1, \varepsilon} \approx \boldsymbol{\Lambda}^{\boldsymbol{\nu}}_{\omega^1,\varepsilon, N}.
\end{equation}
Note that this also provides one with a similar path for the successive approximations of the elliptic density $\rho_{\mathbf{x}_i^D}$, see \eqref{eq:rho omega1} and \eqref{eq:rho omega1 MC}, 
\begin{equation}\label{eq:stages_rho}
\rho_{\mathbf{x}_i^D} \approx \rho_{\mathbf{x}_i^D,\omega^1} \approx \rho_{\mathbf{x}_i^D,\omega^1, \varepsilon, N}
\end{equation}
as well as a similar procedure for approximating the $r-$TSVD solution of the operator equation \eqref{eq:TTh} given by \eqref{eq:u_r}, see relations \eqref{eq:u_omega_M1_r}, \eqref{eq:u_omega_M1_r_MC} and \eqref{eq:u_omega_M1_r_MC A}, in the form
\begin{equation}\label{eq:stages_u(r)}
u^{(r)} \approx 
u_{\omega^1}^{(r)} \approx 
u_{\omega^1, \varepsilon, N}^{(r)}.  
\end{equation}

The remaining of this section is devoted to {\bf Step~1}, whilst {\bf Steps~2--3} are deferred to Section~\ref{s:3}. 
To do so, we consider $M_1$ boundary points $\mathbf{x}^1 = \big( \mathbf{x}^1_i \big)_{i = \overline{1,M_1}} \subset \Gamma_1$ and some generic corresponding boundary weights $\boldsymbol{\omega}^1 \coloneqq \big(\omega^1_i\big)_{i = \overline{1,M_1}}$, namely
\begin{equation}
\begin{aligned}\label{eq:boundary_weights}
&\omega^1_i: \Gamma_1 \longrightarrow [0,1] \mbox{ is measurable}, 
\quad \omega^1_i(\mathbf{x}_j) = \delta_{ij},~i, j = \overline{1,M_1}, 
\quad \sum_{i=1}^{M_1} \omega^1_i(\mathbf{x}) = 1,~\mathbf{x} \in \Gamma_1,\\[2pt]
&\sigma(\omega^1_i) \coloneqq \int_{\Gamma_1} \omega^1_i \, \mathrm{d}\sigma > 0,~i = \overline{1,M_1}.
\end{aligned}
\end{equation}
For some typical examples of such weights, which for further purposes are in fact defined in a neighborhood of $\Gamma_1$, we refer the reader to \Cref{appendix}. 

Further, define the linear operator
\begin{equation*}
\mathbf{A}^{\boldsymbol{\nu}}_{\omega^1} : \mathbb{R}^{M_1} \longrightarrow \mathbb{R}^{M_D}, \quad 
\Bar{\mathbf{u}} \in \mathbb{R}^{M_1} \longmapsto \mathbf{A}^{\boldsymbol{\nu}}_{\omega^1} \Bar{\mathbf{u}} = T_\nu \left(\sum_{i=1}^{M_1} \Bar{u}_i \omega^1_i \right) \in \mathbb{R}^{M_D}.
\end{equation*}
More precisely, $\mathbf{A}^{\boldsymbol{\nu}}_{\omega^1}$ has the following matrix representation in the standard canonical basis
\begin{equation}\label{eq:AM matrix components}
\big[\mathbf{A}^{\boldsymbol{\nu}}_{\omega^1}\big]_{ij} = \nu_i \mu_{\mathbf{x}^D_i}(\omega^1_j) \eqqcolon \nu_i \big[\mathbf{A}_{\omega^1}\big]_{ij}, 
\quad i = \overline{1, M_D}, \quad j = \overline{1, M_1}.
\end{equation}
Next, define the approximation $\rho_{\mathbf{x}_i^D,\omega^1}$ of $\rho_{\mathbf{x}_i^D}$, $i = \overline{1, M_D}$, by
\begin{equation}\label{eq:rho omega1}
\rho_{\mathbf{x}_i^D,\omega^1}(\mathbf{x}) \coloneqq \sum \limits_{j=1}^{M_1} \sigma\big(\omega^1_j\big)^{-1} \, \mu_{\mathbf{x}_i^D}\big(\omega^1_j\big) \, \omega^1_j(\mathbf{x}), 
\quad \mathbf{x} \in \Gamma_1, 
\quad i = \overline{1, M_D}.
\end{equation}
\begin{rem}\label{rem:diam}
Define
\begin{equation}
{\rm diam}(\omega^1) 
\coloneqq \displaystyle \max_{i = \overline{1, M_1}} 
{\rm diam}\big( {\rm supp}(\omega_i^1) \big),
\end{equation}
where ${\rm diam}(V) \coloneqq \displaystyle \sup_{\mathbf{x}, \mathbf{y} \in V} 
\Vert \mathbf{x} - \mathbf{y} \Vert$. 
The approximation error between $\rho_{\mathbf{x}_i^D,\omega^1}$ and $\rho_{\mathbf{x}_i^D}$, and also those obtained by \eqref{eq:stages} and \eqref{eq:stages_rho}, and hence by \eqref{eq:stages_u(r)}, will fundamentally depend on ${\rm diam}(\omega^1)$. 
In fact, if $\rho$ is at least continuous on $\Gamma_1$, then $\sigma\big(\omega^1_j\big)^{-1} \mu_{\mathbf{x}_i^D}\big(\omega^1_j\big)$ is actually a constant approximation of $\rho$ on ${\rm supp}(\omega_j^1)$, whose error is controlled, for example, by the oscillation of $\rho$ on ${\rm supp}(\omega_j^1)$. 
Consequently, it is natural to consider the boundary weights $\mathbf{\omega}^1$ such that
\begin{equation}
\lim_{M_1 \to \infty} {\rm diam}(\omega^1) = 0.
\end{equation}
For a thorough error analysis, we refer the reader to \cite{CiGrMaII}.
\end{rem}

Note that
\begin{align*}
\int_{\Gamma_1} \rho_{\mathbf{x}_i^D,\omega^1} \rho_{\mathbf{x}_j^D,\omega^1} \, \mathrm{d}\sigma 
& = \sum\limits_{k=1}^{M_1} \sigma\big(\omega^1_k\big)^{-1} \, 
\mu_{\mathbf{x}_i^D}\big(\omega^1_k\big) \, 
\mu_{\mathbf{x}_j^D}\big(\omega^1_k\big) 
= \big\{ \mathbf{A}_{\omega^1} \; 
\big[{\rm diag}\big(\sigma(\omega^1) \big) \big]^{-1} \; 
\mathbf{A}_{\omega^1}^{\sf T} \big\}_{ij}, 
\quad i, j = \overline{1,M_D},
\end{align*}
where ${\rm diag}\big( \sigma(\omega^1) \big) \coloneqq {\rm diag}\big( \sigma(\omega^1_1), \sigma(\omega^1_2), \ldots, \sigma(\omega^1_{M_1}) \big) \in \mathbb{R}^{M_1\times M_1}$.
We further set
\begin{align}
\nonumber
&\boldsymbol{\Lambda}_{\omega^1} 
\coloneqq \mathbf{A}_{\omega^1} 
\big[ {\rm diag}\big( \sigma(\omega^1) \big) \big]^{-1} 
\mathbf{A}_{\omega^1}^{\sf T},\\[4pt]
\label{eq:Lambda_hat_nu}
&\boldsymbol{\Lambda}^{\boldsymbol{\nu}}_{\omega^1} 
\coloneqq {\rm diag}(\boldsymbol{\nu}) \; \boldsymbol{\Lambda}_{\omega^1} \; 
{\rm diag}(\boldsymbol{\nu}) 
= \mathbf{A}^\nu_{\omega^1} \; 
\big[ {\rm diag}\big( \sigma(\omega^1) \big) \big]^{-1} \; 
\big( \mathbf{A}^{\boldsymbol{\nu}}_{\omega^1} \big)^{\sf T}.
\end{align}
As a discretised version of $\mathbf{u}^{(r)}$ given by \eqref{eq:uu_r}, consider the $r-$TSVD solution $\mathbf{u}_{\omega^1}^{(r)}$ of the finite-dimensional linear system 
\begin{equation} \label{eq:syst u_r Lambda omega1}
\boldsymbol{\Lambda}^{\boldsymbol{\nu}}_{\omega_1} \mathbf{u} = \mathbf{b}^{\boldsymbol{\nu}}
\end{equation}
given by
\begin{equation} \label{eq:u_r Lambda omega1}
\mathbf{u}_{\omega^1}^{(r)} \coloneqq 
\big(\boldsymbol{\Lambda}^{\boldsymbol{\nu}}_{\omega^1}\big)^{\dag}_r \mathbf{b}^{\boldsymbol{\nu}} \in \mathbb{R}^{M_1}.
\end{equation}
For a given vector $\mathbf{b} \in \mathbb{R}^{M_D}$, we introduce the $(\omega^1, r)-${\it TSVD approximate solution} $u_{\omega^1}^{(r)}\in L^2(\Gamma_1)$ of the operator equation \eqref{eq:TTh} by
\begin{equation} \label{eq:u_omega_M1_r}
u_{\omega^1}^{(r)}(\mathbf{x}) 
\coloneqq \sum \limits_{i=1}^{M_1} \nu_i \left[\mathbf{u}_{\omega^1}^{(r)} \right]_i \rho_{\mathbf{x}_i^D,\omega^1}(\mathbf{x}), 
\quad \mathbf{x}\in \Gamma_1.
\end{equation}
Note that $u_{\omega^1}^{(r)}$ represents an approximation of the $r-$TSVD solution $u^{(r)}$ to the operator equation \eqref{eq:TTh} given by \eqref{eq:u_r}. 
Moreover, it is straightforward to check that relation \eqref{eq:u_omega_M1_r} may also be written as
\begin{equation}\label{eq:u_omega_matrix_form}
u_{\omega^1}^{(r)}(\mathbf{x}) = 
\sum \limits_{i=1}^{M_1} 
\left( \mathbf{A}^{\boldsymbol{\nu}, \dag, r}_{\omega^1} \mathbf{b}^{\boldsymbol{\nu}} \right)_i \omega_i^1(\mathbf{x}), 
\quad \mathbf{x}\in \Gamma_1,
\end{equation}
where
\begin{equation*}\label{eq: A_r}
\mathbf{A}^{\boldsymbol{\nu}, \dag, r}_{\omega^1} 
\coloneqq 
\left[ {\rm diag}\big( \sigma(\omega^1) \big) \right]^{-1/2} 
\Big( \mathbf{A}^{\boldsymbol{\nu}}_{\omega^1}  
\left[ {\rm diag}\big( \sigma(\omega^1) \big) \right]^{-1/2} \Big)^{\dag}_r 
\in \mathbb{R}^{M_1 \times M_D}.
\end{equation*}

\section{From the probabilistic representation of the elliptic measure to the Monte Carlo TSVD} 
\label{s:3}

According to Section~\ref{ss:convergence_spectrum}, the numerical approximation of the spectrum of operator $T_{\boldsymbol{\nu}}^\ast T_{\boldsymbol{\nu}}$ defined via relations \eqref{eq:Th} and \eqref{eq:Th*} reduces to the numerical approximation $\boldsymbol{\Lambda}^{\boldsymbol{\nu}}_{\omega^1}$ of matrix $\boldsymbol{\Lambda}^{\boldsymbol{\nu}}$ given by \eqref{eq:Lambda_hat_nu}, whilst the numerical approximation of the $r-$TSVD solution $u^{(r)}$ to the operator equation \eqref{eq:TTh} given by \eqref{eq:u_r} reduces to the $(\omega^1,r)-$TSVD approximate solution $u_{\omega^1}^{(r)}$ given by \eqref{eq:u_omega_M1_r} or \eqref{eq:u_omega_matrix_form}. 
Moreover, both these matrix and solution approximations $\boldsymbol{\Lambda}^{\boldsymbol{\nu}}_{\omega^1}$ and $u_{\omega^1}^{(r)}$, respectively, are constructed from the unknown matrix $\mathbf{A}_{\omega_1} = \left(\mu_{\mathbf{x}_i^D}(\omega^1_j)\right)_{\substack{i=\overline{1,M_D}\\ j=\overline{1,M_1}}} \in \mathbb{R}^{M_D \times M_1}$ defined by relation \eqref{eq:AM matrix components}. 
Hence the aim of this section is to approach the remaining {\bf Steps~2--3} and approximate numerically the entries of $\mathbf{A}_{\omega_1}$ based upon probabilistic representations and Monte Carlo sampling.

Apart from hypotheses \ref{Hyp_D}--\ref{Hyp_K}, an additional one is made, namely the anisotropic conductivity tensor $\mathbf{K}$ is assumed to be homogeneous, hence a constant matrix. 
This restriction is made for two reasons, namely ({\it i})~the presentation of the Monte Carlo estimators and the corresponding algorithm is simpler, and ({\it ii})~for homogeneous conductivity coefficients, one can directly use the existing and very fast walk-on-spheres or, in the present case, walk-on-ellipsoids numerical algorithm. 
It should be mentioned that the extension of the present case to that of nonhomogeneous Lipschitz conductivity coefficients would fundamentally follow similar steps, whilst in case of piecewise constant conductivity coefficients, we refer the reader to \cite{Ta10}. 

\subsection{Probabilistic representation}
We further recall the probabilistic representation for the solution to the Dirichlet problem
\begin{equation} \label{eq:dirichletproblem}
\div \big(\mathbf{K} \nabla u(\mathbf{x})\big) = 0, 
\quad \mathbf{x} \in D, \qquad
u(\mathbf{x}) = f(\mathbf{x}), \quad \mathbf{x} \in \partial{D}.
\end{equation}
Let $(B_t)_{t \geq 0}$ be a standard $d-$dimensional Brownian motion, starting from zero, on a filtered probability space $\big(\Omega, \mathcal{F}, (\mathcal{F}_t)_{t \geq 0}, \mathbb{P}\big)$ satisfying the usual hypotheses, and set
\begin{equation} \label{eq:Y}
Y_t^{\mathbf{x}} \coloneqq \mathbf{K}^{1/2}B_t + \mathbf{x}, \quad t \geq 0.
\end{equation}
Note that the infinitesimal generator of $\left( Y_t^\mathbf{x}, \; t \geq 0, \; \mathbf{x} \in \mathbb{R}^d \right)$ is $\dfrac{1}{2} \div \big(\mathbf{K} \nabla \big)$, see \cite[Chapter 7]{Oksendal}.
Further, consider the $\big(\mathcal{F}_t\big)-$stopping time 
\begin{equation}\label{eq:stoppingtime}
\tau_{\partial{D}}^\mathbf{x} \coloneqq 
\inf \big\{ t \geq 0 \: \big| \: Y_t^\mathbf{x} \in \partial{D} \big\}, \quad \mathbf{x} \in \mathbb{R}^d.
\end{equation}
If $f \in C(\partial D)$, then there exists a unique function $u \in C^\infty(D) \cap C(\overline{D})$ which is a solution to problem \eqref{eq:dirichletproblem} and, see, e.g., \cite[Chapter 9]{Oksendal} or \cite{ChenZhao95},
\begin{equation} \label{eq:probabilisticrepresentation}
u(\mathbf{x}) = \mathbb{E} \left[ f\left( Y^\mathbf{x}_{\tau^{\mathbf{x}}_{\partial{D}}} \right) \right], \quad \mathbf{x} \in D.
\end{equation}
For each $\mathbf{x} \in D$, the mapping 
\begin{equation*}
f \in b\mathcal{B}(\partial{D}) \longmapsto 
\mathbb{E} \left[ f \left( Y^\mathbf{x}_{\tau^{\mathbf{x}}_{\partial{D}}} \right) \right]\in \mathbb{R}
\end{equation*}
induces a probability measure on $\partial D$ which is precisely the elliptic measure $\mu_\mathbf{x}$ on $\partial D$ given by \Cref{thm:uniquenss local solution bvp}, namely
\begin{equation} \label{eq:harmonicmeasure}
\mu_\mathbf{x}(f) 
\coloneqq \int_{\partial{D}} f(\mathbf{y}) \mu_\mathbf{x}(d \mathbf{y}) 
= \mathbb{E} \left[ f \left( Y^\mathbf{x}_{\tau_{\partial{D}}} \right) \right], \quad \forall f \in b\mathcal{B}(\partial{D}).
\end{equation}
In particular, with the notations introduced in \Cref{ss:discret}, see \eqref{eq:boundary_weights} and \eqref{eq:AM matrix components}, it follows that
\begin{equation}\label{eq:mc}
\left[ \mathbf{A}_{\omega_1} \right]_{ij} = 
\mathbb{E} \left[ \omega^1_{j} \left( Y^{\mathbf{x}_i^D}_{\tau_{\partial{D}}} \right) \right], 
\quad i = \overline{1,M_D}, 
\quad j = \overline{1,M_1}.
\end{equation}
Consequently, it is natural to employ the plain Monte Carlo method to approximate the expectation in \eqref{eq:mc}.
However, to accelerate the simulation of the random variable $Y^\mathbf{x}_{\tau_{\partial{D}}}$, it is reliable to rely on the ellipsoid version of the well-known walk-on-spheres (WoS) algorithm due to \cite{Muller} and this is briefly presented in the following section.

\subsection{Walk-on-ellipsoids Markov chain} \label{ss:WoE}
Without any loss of generality, we further assume that $1$ is the largest eigenvalue of matrix $\mathbf{K}$. 
Let $\mathbf{x} \in D$ and consider the walk-on-ellipsoids (WoE) Markov chain on $\overline{D}$ constructed as
\begin{equation}\label{eq:wos chain}
X^\mathbf{x}_0=\mathbf{x} \in \overline{D}, 
\quad X^\mathbf{x}_{i+1} = X^\mathbf{x}_i 
+ {\rm dist}\big( X^{\mathbf{x}}_i, \partial D \big) \mathbf{K}^{1/2} \mathbf{U}_i, 
\quad i \geq 0,
\end{equation}
where $\big( \mathbf{U}_i \big)_{i \geq 0}$ are i.i.d. random variables on a probability space $\big( \widetilde{\Omega}, \widetilde{\mathcal{F}}, \widetilde{\mathbb{P}} \big)$, uniformly distributed on the $(d-1)-$dimensional unit sphere from $\mathbb{R}^d$ centered at the origin, denoted by $\mathbb{S}^{d-1}(\mathbf{0},1)$. 
Consider the first iteration when the process enters $\overline{D}_\varepsilon$, where $\varepsilon > 0$
\begin{equation}\label{eq:N epsilon x}
\mathrm{N}^\mathbf{x}_\varepsilon \coloneqq 
\min \big\{i \geq 0 \: \big| \: {\rm dist}\big( X^\mathbf{x}_i, \partial D \big) \leq \varepsilon \big\}.
\end{equation}

Note that the chain defined by \eqref{eq:wos chain} is equal in law to the chain obtained by considering the trace on the corresponding ellipsoids inscribed in $D$ of the continuous-time process $Y_t^\mathbf{x}$ given by \eqref{eq:Y}. 
More precisely, for $\mathbf{x} \in D$ and $\varepsilon > 0$, we consider
\begin{align}\label{eq:Y_chain}
\tau^\mathbf{x}_1 & 
\coloneqq \inf \big\{t \geq 0 \: \big| \: Y^\mathbf{x}_t \in \mathbf{x} + \mathbf{K}^{1/2} \mathbb{S}^{d-1}\big( \mathbf{0}, {\rm dist}(\mathbf{x}, \partial D)\big) \big\},\\
\tau^\mathbf{x}_n & 
\coloneqq \inf \big\{t \geq \tau^\mathbf{x}_{n-1} \: \big| \: Y^\mathbf{x}_t \in Y^\mathbf{x}_{\tau^\mathbf{x}_{n-1}} + \mathbf{K}^{1/2} \mathbb{S}^{d-1}\big(\mathbf{0}, {\rm dist}(Y^\mathbf{x}_{\tau^\mathbf{x}_{n-1}}, \partial D) \big) \big\}, \quad n \geq 2,\\ 
\widetilde{\mathrm{N}}^\mathbf{x}_\varepsilon  & \coloneqq \inf \big\{n \geq 0 \: \big| \: {\rm dist}(Y^\mathbf{x}_{\tau^\mathbf{x}_n}, \partial{D}) \leq \varepsilon \big\},\\
\tau^\mathbf{x}_\varepsilon  & 
\coloneqq \inf \big\{t \geq 0 \: \big| \: {\rm dist}(Y^\mathbf{x}_t, \partial{D}) \leq \varepsilon \big\},
\end{align}
then the chains $\big( X^\mathbf{x}_n \big)_{n \geq 0}$ and $\big( Y^\mathbf{x}_{\tau^\mathbf{x}_n} \big)_{n \geq 0}$ have the same law. 
In particular, $X^\mathbf{x}_{\mathrm{N}^\mathbf{x}_\varepsilon}$ and $Y^\mathbf{x}_{\tau^\mathbf{x}_{\widetilde{\mathrm{N}}^\mathbf{x}_\varepsilon}}$ are equal in distribution.

\subsubsection{Complexity of the walk-on-ellipsoids chain}
Clearly, the computational effort of simulating $Y^\mathbf{x}_{\tau^\mathbf{x}_{\widetilde{\mathrm{N}}^\mathbf{x}_\varepsilon}}$ crucially depends upon the number of steps $\mathrm{N}^\mathbf{x}_\varepsilon$, $\varepsilon >0$, given by \eqref{eq:N epsilon x}, required for the WoE chain to be run in order to assure that the final position is close to the boundary with the prescribed margin $\varepsilon > 0$. 
For the sake of conciseness and without any loss of generality, we mention below various known results that provide sharp estimates for this random number, for $\mathbf{K} = \mathbf{I}_d$ only, which corresponds to the standard WoS chain. 

First of all, it was shown in the seminal paper \cite{Muller} that if $D$ is convex, then $\mathbb{E}\left[\mathrm{N}^\mathbf{x}_\varepsilon\right] \in \mathcal{O}(\log(1/\varepsilon))$ and it was a folklore result that for a general bounded domain in $\mathbb{R}^d$, one has at least the estimate $\mathbb{E}\left[\mathrm{N}^\mathbf{x}_\varepsilon\right] \in \mathcal{O}(({\rm diam}(D)/\varepsilon)^2))$, see \cite{KaSh14}.
A thorough analysis of the WoS rates of convergence with respect to $\varepsilon$ has been performed in \cite{BiBr12}, revealing in particular the following:
\begin{enumerate}[label={\rm (}\roman*{\rm )}]
\item For any bounded domain $D \subset \mathbb{R}^d$, $\mathbb{E}\left[ \mathrm{N}^\mathbf{x}_\varepsilon \right] \in \mathcal{O}\left( \log^2(1/\varepsilon) \right)$ if $d = 2$ and $\mathbb{E}\left[ \mathrm{N}^\mathbf{x}_\varepsilon \right] \in \mathcal{O}\left( (1/\varepsilon)^{2-4/d} \right)$ if $d \ge 3$.
\item If $D \subset \mathbb{R}^d$ satisfies an exterior cone condition, then $\mathbb{E}\left[ \mathrm{N}^\mathbf{x}_\varepsilon \right] \in \mathcal{O}\left( \log(1/\varepsilon) \right)$.
\end{enumerate}
More explicit tail bounds were obtained in \cite{BCPZ}, essentially in terms of the principal curvatures at $\partial D$. In particular, if $D \subset \mathbb{R}^d$ is convex, then 
\begin{equation*} \label{eq:exptail}
\mathbb{P}\left( \mathrm{N}_\varepsilon^{\mathbf{x}} > M \right) 
\leq \left(1 - \dfrac{1}{4d}\right)^M \dfrac{\sqrt{{\rm dist}(\mathbf{x}, \partial D)}}{\sqrt{\varepsilon}}, 
\quad \mathbf{x} \in D.
\end{equation*}
For the fractional Laplacian, we refer the reader to \cite{KyOs18}.

To conclude, it can be easily seen that if the anisotropic conductivity matrix $\mathbf{K}$ is constant and some mild regularity for the domain $D \subset \mathbb{R}^d$ is assumed, then the WoE chain is extremely fast in reaching the $\varepsilon-$neighbourhood of $\partial D$, requiring in average $\mathcal{O}\big (\log(1/\varepsilon) \big)$ steps.

\subsection{Monte Carlo estimators}
Further, we systematically develop the following approximations:
\begin{enumerate}[label={\rm (}\roman*{\rm )}]
\item The random matrix $\mathbf{A}_{\omega^1, \varepsilon, N}$ as a Monte Carlo estimator of matrix $\mathbf{A}_{\omega^1}$.
\item The random density $\rho_{\mathbf{x}_i^D, \omega^1, \varepsilon, N}(\mathbf{x})$, $\mathbf{x} \in \Gamma_1$, as a Monte Carlo estimate of the elliptic density $\rho_{\mathbf{x}_i^D}(\mathbf{x})$, $\mathbf{x} \in \Gamma_1$.
\item The random solution  $u^{(r)}_{\omega^1, \varepsilon, N}$ as a Monte Carlo TSVD approximate solution of $u^{(r)}_{\omega^1}$ given by \eqref{eq:u_omega_M1_r}.
\end{enumerate}

\subsubsection{Monte Carlo approximation of matrix \({\bf A}_{\omega^1}\)}
Let $\big( X^\mathbf{x}_n \big)_{n\geq 0}$, $\mathbf{x} \in \overline{D}$, be the WoE constructed by relation \eqref{eq:wos chain}, $\varepsilon > 0$, and recall relation \eqref{eq:mc}. 
Also, assume that the weight functions $\boldsymbol{\omega}^1$ given by \eqref{eq:boundary_weights} are extended to $\overline{D}_\varepsilon$ by some $\boldsymbol{\omega}^{1, \varepsilon} \coloneqq \big( \omega^{1,\varepsilon}_i \big)_{i = \overline{1,M_1}}$ as
\begin{equation}
\begin{aligned}\label{eq:extended_weights_1}
&\omega_i^{1,\varepsilon} : \overline{D}_\varepsilon \longrightarrow [0,1], 
\quad \omega_i^{1,\varepsilon}|_{\Gamma_1} = \omega_i^1, \quad i=\overline{1,M_1}.
\end{aligned}
\end{equation}
Examples and general properties of such weight functions are provided in \Cref{appendix}.

As an approximation for $\mathbf{A}_{\omega^1}$ we define the matrix $\mathbf{A}_{\omega^1,\varepsilon}$ as
\begin{equation}\label{eq:WoE_matrix}
\left[ \mathbf{A}_{\omega^1, \varepsilon} \right]_{ij} \coloneqq \mathbb{E} \left[ \omega^{1, \varepsilon}_{j}\left( X^{\mathbf{x}_i^D}_{\mathrm{N}^{\mathbf{x}_i^D}_\varepsilon} \right) \right], 
\quad i = \overline{1,M_D}, 
\quad j = \overline{1,M_1},
\end{equation}
as well as its Monte Carlo estimator $\mathbf{A}_{\omega^1, \varepsilon, N}$ which is the random matrix given by
\begin{equation}\label{eq:WoE_matrix_MC}
\left[ \mathbf{A}_{\omega^1, \varepsilon, N} \right]_{ij} \coloneqq 
\dfrac{1}{N} \left[ \omega^{1, \varepsilon}_j\left( X^{\mathbf{x}_i^D, 1}_{\mathrm{N}^{\mathbf{x}_i^D, 1}_\varepsilon} \right) 
+ \ldots 
+ \omega^{1, \varepsilon}_j\left( X^{\mathbf{x}_i^D, N}_{\mathrm{N}^{\mathbf{x}_i^D, N}_\varepsilon} \right) \right], 
\quad i = \overline{1,M_D}, 
\quad j = \overline{1,M_1},
\end{equation}
where for $\mathbf{x}\in D$, $X^{\mathbf{x}, 1}_{\mathrm{N}^{\mathbf{x}, 1}_\varepsilon}, X^{\mathbf{x}, 2}_{\mathrm{N}^{\mathbf{x}, 2}_\varepsilon}, \ldots, X^{\mathbf{x}, N}_{\mathrm{N}^{\mathbf{x}, N}_\varepsilon}$ are $N \geq 1$ i.i.d. copies of $X^\mathbf{x}_{\mathrm{N}^\mathbf{x}_\varepsilon}$.

\begin{rem}
In the particular case of the extrinsic Voronoi weights \eqref{eq:voronoi weights_ext}, the entries of matrix $\mathbf{A}_{\omega^1, \varepsilon, N}$ defined by \eqref{eq:WoE_matrix_MC} are given by
\begin{equation*}\label{eq:Atilde01 Voronoi}
\begin{aligned}
\left[ \mathbf{A}_{\omega^1, \varepsilon, N} \right]_{ij} 
&= \dfrac{1}{N} \left\vert \left\{k \in \overline{1,N} \: \left| \: X^{\mathbf{x}^D_i, k}_{\mathrm{N}^{\mathbf{x}^D_i, k}_\varepsilon} \in V^1_j \right. \right\} \right\vert, 
\quad i = \overline{1,M_D}, 
\quad j = \overline{1,M_1},
\end{aligned}
\end{equation*}
where $V^1_j$, $j = \overline{1,M_1}$, are the extrinsic Voronoi cells in $\mathbb{R}^d$ corresponding to points $\mathbf{x}^1_j$, $j = \overline{1,M_1}$, given by \eqref{eq:disjoint voronoi_ext}.    
\end{rem}
Analogous to \eqref{eq:Lambda_hat_nu}, we also set
\begin{equation}
\label{eq:Lambda_hat_eps_nu}
\boldsymbol{\Lambda}_{\omega^1, \varepsilon} 
\coloneqq \mathbf{A}_{\omega^1, \varepsilon} \; 
\left[ {\rm diag}\big( \sigma(\omega^1) \big) \right]^{-1} \; 
\mathbf{A}_{\omega^1, \varepsilon}^{\sf T}, 
\quad \boldsymbol{\Lambda}^{\boldsymbol{\nu}}_{\omega^1, \varepsilon} 
\coloneqq {\rm diag}(\boldsymbol{\nu}) \; \boldsymbol{\Lambda}_{\omega^1, \varepsilon} \; 
{\rm diag}(\boldsymbol{\nu}) 
\end{equation}
and
\begin{equation}
\label{eq:Lambda_tilde_nu}
\boldsymbol{\Lambda}_{\omega^1, \varepsilon, N} 
\coloneqq \mathbf{A}_{\omega^1, \varepsilon, N} \; 
\left[ {\rm diag}\big( \sigma(\omega^1) \big) \right]^{-1} \; 
\mathbf{A}_{\omega^1, \varepsilon, N}^{\sf T}, 
\quad \boldsymbol{\Lambda}^{\boldsymbol{\nu}}_{\omega^1, \varepsilon, N} 
\coloneqq {\rm diag}(\boldsymbol{\nu}) \; 
\boldsymbol{\Lambda}_{\omega^1, \varepsilon, N} \; 
{\rm diag}(\boldsymbol{\nu}).
\end{equation}

\subsubsection{Monte Carlo solution to the inverse problem (\ref{eq:insidemeasurements})}
We further define
\begin{equation}\label{eq:rho omega1 MC}
\rho_{\mathbf{x}_i^D, \omega^1, \varepsilon, N}(\mathbf{x}) 
\coloneqq \sum \limits_{j=1}^{M_1} \sigma\big( \omega^1_j \big)^{-1} \left[ \mathbf{A}_{\omega^1, \varepsilon, N} \right]_{ij} \omega^1_j(\mathbf{x}), 
\quad \mathbf{x}\in \Gamma_1, 
\quad i = \overline{1,M_D}.
\end{equation}
For any $i = \overline{1,M_D}$, the random function $\rho_{\mathbf{x}_i^D, \omega^1, \varepsilon, N}$ is the final estimator of $\rho_{\mathbf{x}_i^D, \omega^1}$ and thus of the elliptic density $\rho_{\mathbf{x}_i^D}$.

We also consider the $r-${\it TSVD random solution} $\mathbf{u}^{(r)}_{\omega^1, \varepsilon, N}$ defined as
\begin{equation} \label{eq:u_r Lambda omega1 MC}
\mathbf{u}_{\omega^1,\varepsilon,N}^{(r)} \coloneqq 
\big(\boldsymbol{\Lambda}^\nu_{\omega^1,\varepsilon,N}\big)^{\dag}_r \mathbf{b}^\nu \in \mathbb{R}^{M_D}.
\end{equation}
corresponding to the finite-dimensional linear system 
\begin{equation} \label{eq:syst u_r Lambda omega1 MC}
\boldsymbol{\Lambda}^{\boldsymbol{\nu}}_{\omega^1, \varepsilon, N} \mathbf{u} = \mathbf{b}^\nu.
\end{equation}
Finally, we define the Monte Carlo version of \eqref{eq:u_omega_M1_r} by introducing the $(\omega^1, \varepsilon, N, r)-${\it Monte Carlo--TSVD approximate solution} $u^{(r)}_{\omega^1, \varepsilon, N}$ of the operator equation \eqref{eq:TTh} as
\begin{align}\label{eq:u_omega_M1_r_MC}
u_{\omega^1, \varepsilon, N}^{(r)}(\mathbf{x}) 
& \coloneqq \sum \limits_{i=1}^{M_D} \nu_i \left[ \mathbf{u}_{\omega^1, \varepsilon, N}^{(r)} \right]_i \rho_{\mathbf{x}_i^D, \omega^1, \varepsilon, N}(\mathbf{x})
\end{align}
which is thus a random function that approximates $u_{\omega^1}^{(r)}$ and hence $u^{(r)}$ given by \eqref{eq:u_omega_M1_r} and \eqref{eq:u_r}, respectively.
Note that relation \eqref{eq:u_omega_M1_r_MC} may also be recast as
\begin{equation}\label{eq:u_omega_M1_r_MC A}
u^{(r)}_{\omega^1, \varepsilon, N}(\mathbf{x}) 
\coloneqq \sum \limits_{i=1}^{M_1} \left[ \mathbf{A}^{\boldsymbol{\nu}, \dag, r}_{\omega^1, \varepsilon, N} \mathbf{b}^{\boldsymbol{\nu}} \right]_i \omega_i^1(\mathbf{x}), 
\quad \mathbf{x} \in \Gamma_1,
\end{equation}
where
\begin{equation}\label{eq: A_r_MC}
\mathbf{A}^{\nu, \dag, r}_{\omega^1,\varepsilon, N} 
\coloneqq \left[ {\rm diag}\big( \sigma(\omega^1) \big) \right]^{-1/2} 
\Big( {\rm diag}(\boldsymbol{\nu}) \; 
\mathbf{A}_{\omega^1, \varepsilon, N} \; 
\left[ {\rm diag}\big( \sigma(\omega^1) \big) \right]^{-1/2} \Big)^{\dag}_r 
\in \mathbb{R}^{M_1 \times M_D}.
\end{equation}

To conclude, the following estimators have been derived:
\begin{enumerate}[label={\rm (}\roman*{\rm )}]
\item The random matrix $\boldsymbol{\Lambda}^{\boldsymbol{\nu}}_{\omega^1, \varepsilon, N}$ that approximates operator $T_{\boldsymbol{\nu}} T^\ast_{\boldsymbol{\nu}}$ and, consequently, the spectrum of operator $T^\ast_{\boldsymbol{\nu}} T_{\boldsymbol{\nu}}$ can be approximated via \Cref{coro:B-Lambda}. 
\item The random functions $\rho_{\mathbf{x}_i^D, \omega^1, \varepsilon, N}$, $i = \overline{1,M_D}$, that estimate the elliptic densities $\rho_{\mathbf{x}_i^D}$, $i = \overline{1,M_D}$.
\item The random functions $u_{\omega^1, \varepsilon, N}^{(r)}$ that approximate the $r-$TSVD solution $u^{(r)}$ given by \eqref{eq:u_r}.
\end{enumerate}

\begin{rem}
Note that under suitable, yet mild regularity conditions on the domain $D \subset \mathbb{R}^d$, we prove rigorously the strong consistency of the above biased estimators towards their corresponding estimated quantities with explicit control of the errors, see \cite{CiGrMaII}. 
\end{rem}

\section{Numerical simulations and discussion} \label{section:simulations}
The present section is devoted to the description of the algorithm that results as direct consequence Sections~\ref{ss:convergence_spectrum} and \ref{s:3}, see \Cref{ss:algorithm}, and testing its accuracy, convergence, stability and efficiency when solving inverse Cauchy problems for isotropic and anisotropic materials in two- and three-dimensional multiply-connected domains with smooth and piecewise smooth boundaries, see Sections~\ref{ss:2Dexamples} and \ref{ss:3Dexamples}, respectively. 
\subsection{Description of the algorithm} \label{ss:algorithm}
Consider $D \subset \mathbb{R}^d$, $d \geq 2$, a bounded domain occupied by an (an)isotropic material characterised by a constant heat conductivity matrix $\mathbf{K} \in \mathbb{R}^{d \times d}$ such that hypotheses \ref{Hyp_D}--\ref{eq:Hq} are fulfilled. Without any loss of generality, we assume $\mathbf{u}^0 =\big( u^0_i \big)_{i = \overline{1,M_0}} \equiv \mathbf{0} \in \mathbb{R}^{M_0}$ for the measurements taken at the boundary points $\mathbf{x}^0 = \big( \mathbf{x}^0_i \big)_{i = \overline{1,M_0}} \subset \Gamma_0$, see also \Cref{rem:b and u0 simulations}, whilst the measurements $\mathbf{u}^D = \big( u^D_i \big)_{i = \overline{1,M_D}}$ made at the internal points $\mathbf{x}^D = \big( \mathbf{x}^D_i \big)_{i \overline{1,M_D}} \subset D$ and their corresponding weights $\big( \nu_i \big)_{i = \overline{1,M_D}}$ satisfy relation \eqref{eq:nu_sum}. 
\begin{algorithm}[H]
\caption{Main}\label{alg:main}
\begin{algorithmic}
\Require $\big( \mathbf{x}^D_i \big)_{i=\overline{1,M_D}}$,
$\big( \nu_i \big)_{i=\overline{1,M_D}}$,
$\big( u^D_i \big)_{i=\overline{1,M_D}}$, $\mathbf{K}$, ${\rm dist}(\cdot, \partial D)$ \Comment{Given data} \\ 
\hspace{38pt} 
$\big( \omega^1_i(\cdot) \big)_{i=\overline{1,M_1}}$, $\varepsilon$, $N$; \Comment{Approximation parameters}
\Ensure $u^{(r)}_{\omega^1,\varepsilon,N}$

\State $\mathbf{A}_{\omega^1,\varepsilon,N} \gets {\bf MC-REM}\Big( (\mathbf{x}^D_i \big)_{i=\overline{1,M_D}}, \mathbf{K}, {\rm dist}(\cdot, \partial D); \big( \omega^1_i(\cdot) \big)_{i=\overline{1,M_1}}, \varepsilon, N \Big)$ \Comment{Compute MC--REM}

\State $u^{(r)}_{\omega^1, \varepsilon, N}(\mathbf{x}) 
\gets \sum \limits_{i=1}^{M_1} \left[\mathbf{A}^{\boldsymbol{\nu}, \dag, r}_{\omega^1, \varepsilon, N}  {\rm diag}(\boldsymbol{\nu}) \mathbf{u}^D \right]_i \omega_i^1(\mathbf{x})$ \Comment{Compute TSVD solution}
\end{algorithmic}
\end{algorithm}

\begin{algorithm}[H]
\caption{Monte Carlo -- redistributed elliptic measure (MC--REM)}\label{alg:MC-REM}
\begin{algorithmic}
\Require $\big( \mathbf{x}^D_i \big)_{i=\overline{1,M_D}}$, $\mathbf{K}$, ${\rm dist}(\cdot, \partial D)$ \Comment{Given data} \\ 
\hspace{38pt} 
$\big( \omega^1_i(\cdot) \big)_{i=\overline{1,M_1}}$, $\varepsilon$, $N$ \Comment{Approximation parameters}
\Ensure $\mathbf{A}_{\omega^1,\varepsilon,N}$;

\State $\mathbf{A}_{\omega^1, \varepsilon, N} \gets \mathbf{0} \in \mathbb{R}^{M_D \times M_1}$ \Comment{Initialise MC--REM weights matrix}

\For{$k=\overline{1,N}$} \Comment{Generate WoS trajectory for all points $\big( \mathbf{x}^D_i \big)_{i=\overline{1,M_D}}$}

\State $\mathbf{v} \gets \{1, 2, \ldots, M_D\}$ \Comment{Initialise index vector}

\State $X^{\rm WoS}_{\mathbf{v}} \gets \mathbf{x}^D_{\mathbf{v}} \in \mathbb{R}^{d \times M_D}$ \Comment{Initialise coordinates matrix}

\While{$\mathbf{v} \neq [ \, ]$}

\State $\mathbf{r}_{\mathbf{v}} \gets {\rm dist}\big( X^{\rm WoS}_{\mathbf{v}}, \partial D \big)$ \Comment{Compute distances to boundary}

\State $ \mathbf{v} \gets \big\{i \in \mathbf{v} \: \big| \: \ r_i > \varepsilon \big\}$ \Comment{Update index vector}

\State $\mathbf{U} \sim U\big( \mathbb{S}^{d-1}(\mathbf{0},1) \big)$ \Comment{Sample uniformly distributed points on unit sphere}

\State $X^{\rm WoS}_{\mathbf{v}} \gets X^{\rm WoS}_{\mathbf{v}} + \mathbf{K}^{1/2} \mathbf{r}_{\mathbf{v}} \mathbf{U} $\Comment{Update WoS coordinates matrix}

\EndWhile

\For{$i = \overline{1,M_D}$} \Comment{Distribute point weights}

\If{${\rm dist}\big( X^{\rm WoS}_i, \Gamma_1 \big) \leq \varepsilon$}

\State $\left[ \mathbf{A}_{\omega^1, \varepsilon,N} \right]_{ij} \gets \left[ \mathbf{A}_{\omega^1, \varepsilon, N} \right]_{ij} + \frac{1}{N}\omega^{1, \varepsilon}_j(X^{WoS}_i), \ j = \overline{1,M_1}$

\EndIf
\EndFor
\EndFor
\end{algorithmic}
\end{algorithm}
Let ${\rm dist}(\cdot, \partial D)$ be the Euclidean distance function to the boundary $\partial D$, see \Cref{rem:wos algorithm}(ii). We further select some weight functions $\big( \omega^1_i(\cdot) \big)_{i = \overline{1,M_1}}$ associated with points $\big( \mathbf{x}^1_i \big)_{i = \overline{1,M_1}} \subset \Gamma_1$, see relation \eqref{eq:boundary_weights}, which are extended to $\overline{D}_\varepsilon$ by $\big( \omega^{1,\varepsilon}_i \big)_{i = \overline{1,M_1}}$ according to relation \eqref{eq:extended_weights_1}, for a fixed diameter $\varepsilon > 0$ of the $\varepsilon-$shell. Finally, we set the number of Monte Carlo samples as $N \in \mathbb{N}$, i.e. the number of WoS/WoE trajectories of \eqref{eq:wos chain}. The regularised numerical solution $u^{(r)}_{\omega^1,\varepsilon,N}$ given by relation \eqref{eq:u_omega_M1_r_MC A} is determined using the main algorithm (see \Cref{alg:main}) and matrix ${\mathbf{A}}_{\omega^1, \varepsilon, N}$ that is defined by \eqref{eq:WoE_matrix_MC} and computed according to the Monte Carlo -- redistributed elliptic measure (MC--REM) algorithm (see \Cref{alg:MC-REM}).

\begin{rem}\label{rem:wos algorithm} 
\begin{enumerate}[label={\rm (}\roman*{\rm )}]
\item The while loop in \Cref{alg:MC-REM} ends in a finite number of steps a.s., see \Cref{ss:WoE}. Moreover, the while loop with a random number of steps can be replaced with a loop with a fixed deterministic number of iterations \cite{BCPZ}.

\item In many practical situations, approximations of the distance function ${\rm dist}(\cdot, \partial D)$ are actually used which may easily be computed for numerous domains \cite[Section 5.1]{SaCr20}. One may use a Lipschitz replacement that bounds the distance function from below in $D$ and from above (modulo a constant) in $D \setminus D_{\varepsilon}$, see \cite{BCPZ}, respectively.
\end{enumerate}
\end{rem}

\begin{rem}\label{rem:b and u0 simulations}
If $\mathbf{u}^0 = \big( u^0_i \big)_{i=\overline{1,M_0}} \neq \mathbf{0}$, then according to \Cref{rem:b and u0}, one computes the TSVD solution as
\begin{equation}\label{eq:tsvd solution simulations}
u^{(r)}_{\omega^1, \varepsilon, N}(\mathbf{x}) 
\coloneqq \sum \limits_{i=1}^{M_1} \left[\mathbf{A}^{\boldsymbol{\nu}, \dag, r}_{\omega^1,\varepsilon, N} \; {\rm diag}(\boldsymbol{\nu}) \left( \mathbf{u}^D - \mathbf{A}_{\omega^0, \varepsilon, N} \mathbf{u}^0 \right) \right]_i \omega_i^1(\mathbf{x}),
\end{equation}
where $\mathbf{A}_{\omega^0, \varepsilon, N}$ is a random matrix similar to $\mathbf{A}_{\omega^1, \varepsilon, N}$ defined by \eqref{eq:WoE_matrix_MC}, for some weights $\big( \omega_i^{0, \varepsilon} \big)_{i=\overline{1,M_0}}$ similar to $\big( \omega_i^{1, \varepsilon} \big)_{i=\overline{1,M_1}}$ given by relation \eqref{eq:extended_weights_1}, however with $\Gamma_0$ instead of $\Gamma_1$.
\end{rem}

\subsection{Two-dimensional examples} 
\label{ss:2Dexamples}
For the two-dimensional examples investigated herein, the measurements taken at $M_0$ boundary points $\big( \mathbf{x}_i^0 \big)_{i= \overline{1,M_0}} \subset \Gamma_0$ and $M_D$ internal points $\big( \mathbf{x}_i^D \big)_{i= \overline{1,M_D}} \subset D$ are denoted by $u^0_i \coloneqq u^{\rm (ex)}\big( \mathbf{x}_i^0 \big)$, $i = \overline{1, M_0}$, and $u^D_i \coloneqq u^{\rm (ex)}\big( \mathbf{x}_i^D \big)$, $i = \overline{1, M_D}$, respectively, where $u^{\rm (ex)}$ is the corresponding exact solution. In addition, $M_1$ boundary points $\big( \mathbf{x}^1_i \big)_{i = \overline{1, M_1}} \subset \Gamma_1$ are considered, while $\nu_i = \dfrac{1}{\sqrt{M_D}}$, ${i=\overline{1,M_D}}$, and the associated weight functions $\big( \omega^1_i(\cdot) \big)_{i = \overline{1,M_1}}$ are taken to be the corresponding extrinsic Voronoi weights, see \Cref{ex:A.2}.

Furthermore, the performance of the proposed algorithm described in \Cref{ss:algorithm} is tested for various homogeneous (an)isotropic materials, i.e. characterised by a constant conductivity matrix $\mathbf{K} \in \mathbb{R}^{d \times d}$, $d = 2, 3$, occupying a bounded domain $D \subset \mathbb{R}^d$, $d = 2, 3$, by analysing the following:
\begin{enumerate}[label={\rm (}\roman*{\rm )}]
\item The Monte-Carlo estimator $\rho_{\mathbf{x}_i^D, \omega^1, \varepsilon, N}$ defined by \eqref{eq:rho omega1 MC} of the density of the associated elliptic measure $\mu_{\mathbf{x}^D_i}$, $\mathbf{x}^D_i \in D$, given by \eqref{eq:harmonicmeasure}, on the inaccessible boundary $\Gamma_1 \subset \partial D$.

\item The eigenvalues $\widetilde{\lambda}_1 \geq \widetilde{\lambda}_2 \geq \ldots \geq \widetilde{\lambda}_{M_D}$ of the Monte-Carlo estimator $\boldsymbol{\Lambda}^{\boldsymbol{\nu}}_{\omega^1, \varepsilon, N}$, see \eqref{eq:Lambda_tilde_nu}, of operator $T_{\boldsymbol{\nu}}^\ast T_{\boldsymbol{\nu}}$ defined by \eqref{eq:B Th*Th}, as well as the corresponding Monte-Carlo estimators of eigenvectors of  $T_{\boldsymbol{\nu}}^\ast T_{\boldsymbol{\nu}}$
\begin{equation}\label{eq:eigenfunctions_MC}
\widetilde{u}_j 
\coloneqq \widetilde{\lambda}_j^{-1/2} 
\displaystyle \sum \limits_{i=1}^{M_D} 
\nu_i [\widetilde{\mathbf{u}}_j]_{i} \rho_{\mathbf{x}_i^D, \omega^1, \varepsilon, N}, 
\quad j = \overline{1, k},
\end{equation}
where $\big( \widetilde{\mathbf{u}}_j \big)_{j=\overline{1,M_D}}$ is an orthonormal eigenbasis of $\left({\rm Ker}~\big( \boldsymbol{\Lambda}^{\boldsymbol{\nu}}_{\omega^1, \varepsilon, N} \big)\right)^{\perp}$.

\item The Monte-Carlo estimator $u_{\omega^1, \varepsilon, N}^{(r)}$ defined by \eqref{eq:tsvd solution simulations}, see also relations \eqref{eq:u_omega_M1_r_MC}--\eqref{eq:u_omega_M1_r_MC A}, at $M_1$ boundary points $\big( \mathbf{x}^1_i \big)_{i = \overline{1, M_1}} \subset \Gamma_1$, for various levels of the truncation $r$.

\item The pointwise approximation error $\big\vert u^{\rm (ex)}\big( \mathbf{x}^D_i \big) - \widehat{u}^{D,r}_i \big\vert$ at $M_D$ internal points $\mathbf{x}^D_i \in D$, $i = \overline{1, M_D}$, where 
$\widehat{u}^{D,r}_i 
\coloneqq \left( \mathbf{A}_{\omega^1, \varepsilon, N} \mathbf{u}_{\omega^1, \varepsilon, N}^{(r)} 
+ \mathbf{A}_{\omega^0, \varepsilon, N} \mathbf{u}^0 \right)_i$, $i = \overline{1, M_D}$, 
with $\mathbf{A}_{\omega^1, \varepsilon, N}$ given by \eqref{eq:WoE_matrix_MC}, $\mathbf{u}_{\omega^1, \varepsilon, N}^{(r)} \coloneqq \left( u_{\omega^1, \varepsilon, N}^{(r)}\big( \mathbf{x}^1_j \big) \right)_{j=\overline{1,M_1}}$, $\mathbf{A}_{\omega^0, \varepsilon, N}$ defined according to \Cref{rem:b and u0 simulations}, and $\mathbf{u}^0 \coloneqq \big(u_i^0\big)_{i = \overline{1,M_0}}$, for various truncation levels $r$.
\end{enumerate}

\begin{ex} \label{ex:ex01}
Consider $D = D_0 \setminus \overline{D}_1 \subset \mathbb{R}^2$, where $D_0 = (-1,1)^2$, $D_1 = \big\{ \mathbf{x} \in \mathbb{R}^2 \: \big| \: \Vert \mathbf{x} - \mathbf{x}^{(1)} \Vert < R_1 \big\}$, $\mathbf{x}^{(1)} = (0.5,0)$ and $R_1 = 0.2$, while $\Gamma_\ell = \partial D_\ell$, $\ell = 0, 1$, and $\Gamma_D = \big\{\mathbf{x} \in D \: \big| \: {\rm dist}(\mathbf{x}, \Gamma_0) = 0.05 \big\}$, see \Cref{Fig1}. Here, we set $M_0 = 400$, $M_1 = 50$, $M_D = 40$, $N = 10^6$ Monte Carlo samples, and $\varepsilon = 10^{-7}$. 
\end{ex}

\begin{ex} \label{ex:ex02}
Let $D = D_0 \setminus \overline{D}_1 \subset \mathbb{R}^2$, where $D_\ell = \big\{ \mathbf{x} \in \mathbb{R}^2 \: \big| \: \Vert \mathbf{x} - \mathbf{x}^{(\ell)} \Vert < R_\ell \big\}$ and $\mathbf{x}^{(\ell)} = (0,0)$, $\ell = 0, 1$, $R_0 = 1$ and $R_1 = 0.5$, while $\Gamma_\ell = \partial D_\ell$, $\ell = 0, 1$, and $\Gamma_D = \big\{ \mathbf{x} \in D \: \big| \: {\rm dist}(\mathbf{x}, \Gamma_0) = 0.05 \big\}$, see \Cref{Fig6a}. Here, $M_0 = 500$, $M_1 = 100$, $M_D = 100$, $N = 10^6$ Monte Carlo samples, and $\varepsilon = 10^{-10}$. 
\end{ex}

\begin{ex} \label{ex:ex03}
We let $D = D_0 \setminus \overline{D}_1 \subset \mathbb{R}^2$, where $D_\ell = \big\{ \mathbf{x} \in \mathbb{R}^2 \: \big| \: \Vert \mathbf{x} - \mathbf{x}^{(\ell)} \Vert < R_\ell \big\}$ and $\mathbf{x}^{(\ell)} = (0,0)$, $\ell = 0, 1$, $R_0 = 1$ and $R_1 = 0.2$, while $\Gamma_\ell = \partial D_\ell$, $\ell = 0, 1$, and $\Gamma_D = \big\{\mathbf{x} \in D \: \big| \: {\rm dist}(\mathbf{x}, \Gamma_0) = 0.05 \big\}$, see \Cref{Fig6b}. Here, we set $M_0 = 500$, $M_1 = 100$, $M_D = 100$, $N = 10^6$ Monte Carlo samples, and $\varepsilon = 10^{-10}$.
\end{ex}

\begin{ex} \label{ex:ex04}
Consider $D = D_0 \setminus \displaystyle \bigcup_{\ell=1}^5 \overline{D}_\ell \subset \mathbb{R}^2$, where $D_\ell = \big\{ \mathbf{x} \in \mathbb{R}^2 \: \big| \: \Vert \mathbf{x} - \mathbf{x}^{(\ell)} \Vert < R_\ell \big\}$, $\ell = \overline{0, 5}$, $\mathbf{x}^{(0)} = (0,0)$, $R_0 = 1$, $\mathbf{x}^{(\ell)} = 0.5 \: (\cos{\theta_\ell}, \sin{\theta_\ell})$, $\theta_\ell = 2 (\ell - 1) \pi \big/ 5$ and $R_\ell = 0.2$, $\ell = \overline{1, 5}$, while $\Gamma_0 = \partial D_0$, $\Gamma_1 = \displaystyle \bigcup_{\ell=1}^5 \partial D_\ell$, and $\Gamma_D = \big\{ \mathbf{x} \in \mathbb{R}^2 \: \big| \: {\rm dist}(\mathbf{x}, \Gamma_0) = 0.05 \big\}$, see \Cref{Fig6c}. Here, $M_0 = 500$, $M_1 = 500$ {\rm (}$M_1/5$ for each $\partial D_\ell$, $\ell = \overline{1,5}${\rm )}, $M_D = 100$, $N = 10^6$ Monte Carlo samples, and $\varepsilon = 10^{-10}$. 
\end{ex}

\begin{ex} \label{ex:ex05}
Let $D = D_0 \setminus \overline{D}_1 \subset \mathbb{R}^2$, where $D_0 = (-1,1)^2$, $D_1 = \big\{ \mathbf{x} \in \mathbb{R}^2 \: \big| \: \Vert \mathbf{x} - \mathbf{x}^{(1)} \Vert < R_1 \big\}$, $\mathbf{x}^{(1)} = (0.5,0)$ and $R_1 = 0.2$, while $\Gamma_0 = \varnothing$, $\Gamma_1 = \partial D_0 \cup \partial D_1$, and $\Gamma_D = \big\{ \mathbf{x} \in \mathbb{R}^2 \: \big| \: \Vert \mathbf{x} - \mathbf{x}^{(D)} \Vert = R_D \big\}$ with $\mathbf{x}^{(D)} = (-0.5, 0.5)$ and $R_D = 0.3$, see \Cref{Fig6d}. Here, we set $M_1 \coloneqq M_{\partial D_0} + M_{\partial D_1} = 450$ {\rm (}$M_{\partial D_0} = 400$ and $M_{\partial D_1} = 50${\rm )}, $M_D = 40$, $N = 10^6$ Monte Carlo samples, and $\varepsilon = 10^{-10}$.
\end{ex}

For Examples~\ref{ex:ex01}--\ref{ex:ex05}, we also consider three exact solutions in the isotropic case $\mathbf{K} = \mathbf{I}_2$, namely
\begin{subequations} \label{eq:2D-sol}
\begin{align}
\label{eq:2D-sol1}
&u^{\rm (ex)}(x_1,x_2) = \dfrac{3}{4} (1 - x_1^2) + \dfrac{1}{4} x_2^3, 
\quad \mathbf{x} = (x_1,x_2) \in \overline{D}\\
\label{eq:2D-sol2}
&u^{\rm (ex)}(x_1,x_2) = -\dfrac{1}{3} x_1^3 - x_1^2 x_2 + x_1 x_2^2 + \dfrac{1}{3} x_2^3, 
\quad \mathbf{x} = (x_1,x_2) \in \overline{D}\\
\label{eq:2D-sol3}
&u^{\rm (ex)}(x_1,x_2) = x_1^2 - x_2^2, 
\quad \mathbf{x} = (x_1,x_2) \in \overline{D}
\end{align}
as well as an additional exact solution in the anisotropic case $\mathbf{K} = \big( K_{ij} \big)_{i, j = 1, 2}$ with $K_{11} = 1$, $K_{12} = K_{21} = 0.3$ and $K_{22} = 0.4$, given by 
\begin{align}
\label{eq:2D-sol4}
&u^{\rm (ex)}(x_1,x_2) = \dfrac{2 K_{21} - K_{22}}{3 K_{11}} x_1^3 - x_1^2 x_2 + x_1 x_2^2 - \dfrac{2 K_{12} - K_{11}}{3 K_{22}} x_2^3, 
\quad \mathbf{x} = (x_1,x_2) \in \overline{D}.
\end{align}
\end{subequations}

Moreover, for the aforementioned Examples~\ref{ex:ex01}--\ref{ex:ex05}, the boundary and internal points are uniformly distributed on the boundaries $\Gamma_\ell$, $\ell = 0, 1$, and the closed curve $\Gamma_D \subset D$, respectively, and ordered in a counterclockwise sense.

\Cref{Fig01} presents the independent samples drawn from the harmonic density estimators for \Cref{ex:ex01}. 
More precisely, \Cref{Fig1a} displays the color map on the inaccessible boundary $\Gamma_1$, obtained from one sample drawn from the averaged estimators of the harmonic densities with poles at $\big( \mathbf{x}_i^D \big)_{i= \overline{1,M_D}}$, namely the random vector $\left(\dfrac{1}{M_D} \displaystyle \sum_{i=1}^{M_D} \rho_{\mathbf{x}_i^D, \omega^1, \varepsilon, N}\big( \mathbf{x}_j^1 \big)\right)_{j= \overline{1,M_1}}$, for \Cref{ex:ex01}, and this shows the influence, on average, of the location of the internal measurements on the inaccessible boundary with a decreasing impact from the red to the blue regions on $\Gamma_1$, respectively.
The numerical representation of the color map depicted in \Cref{Fig1a} is showed in \Cref{Fig1b}. 
Ten independent samples drawn from the aforementioned random vector are illustrated in \Cref{Fig1c}, as well as their corresponding empirical mean which shows that the Monte Carlo estimator has a small variance.
\Cref{Fig1d} presents ten independent samples drawn from the harmonic density estimator $\rho_{\mathbf{x}_i^D, \omega^1, \varepsilon, N}\big( \mathbf{x}_j^1 \big)$, $j \in \big\{ 0, 4, 8, \ldots, 36 \big\}$, and their empirical mean which is also small and hence confirms that the estimator itself has a small variance.

The reconstruction of the eigenvalues of the direct operator $T^{\ast}_{\boldsymbol{\nu}} T_{\boldsymbol{\nu}}$ corresponding to the geometry considered in \Cref{ex:ex01} is presented in \Cref{Fig02}. 
\Cref{Fig2a,Fig2b} display on a linear and a semi-logarithmic scales, respectively, ten independent samples drawn from $\widetilde{\lambda}^{\boldsymbol{\nu}}_i$, $i = \overline{1,M_D}$, which represent the estimators for the eigenvalues of the matrix operator $\Lambda_{\omega^1, \varepsilon, N}^{\boldsymbol{\nu}}$ and hence the estimators for the eigenvalues of the direct operator $T^{\ast}_{\boldsymbol{\nu}} T_{\boldsymbol{\nu}}$, as well as their empirical mean, whilst the gaps of these sampled eigenvalues are presented, on a semi-logarithmic scale, in \Cref{Fig2c}. 
\Cref{Fig4b} illustrates ten independent samples of the eigenvector estimators $\widetilde{u}_\ell$, $\ell = \overline{1,15}$, given by \eqref{eq:eigenfunctions_MC}, for the corresponding first fifteen eigenvectors of $T^{\ast}_{\boldsymbol{\nu}} T_{\boldsymbol{\nu}}$, retrieved at $\big( \mathbf{x}_j^1 \big)_{j = \overline{1,M_1}} \subset \Gamma_1$. 
The following conclusions can be drawn from the results presented in \Cref{Fig02}. Firstly, the estimators for the eigenvalues of the direct operator $T^{\ast}_{\boldsymbol{\nu}} T_{\boldsymbol{\nu}}$ decay drastically, hence the severe instability of the inverse problem. Secondly, the estimators of the first ten eigenvectors have small variances, whilst the remaining ones seem to be poorly recovered and with a high variance. This indicates that without any further {\it a priori} information on the unknown solution, one can hope for a reliable reconstruction only in the subspace spanned by the first ten eigenvectors sampled in \Cref{Fig4b}. 
Consequently, if the unknown true solution can be well represented by the above first ten eigenvectors, then its reconstruction on $\Gamma_1$ from the measurements provided on $\Gamma_D$ and $\Gamma_0$ should be feasible and stable.

\Cref{Fig5a} displays ten independent samples drawn from the Monte Carlo estimator $u^{(r)}_{\omega^1, \varepsilon, N}$ of the corresponding SVD solution truncated at the $r-$th estimated eigenvalue, $r = \overline{1,15}$, at $\big( \mathbf{x}_j^1 \big)_{j = \overline{1, M_1}} \subset \Gamma_1$, obtained for \Cref{ex:ex01}, their corresponding empirical mean and the exact solution \eqref{eq:2D-sol1}.
The results showed in \Cref{Fig5a} can easily be correlated with those illustrated in \Cref{Fig4b}, in the sense that the empirical variance of the estimated TSVD solution increases significantly, provided that truncation is performed below the same first ten eigenvalues considered to be reliable according to \Cref{Fig4b}. 
For the exact solution $u^{\rm (ex)}$ given by \eqref{eq:2D-sol1}, we set $f \coloneqq u^{(r)}_{\omega^1, \varepsilon, N}$ on $\Gamma_1$, $r = \overline{1,15}$, and $f \coloneqq u^{\rm (ex)}\big|_{\Gamma_0}$ on $\Gamma_0$, and compute the approximations $\widehat{u}_i^{D,r}$, $r = \overline{1,15}$, of $\mu_{\mathbf{x}_i^D}(f)$, $i = \overline{1,M_D}$, for comparison purposes with the corresponding prescribed internal measurements $u_i^D$, $i = \overline{1,M_D}$. The absolute errors $\big\vert u_i^D - \widehat{u}_i^{D,r} \big\vert$, $i = \overline{1, M_D}$, $r = \overline{1, 15}$, retrieved for \Cref{ex:ex01} and the exact solution \eqref{eq:2D-sol1}, are presented in \Cref{Fig5d}. It can be seen from this figure that each of the fifteen TSVD solutions fits the inner measurements within a uniform error of order $O(10^{-3})$. 
Therefore, deciding which of the fifteen possible approximate solutions should be considered as an approximation of the exact unknown Dirichlet data should also depend on some {\it a priori} knowledge of the solution, provided that the latter exists. 
Nonetheless, the information provided by \Cref{Fig1a} can always guide the selection since one can observe that the boundary points on $\Gamma_1$ farther away from $\Gamma_0$, i.e. $\big( \mathbf{x}_i^1 \big)_{i = \overline{13,38}}$, are less significant as measurements than the boundary points on $\Gamma_1$ closer to $\Gamma_0$, i.e. $\big( \mathbf{x}_i^1 \big)_{i = \overline{1,12}}$. 
This may also be correlated with the fact that the truncated solutions in \Cref{Fig5a} oscillate more rapidly at the boundary points on $\Gamma_1$ farther away from $\Gamma_0$ with respect to increasing the truncation parameter $r$. 
Moreover, an L--curve type criterion for the selection of the truncation parameter $r$ applied to \Cref{Fig2a,Fig2b} suggests that a good choice would be $r \in \overline{3,6}$, with the mention that $r = 6$ is preferred since the approximate eigenvalues $\widetilde{\lambda}_6$ and $\widetilde{\lambda}_7$ are well separated. 

Similar results to those depicted in \Cref{Fig5ad} have been obtained for \Cref{ex:ex01} and the exact solution given by \eqref{eq:2D-sol2} and \eqref{eq:2D-sol3}, and these are presented in \Cref{Fig5be,Fig5cf}, respectively. 
It should be mentioned that, in case of \Cref{ex:ex01} and the exact solution given by \eqref{eq:2D-sol2} and \eqref{eq:2D-sol3}, the accuracy of the numerical reconstruction of the Dirichlet data on $\Gamma_1$ is guaranteed by the good approximation of its exact counterpart $u^{\rm (ex)}\big|_{\Gamma_1}$ via its projection on the linear subspace spanned by the first ten--eleven eigenvectors displayed in \Cref{Fig4b}. 
However, if the unknown Dirichlet data on $\Gamma_1$ is not well approximated by these so-called stable eigenvectors, then its reconstruction from noisy measurements may turn out to be inaccurate.

\Cref{Fig7a} presents the color map on the inaccessible inner boundary $\Gamma_1$, retrieved from one sample drawn from the random vector $\left(\dfrac{1}{M_D} \displaystyle \sum_{i=1}^{M_D} \rho_{\mathbf{x}_i^D, \omega^1, \varepsilon, N}\big( \mathbf{x}_j^1 \big)\right)_{j= \overline{1,M_1}}$, for \Cref{ex:ex02}, whilst ten independent samples drawn from the harmonic density estimator $\rho_{\mathbf{x}_i^D, \omega^1, \varepsilon, N}\big( \mathbf{x}_j^1 \big)$, $j \in \big\{ 0, 4, 8, \ldots, 36 \big\}$, and their empirical mean are displayed in \Cref{Fig7b}. 
These two figures reconfirm that the inaccessible boundary $\Gamma_1$ is uniformly relevant for the locations of the internal measurements due to the rotation invariance of the domain. 
The remaining Figures~\ref{Fig7}(c)--(f) associated with \Cref{ex:ex02} illustrate ten independent samples drawn from the harmonic density estimator $\rho_{\mathbf{x}_i^D, \omega^1, \varepsilon, N}\big( \mathbf{x}_j^1 \big)$, $i \in \big\{ 0, 4, 8, \ldots, 36 \big\}$, and their empirical (\Cref{Fig7c}), the eigenvalues estimators $\widetilde{\lambda}_i^{\mathbf{\nu}}$, $i = \overline{1,M_D}$, for the direct operator $T^{\ast}_{\boldsymbol{\nu}} T_{\boldsymbol{\nu}}$ on a linear and a semi-logarithmic scales (\Cref{Fig7d,Fig7e}, respectively), and the eigenvectors estimators $\widetilde{u}_\ell$, $\ell = \overline{1,15}$, given by \eqref{eq:eigenfunctions_MC}, for the corresponding first fifteen eigenvectors of $T^{\ast}_{\boldsymbol{\nu}} T_{\boldsymbol{\nu}}$, retrieved at $\big( \mathbf{x}_j^1 \big)_{j = \overline{1,M_1}} \subset \Gamma_1$ (\Cref{Fig7f}).
The conclusions drawn for the results obtained for \Cref{ex:ex02} and presented in \Cref{Fig7} are similar to those retrieved for \Cref{ex:ex01} and displayed in \Cref{Fig01,Fig02}. However, \Cref{Fig7f} requires some additional comments. If one looks at Figures~\ref{Fig7}(a)--(e), it might seem that the corresponding ten samples have a large variance from their empirical mean. 
This phenomenon can be explained as follows. As derived in \cite{CiGrMaII}, the recovery of a particular eigenvector is very much conditioned by the multiplicity and the gap of its corresponding eigenvalue.
From \Cref{Fig7d} one can see that all eigenvalues except the first one seem to have multiplicity two, and hence one should really account for the recovery of the span of eigenvectors $2$ and $3$, the span of eigenvectors $4$ and $5$, etc.

In \Cref{Fig8be}, we test TSVD--Monte Carlo estimators derived in \Cref{s:3} for \Cref{ex:ex02} and the exact solution \eqref{eq:2D-sol2}. 
More precisely, ten independent samples drawn from the Monte Carlo estimator $u^{(r)}_{\omega^1, \varepsilon, N}$ of the corresponding SVD solution truncated at the $r-$th estimated eigenvalue, $r = \overline{1,25}$, at $\big( \mathbf{x}_j^1 \big)_{j = \overline{1, M_1}} \subset \Gamma_1$, retrieved for \Cref{ex:ex02}, their corresponding empirical mean and the exact solution \eqref{eq:2D-sol2} are displayed in \Cref{Fig8b}, whilst the corresponding absolute error in the reconstructed TSVD--Monte Carlo solution at the accessible internal points, $\big\vert u_i^D - \widehat{u}_i^{D,r} \big\vert, i = \overline{1, M_D}$, $r = \overline{1, 25}$, is presented in \Cref{Fig8e}. It can be concluded from \Cref{Fig8b,Fig8e} that although the exact solution \eqref{eq:2D-sol2} can be stably reconstructed within the subspace spanned by the first five eigenvectors of of the matrix operator $\Lambda_{\omega^1, \varepsilon, N}^{\boldsymbol{\nu}}$ and hence the estimators for the first five eigenfunctions of the direct operator $T^{\ast}_{\boldsymbol{\nu}} T_{\boldsymbol{\nu}}$, one cannot discriminate any of the twenty-five TSVD solutions illustrated in \Cref{Fig8b} without some discriminative {\it a priori} information if the measurement error is lager than $O(10^{-2})$.

The geometry considered in \Cref{ex:ex03} is similar to that of \Cref{ex:ex02}, however the inner boundary $\Gamma_1$ is a circle with a smaller radius than that of \Cref{ex:ex02}, and is therefore investigated along with the exact solution \eqref{eq:2D-sol2} for comparison purposes with \Cref{ex:ex02}. The numerical results retrieved for \Cref{ex:ex03} and \eqref{eq:2D-sol2}, i.e. the counterparts for the corresponding results obtained for \Cref{ex:ex02} and \eqref{eq:2D-sol2}, are presented in \Cref{Fig9,Fig10be}. 
As expected, in the case of \Cref{ex:ex03}, the corresponding eigenvalues estimators $\widetilde{\lambda}_i^{\mathbf{\nu}}$, $i = \overline{1,M_D}$, for the direct operator $T^{\ast}_{\boldsymbol{\nu}} T_{\boldsymbol{\nu}}$ depicted in \Cref{Fig9d} are significantly smaller than those corresponding to the more stable situation considered in \Cref{ex:ex02}, whilst the estimators of the fifteen TSVD solutions displayed in \Cref{Fig10b} fail to reconstruct the exact solution \eqref{eq:2D-sol2}. 
Nonetheless, any of the fifteen TSVD solutions illustrated in \Cref{Fig10e} matches the observed data at the internal locations with an absolute uniform error less then $O(10^{-4})$. 
Clearly, it is possible that certain exact solutions that are well represented by, e.g., the first fifteen eigenvectors depicted in \Cref{Fig9f}, may be stably recovered by the TSVD solutions. Such an example is given by \eqref{eq:2D-sol1} and its numerical reconstruction is provided in the extended version of this manuscript, see \cite{CiGrMaExtended}.

In \Cref{ex:ex04}, we consider a more complex geometry than that of Examples~\ref{ex:ex01}--\ref{ex:ex03}, namely a multiply connected two-dimensional domain given by a disc perforated by five smaller discs, together with the exact solution \eqref{eq:2D-sol3} and perform similar tests which are illustrated in \Cref{Fig11,Fig12cf}.  
However, instead of displaying ten independent samples for each of the estimators analysed so far, for simplicity, we decided to present a single one in this case. 
\Cref{Fig11a} shows that one can determine, in a quantifiable manner, the parts (red colour) of the inner boundaries that are more exposed to the location of the internal measurements (red points) have a significantly higher influence on the measurements than the rest of the inner circles (blue colour).
This information can be correlated with the results presented in \Cref{Fig12c}, where one sample drawn from the Monte Carlo estimator of the first fifteen TSVD solutions corresponding to the exact solution \eqref{eq:2D-sol3} is plotted.
It can be seen from \Cref{Fig12f} that the reconstruction of solution \eqref{eq:2D-sol3} in the red regions of the inner boundaries more accurate than in the blue ones.

In \Cref{ex:ex05}, we investigate the data assimilation problem for a harmonic function given by \eqref{eq:2D-sol1} in a two-dimensional simply connected domain, i.e. solely discrete Dirichlet measurements are available at some internal points marked by green dots in \Cref{Fig13a}, whilst $\Gamma_0 = \varnothing$, by using the TSVD--Monte Carlo approach proposed herein and performing essentially the same tests as for Examples~\ref{ex:ex02}--\ref{ex:ex03}.  
It can be seen from \Cref{Fig13a} that the direct operator $T^{\ast}_{\boldsymbol{\nu}} T_{\boldsymbol{\nu}}$ has a smoother effect in the blue regions than in the red ones and hence the capability of the inverse of $T^{\ast}_{\boldsymbol{\nu}} T_{\boldsymbol{\nu}}$ to recover a broader class of solutions is more limited in the blue region than in the red one. 
This phenomenon is also illustrated in \Cref{Fig14a}, where of the exact solution \eqref{eq:2D-sol1} is smoothed out in the blue regions by its the reconstruction via the estimators for the TSVD solutions. 
It is important to recall that the rectangular boundary is discretised in the counterclockwise sense starting from the bottom-right corner.

To investigate the influence of the anisotropy on the numerically reconstructed results provided by the probabilistic approach described in \Cref{s:3}, we consider yet again \Cref{ex:ex02} with the exact solution \eqref{eq:2D-sol4} and perform the same tests as those corresponding to \Cref{ex:ex02} and the exact solution \eqref{eq:2D-sol3}.
As expected, the elliptic measure on $\Gamma_1$ is slightly more concentrated along the principal eigenvector of the anisotropic conductivity matrix $\mathbf{K}$, see \Cref{Fig15a}, whilst the estimators for the elliptic densities are no longer stable under rotations, see \Cref{Fig15c}.
In case of \Cref{ex:ex02} and the anisotropic solution \eqref{eq:2D-sol3} and in contrast to the isotropic solution \Cref{ex:ex02}, the eigenvalue estimators for the direct operator $T^{\ast}_{\boldsymbol{\nu}} T_{\boldsymbol{\nu}}$ do not have multiplicities and the corresponding individual eigenfunction estimators can be properly identified, see \Cref{Fig7d,Fig15f}. 
It should also be noted that the first twenty eigenfunction estimators can be stably reconstructed and that they span a reach subspace where an unknown solution may be captured in a stable manner. 
As can be noticed from \Cref{Fig16a}, the reconstruction of the exact solution \eqref{eq:2D-sol4} for \Cref{ex:ex02} is accurate even when approximated in the space spanned by the first seven eigenfunction estimators of the direct operator $T^{\ast}_{\boldsymbol{\nu}} T_{\boldsymbol{\nu}}$ only. 
However, \Cref{Fig16b} shows that if the internal measurements are affected by noise of at least $O(10^{-3})$ in absolute value, then the empirical mean (black continuous line) of the ten independent samples drawn from the Monte Carlo estimator $u^{(r)}_{\omega^1, \varepsilon, N}$ of the corresponding SVD solution truncated at the $r-$th estimated eigenvalue, $r = \overline{1,25}$, at $\big( \mathbf{x}_j^1 \big)_{j = \overline{1, M_1}} \subset \Gamma_1$, displayed in \Cref{Fig16a}, fits the noisy measurements, as well as the corresponding exact solution (blue continuous line), presented in the very first and the very last subfigures of \Cref{Fig16a}.



\subsection{Three-dimensional examples} 
\label{ss:3Dexamples}

For the three-dimensional examples investigated herein, the $M_0$ boundary points $\big( \mathbf{x}_i^0 \big)_{i= \overline{1,M_0}} \subset \Gamma_0$ and $M_D$ internal points $\big( \mathbf{x}_i^D \big)_{i= \overline{1,M_D}} \subset D$, where measurements are taken, as well as the $M_1$ boundary points $\big( \mathbf{x}^1_i \big)_{i = \overline{1, M_1}} \subset \Gamma_1$, where the numerical solution is computed using the algorithm described in \Cref{ss:algorithm}, are chosen uniformly on the Bauer spherical spiral \cite{bauer2000}. The Cartesian coordinates $\big( x_k, y_k, z_k \big)$, $k = \overline{1,n}$, of $n$ points uniformly distributed on the Bauer spherical spiral of radius $R$ are given by \cite{bauer2000}
\begin{align*}
&x_k = R \, \sin{\phi_k} \, \cos{\theta_k}, \quad 
y_k = R \, \sin{\phi_k} \, \sin{\theta_k}, \quad 
z_k = R \, \cos{\phi_k},\\
&\phi_k = \arccos{\left(1 - \dfrac{2 k-1}{n}\right)}, \quad 
\theta_k = \sqrt{n \pi} \, \phi_k, \quad 
k = \overline{1,n}.
\end{align*}
In addition, we set $\nu_i = \dfrac{1}{\sqrt{M_D}}$, ${i=\overline{1,M_D}}$, and the weight functions $\big( \omega^1_i(\cdot) \big)_{i = \overline{1,M_1}}$ associated with $\big( \mathbf{x}^1_i \big)_{i = \overline{1,M_1}}$ are taken to be the corresponding extrinsic Voronoi weights, see \Cref{ex:A.2}.
 
\begin{ex} \label{ex:ex06}
Consider $D = D_0 \setminus \overline{D}_1 \subset \mathbb{R}^3$, where $D_\ell = \big\{ \mathbf{x} \in \mathbb{R}^3 \: \big| \: \Vert \mathbf{x} - \mathbf{x}^{(\ell)} \Vert < R_\ell \big\}$ and $\mathbf{x}^{(\ell)} = (0,0,0)$, $\ell = 0, 1$, $R_0 = 1$ and $R_1 = 0.5$, while $\Gamma_\ell = \partial D_\ell$, $\ell = 0, 1$, and $\Gamma_D = \big\{ \mathbf{x} \in D \: \big| \: {\rm dist}(\mathbf{x}, \Gamma_0) = 0.05 \big\}$. Here, we set $M_0 = 1000$, $M_1 = 100$, $M_D = 100$, $N = 10^5$ Monte Carlo samples, and $\varepsilon = 10^{-10}$.
\end{ex}

\begin{ex} \label{ex:ex07}
Let $D = D_0 \setminus \overline{D}_1 \subset \mathbb{R}^3$, where $D_\ell = \big\{ \mathbf{x} \in \mathbb{R}^3 \: \big| \: \Vert \mathbf{x} - \mathbf{x}^{(\ell)} \Vert < R_\ell \big\}$, $\ell = 0, 1$, $\mathbf{x}^{(0)} = (0,0,0)$, $R_0 = 1$, $\mathbf{x}^{(1)} = (0.3,0,0)$ and $R_1 = 0.5$, while $\Gamma_\ell = \partial D_\ell$, $\ell = 0, 1$, and $\Gamma_D = \big\{ \mathbf{x} \in D \: \big| \: {\rm dist}(\mathbf{x}, \Gamma_0) = 0.05 \big\}$. Here, $M_0 = 1000$, $M_1 = 100$, $M_D = 100$, $N = 10^5$ Monte Carlo samples, and $\varepsilon = 10^{-10}$.
\end{ex}

For Examples~\ref{ex:ex06}--\ref{ex:ex07}, we also consider three exact solutions in the isotropic case $\mathbf{K} = \mathbf{I}_3$, namely
\begin{subequations} \label{eq:3D-sol}
\begin{align}
\label{eq:3D-sol1}
&u^{\rm (ex)}(x_1,x_2,x_3) = x_1 x_2 + x_2^2 - x_3^2, 
\quad \mathbf{x} = (x_1,x_2,x_3) \in \overline{D}\\
\label{eq:3D-sol2}
&u^{\rm (ex)}(x_1,x_2,x_3) = \dfrac{1}{\sqrt{x_1^2 + x_2^2 + x_3^2}}, 
\quad \mathbf{x} = (x_1,x_2,x_3) \in \overline{D}\\
\label{eq:3D-sol3}
&u^{\rm (ex)}(x_1,x_2,x_3) = \dfrac{3}{4} (1 - x_1^2) + \dfrac{1}{4} x_2^3, 
\quad \mathbf{x} = (x_1,x_2,x_3) \in \overline{D}.
\end{align}
\end{subequations}


For a better illustration of the Monte Carlo estimate of the harmonic densities averaged over their poles, in case of \Cref{ex:ex06}, i.e. the three-dimensional analogue of the two-dimensional \Cref{ex:ex02}, the color map of a single sample drawn from the aforementioned estimate is represented on the inaccessible part $\Gamma_1$ only, see \Cref{Fig17a}, more precisely, on the inner spherical surface discretised by points chosen uniformly on the Bauer spiral. 
Interestingly, \Cref{Fig17d} shows that the multiplicities of the estimated eigenvalues of the direct operator $T^{\ast}_{\boldsymbol{\nu}} T_{\boldsymbol{\nu}}$ are given by $2 k - 1$, $k \in \mathbb{Z}_+$.  
In particular, the eigenvectors estimated in \Cref{Fig17f} should not be regarded individually, but rather grouped as the first eigenvector, the following three ones, the next five ones, etc.
By comparing \Cref{Fig7d,Fig17d}, it can be noticed that the estimated eigenvalues of the direct operator $T^{\ast}_{\boldsymbol{\nu}} T_{\boldsymbol{\nu}}$ and their corresponding gaps are significantly smaller in case of \Cref{ex:ex06} than those corresponding to its two-dimensional analogue given by \Cref{ex:ex02}.
Consequently, we may conclude that when passing from two- to three-dimensions, the direct operator becomes more unstable under inversion and this can be quantified. 
In particular, it becomes more difficult to recover accurately many eigenvectors and this may clearly be noticed by comparing \Cref{Fig7f,Fig17f}. 

Nonetheless, the corresponding estimated eigenspace for \Cref{ex:ex06} is sufficiently rich to reconstruct accurately the exact solution \eqref{eq:3D-sol1} (see \Cref{Fig18a}), but not rich enough to recover accurately the exact solution \eqref{eq:3D-sol2}, which actually is constant (see \Cref{Fig18b}). In the latter case, it is expected that taking more Monte Carlo samples would reduce the variance of the estimators for the eigenvectors, as well as the variance of the estimators for the corresponding TSVD solutions. 
In case of \Cref{ex:ex06}, the corresponding absolute errors in the reconstructed TSVD--Monte Carlo solutions, corresponding to the exact ones \eqref{eq:3D-sol1} and \eqref{eq:3D-sol2}, at the accessible internal points, $\big\vert u_i^D - \widehat{u}_i^{D,r} \big\vert$, $i = \overline{1, M_D}$, $r = \overline{1, 30}$, are presented in \Cref{Fig18d,Fig18e}, respectively.

We finally consider \Cref{ex:ex07} which is similar to \Cref{ex:ex06}, except that the interior sphere is shifted by $0.3$ along the $x-$coordinate. 
This change in the geometry turns out to make the inverse problem investigated herein more stable. 
For example, the estimated eigenvalues for \Cref{ex:ex07} and presented in \Cref{Fig19d} are simple and larger in comparison to those retrieved for \Cref{ex:ex06} and displayed in \Cref{Fig17d}. 
Also, the corresponding estimated eigenvectors depicted in \Cref{Fig19f} are better identified and exhibit smaller variance in comparison to those obtained for \Cref{ex:ex06} and presented in \Cref{Fig17f}.

For the sake of simplicity, we consider yet again the recovery of a single exact solution, namely that given by \eqref{eq:3D-sol2}, for \Cref{ex:ex07}. On the one hand, this exact solution is no longer constant on $\Gamma_1$ and, on the other hand, its reconstruction in the subspace spanned by the first 15--20 eigenvectors is more accurate for \Cref{ex:ex07} than for \Cref{ex:ex06}, see \Cref{Fig20b}. 
Also, \Cref{Fig20e} indicates that the exact solution could be restored in the aforementioned subspace as long as the noise of the interior measurements is at most of $O(10^{-2})$ in absolute value.



\section{Conclusions} \label{section:conclusions}
In this paper, Monte Carlo-based methods for the numerical analysis of the severely ill-posed inverse problem $\eqref{eq:ICP}$ or, more precisely, its corresponding discrete version \eqref{eq:insidemeasurements} have been developed. These methods cover several fundamental aspects that, surprisingly, have not been explored in the literature as yet, namely
\begin{enumerate}[label={\rm (}{\it \alph*}{\rm )}]
\setlength\itemsep{1pt}
\item The densities of the elliptic measures on the inaccessible boundary $\Gamma_1$, with poles prescribed in the domain $D$, have been estimated numerically, whilst numerical coloured maps displaying the possibly non-uniform instability/smoothing features on $\Gamma_1$, for the prescribed locations of the internal measurements, have also been provided.
\item The spectrum of the symmetrised direct operator associated with the discrete inverse problem \eqref{eq:insidemeasurements} has been estimated and, consequently, the ill-posedness of \eqref{eq:insidemeasurements} could be quantified numerically.
\item Not only a solution to the discrete inverse problem \eqref{eq:insidemeasurements}, but also a spectrum of solutions to \eqref{eq:insidemeasurements} has been reconstructed using the SVD of the random matrix that approximates the direct operator under investigation.
\item Memory efficient, meshfree, parallel, and easy to be implemented on GPU algorithms have been proposed. These algorithms have also shown that the probabilistic methods employed therein have very promising capabilities of representing the direct operator and its spectrally truncated inverse. 
\item The numerical method and the corresponding algorithm presented herein are also supported, in a second paper \cite{CiGrMaII}, by a comprehensive and systematic convergence and stability analysis.
\end{enumerate}

Further developments of the numerical methods proposed herein are the following: 
\begin{enumerate}[label={\rm (}{\it \roman*}{\rm )}]
\setlength\itemsep{1pt}
\item Extend the proposed Monte Carlo-based methods to non-homogeneous conductivity tensors. 
\item Investigate whether the techniques developed herein can be efficiently used not only for solving the inverse problem $\eqref{eq:ICP}$ or \eqref{eq:insidemeasurements}, but also for optimising the measurements locations in order to overcome and amelirate as much as possible the instability of the inverse problem.
\item Enhance or upgrade the methods proposed in this study by accounting for {\it a priori} information on the unknown solution, provided this is available.
\end{enumerate} 

\appendix
\section{Appendix: Weight functions and interpolation of discrete boundary data} \label{appendix}

Given a bounded domain $D \subset \mathbb{R}^d$ with the boundary $\partial D$ and $\varepsilon > 0$, we consider the corresponding $\varepsilon-$shells defined as
\begin{equation}\label{defi:shell}
D_\varepsilon \coloneqq 
\big\{\mathbf{x} \in D \; \big| \; {\rm dist}(\mathbf{x},\partial D) < \varepsilon \big\}, 
\quad 
\overline{D}_\varepsilon \coloneqq 
\big\{\mathbf{x} \in \overline{D} \; \big| \; {\rm dist}(\mathbf{x},\partial D) < \varepsilon \big\}.
\end{equation}
To a given finite set of points $\big( \mathbf{x}_i \big)_{i = \overline{1,n}} \subset \partial D$, we associate the following weight functions on $\overline{D}_\varepsilon$ 
\begin{equation}\label{eq:weights}
\begin{aligned}
&\omega_i^\varepsilon : \overline{D}_\varepsilon \longrightarrow [0,1], 
\quad \omega_i^\varepsilon(\mathbf{x}_j) = \delta_{ij},~i,j = \overline{1,n}, 
\quad \sum_{i=1}^{n} \omega_i^\varepsilon(\mathbf{x}) = 1,~\mathbf{x} \in \overline{D}_\varepsilon.
\end{aligned}
\end{equation}
The $\omega^\varepsilon-$interpolant on $\overline{D}_\varepsilon$ at points $(\mathbf{x}_i)_{i = \overline{1,n}} \subset \partial D$ of a given function $f: \partial D \longrightarrow \mathbb{R}$ is defined by
\begin{align}\label{eq:interpolant}
&f_{\omega^\varepsilon} : \overline{D}_\varepsilon \longrightarrow \mathbb{R}, 
\quad f_{\omega^\varepsilon}(\mathbf{x}) 
= \sum_{i=1}^n \omega_i^\varepsilon (\mathbf{x}) f(\mathbf{x}_i), \quad \mathbf{x} \in \overline{D}_\varepsilon.
\end{align}

\begin{ex}[{\bf Inverse distance weighting}]
\label{ex:A.1}
An easy way to construct Lipschitz weights defined directly on $\overline{D}_\varepsilon$ is to consider 
\begin{equation}\label{eq:IDW}
\omega^\varepsilon_i(\mathbf{x}) \sim \left(\frac{\max(0,R_i - \Vert \mathbf{x}-\mathbf{x}_i \Vert)}{R_i \Vert \mathbf{x}-\mathbf{x}_i \Vert}\right)^p, 
\quad \mathbf{x}\in \overline{D}_\varepsilon, 
\quad i = \overline{1, M},
\end{equation}
where $\big( \mathbf{x}_i \big)_{i = \overline{1, M}} \subset \partial D$ are prescribed, whilst $R_i > 0$ is the radius controlling ${\rm supp}\big( \omega^\varepsilon_i \big)$, $i = \overline{1, M}$. 
\end{ex}

\begin{ex}[{\bf Extrinsic Voronoi weights}]
\label{ex:A.2}
Suppose that $\big( \mathbf{x}_i \big)_{i = \overline{1,n}} \subset \partial D$ are given and consider the Voronoi diagrams on $\mathbb{R}^d$
\begin{equation}\label{eq:voronoi_ext}
\widetilde{V}_i \coloneqq \widetilde{V}(\mathbf{x}_i) 
= \big\{ \mathbf{x} \in \mathbb{R}^d \; \big| \; \Vert \mathbf{x} - \mathbf{x}_i \Vert  
= \inf_{j = \overline{1,n}} \Vert \mathbf{x} - \mathbf{x}_j \Vert \big\}, 
\quad i = \overline{1,n}.
\end{equation}
Since $\widetilde{V}_i$, $i = \overline{1,n}$ may be overlapping, for the sake of rigour, we consider the following partition
\begin{equation} \label{eq:disjoint voronoi_ext}
V_i \coloneqq \widetilde{V}_i \setminus \bigcup\limits_{j = 1}^{i-1} \widetilde{V}_j, 
\quad \bigcup\limits_{i = 1}^n V_i = \mathbb{R}^d.
\end{equation}
The induced weights on $\overline{D}_\varepsilon$ are given by
\begin{equation}\label{eq:voronoi weights_ext}
\omega_i^{\varepsilon} (\mathbf{x}) = \mathbf{1}_{V_i\cap \overline{D}_\varepsilon}(\mathbf{x}), 
\quad \mathbf{x}\in \overline{D}_{\varepsilon}, 
\quad i = \overline{1,n}.
\end{equation}
Hence the corresponding $\omega^\varepsilon-$interpolant on $\overline{D}_\varepsilon$ at points $(\mathbf{x}_i)_{i = \overline{1,n}} \subset \partial D$ of $f$, defined by \eqref{eq:interpolant}, is just the piecewise constant interpolant of $f$ given by
\begin{equation}\label{eq: f omega eps voronoi_ext}
f_{\omega^\varepsilon}(\mathbf{x}) = f(\mathbf{x_i}), 
\quad \mathbf{x} \in V_i^\varepsilon, 
\quad i = \overline{1,n}. 
\end{equation}
\end{ex}

Sometimes, it is more natural to define the weight functions on $\partial D$ and not inside the domain $D$. 
However, for the methodology presented herein, it is important to be able to extend this construction to a neighbourhood $\overline{D}_\varepsilon$ of $\partial D$ as considered above. 
This can always be achieved by means of the projection of a point $\mathbf{x} \in \overline{D}$ onto $\partial D$, denoted by $\pi(\mathbf{x})$. 
Recall that if $D$ has a sufficiently smooth boundary, then $\pi$ is single-valued and regular at least in a neighbourhood of $\partial D$, see e.g. \cite[Section 14.6]{GiTr01}.
However, such a projection is not well-defined in general, in the sense that it could be set-valued. 
Nonetheless, we can and further shall define
\begin{equation}\label{eq:projector on boundary}
\pi : \overline{D} \longrightarrow \partial D
\end{equation}
to be any fixed measurable selection of the possibly set-valued mapping
\begin{equation}\label{eq:set_valued}
\mathbf{x} \in \overline{D} \longmapsto 
\big\{ \mathbf{y}\in \partial D \; \big| \; \Vert \mathbf{x} - \mathbf{y} \Vert = {\rm dist}(\mathbf{x}, \partial D) \big\}.
\end{equation}
Note that such a measurable selection always exists by a direct application of \cite[Corollary 10.3]{EK09}.

\begin{ex}[{\bf Intrinsic Voronoi weights}] 
\label{ex:A.3}
Suppose that $\big( \mathbf{x}_i \big)_{i = \overline{1,n}} \subset \partial D$ are given and denote by ${\rm dist}_{\partial D}(\cdot,\cdot)$ the intrinsic distance on $\partial D$. 
Consider the Voronoi diagrams on $\partial D$
\begin{equation}\label{eq:voronoi}
\widetilde{V}_i 
\coloneqq \widetilde{V}(\mathbf{x}_i) 
= \Big\{ \mathbf{x} \in \partial{D} \; \Big| \; {\rm dist}_{\partial D}(\mathbf{x}, \mathbf{x}_i) 
= \inf_{j = \overline{1,n}} {\rm dist}_{\partial D}(\mathbf{x}, \mathbf{x}_j) \Big\}, 
\quad i = \overline{1,n}.
\end{equation}
Note, yet again, that $\widetilde{V}_i$, $i = \overline{1,n}$, may be overlapping, and for the sake of rigour, we consider the following partition
\begin{equation}\label{eq:disjoint voronoi}
V_i \coloneqq \widetilde{V_i} \setminus \bigcup\limits_{j=1}^{i-1} \widetilde{V}_j, 
\quad \bigcup\limits_{i=1}^n V_i = \partial D.
\end{equation}
The induced weights on $\partial D$ are given by
\begin{equation}\label{eq:voronoi weights}
\omega_i(\mathbf{x}) = \mathbf{1}_{V_i}(\mathbf{x}), \quad i= \overline{1,n},
\end{equation}
provides one with the extended weight functions on $\overline{D}_{\varepsilon}$ defined by
\begin{equation}\label{eq:voronoi weights extended}
\omega_i^\varepsilon(\mathbf{x}) = \mathbf{1}_{V_i}\big( \pi(\mathbf{x}) \big), 
\quad \mathbf{x}\in \overline{D}_\varepsilon, 
\quad i = \overline{1,n}.
\end{equation}
Clearly, the corresponding $\omega^\varepsilon-$interpolant on $\overline{D}_\varepsilon$ at points $(\mathbf{x}_i)_{i = \overline{1,n}} \subset \partial D$ of $f$, defined by \eqref{eq:interpolant}, is just the piecewise constant interpolant of $f$ given by
\begin{equation}\label{eq: f omega eps voronoi}
f_{\omega^\varepsilon}(\mathbf{x}) = f(\mathbf{x_i}), 
\quad \mathbf{x} \in V_i^\varepsilon \coloneqq \pi^{-1}(V_i) \cap \overline{D}_\varepsilon, 
\quad i = \overline{1,n}. 
\end{equation}
\end{ex}


\bigskip
\noindent \textbf{Acknowledgements.} I.C. acknowledges support from the Ministry of Research, Innovation and Digitization (Romania), grant CF-194-PNRR-III-C9-2023.

\addcontentsline{toc}{section}{\refname}
\bibliographystyle{abbrv}

\newpage
\section*{Figures} 
\label{section:figures}

\begin{figure}[H]
\centering
\subfigure[]{
\includegraphics[scale=0.75]{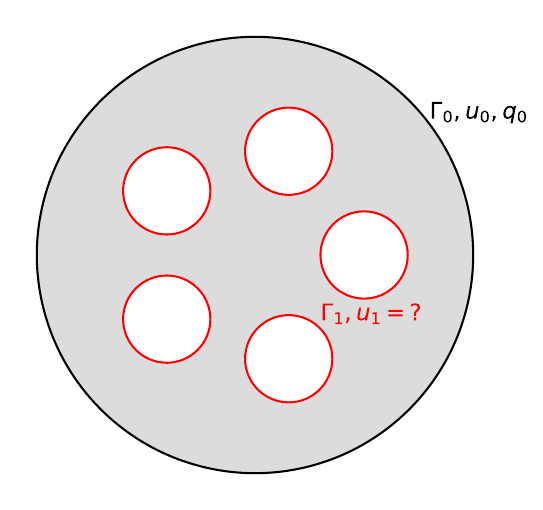}
\label{Fig00a}}
\subfigure[]{
\includegraphics[scale=0.65]{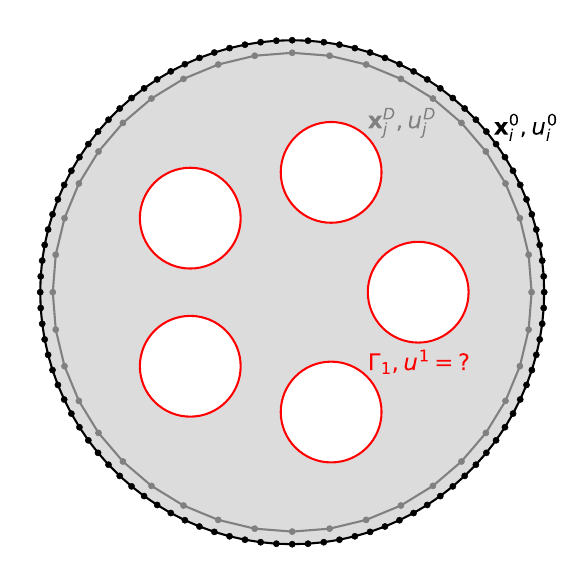}
\label{Fig00b}}
\caption[Figure]{\footnotesize Schematic diagram of \subref{Fig00a}~the continuous inverse Cauchy problem, and \subref{Fig00b}~the discrete version of its generalisation.}
\label{Fig00}
\end{figure}

\begin{figure}[H]
\centering
\subfigure[]{
\includegraphics[scale=0.6]{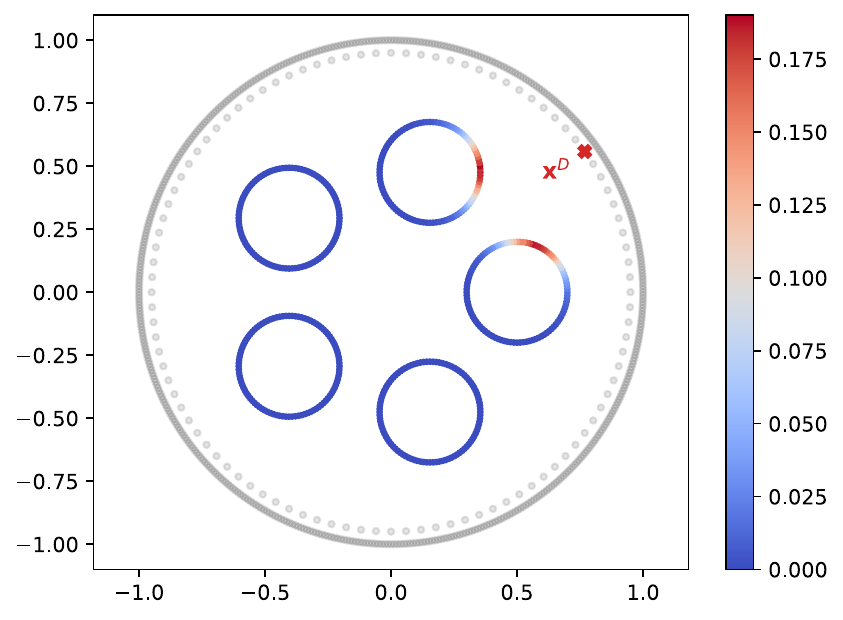}
\label{Fig001a}}
\subfigure[]{
\includegraphics[scale=0.6]{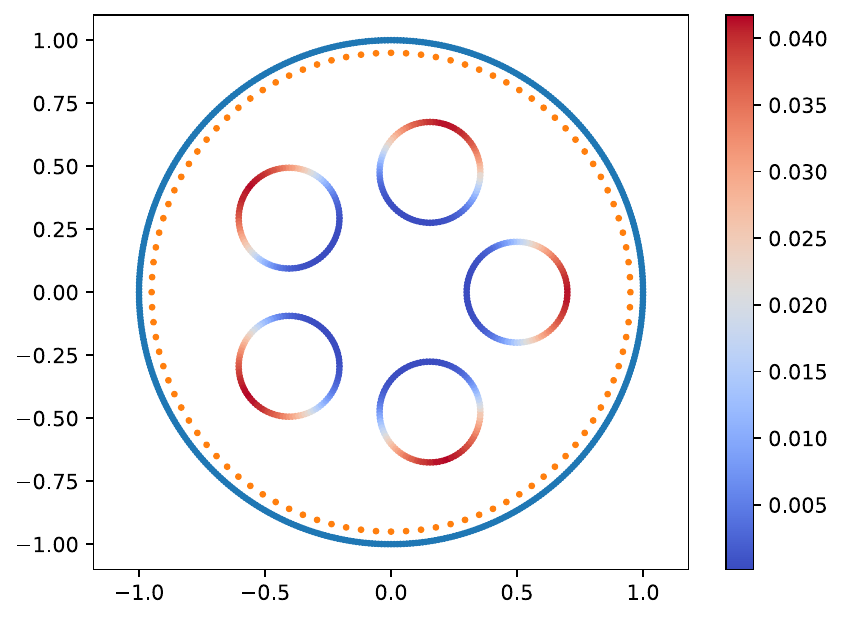}
\label{Fig001b}}
\caption[Figure]{\footnotesize Restriction to the inner boundaries $\Gamma_1$ of \subref{Fig001a}~the harmonic measure with the pole at a generic location $\mathbf{x}^D$, and \subref{Fig001b}~the averaged harmonic measures with the poles at all points represented by the dotted circle, for a disc perforated by five smaller discs, occupied by a two-dimensional isotropic material, $\mathbf{K} = \mathbf{I}_2$, with Dirichlet data on the outer boundary $\Gamma_0$ and discrete Dirichlet measurements on the dotted circle.}
\label{Fig001}
\end{figure}

\begin{figure}[H]
\centering
\subfigure[]{
\includegraphics[scale=0.45]{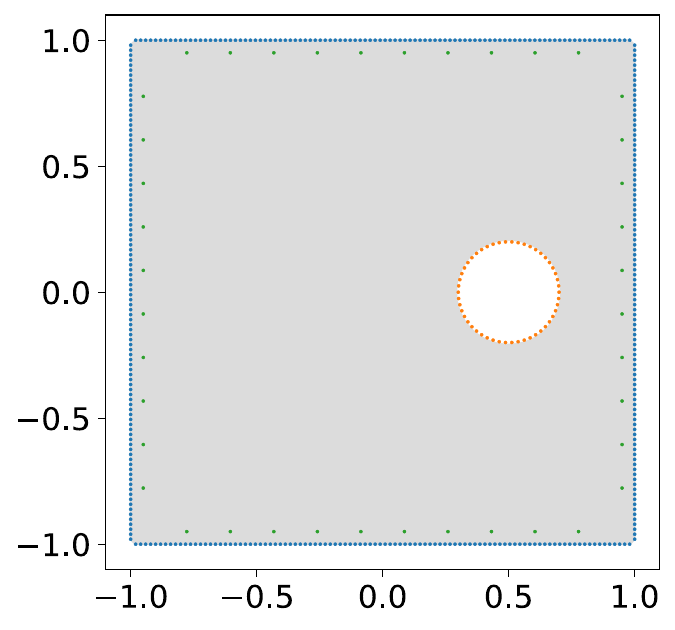}
\label{Fig1}}
\subfigure[]{
\includegraphics[scale=0.45]{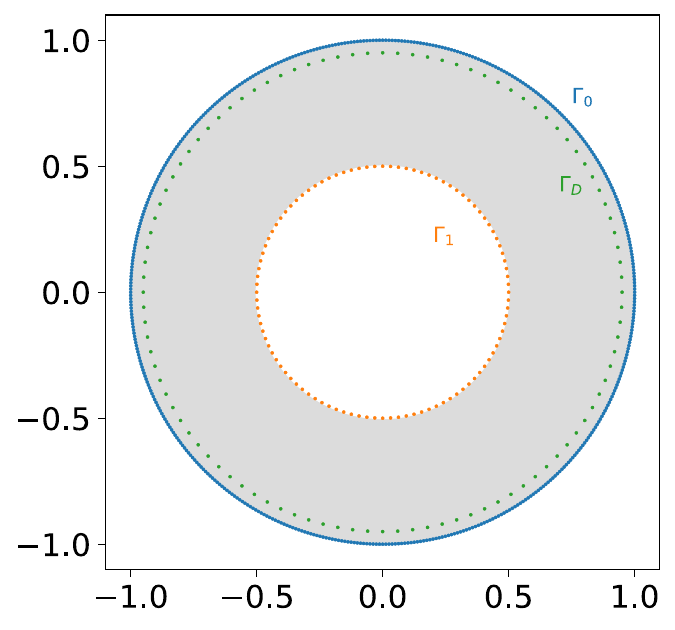}
\label{Fig6a}}
\subfigure[]{
\includegraphics[scale=0.45]{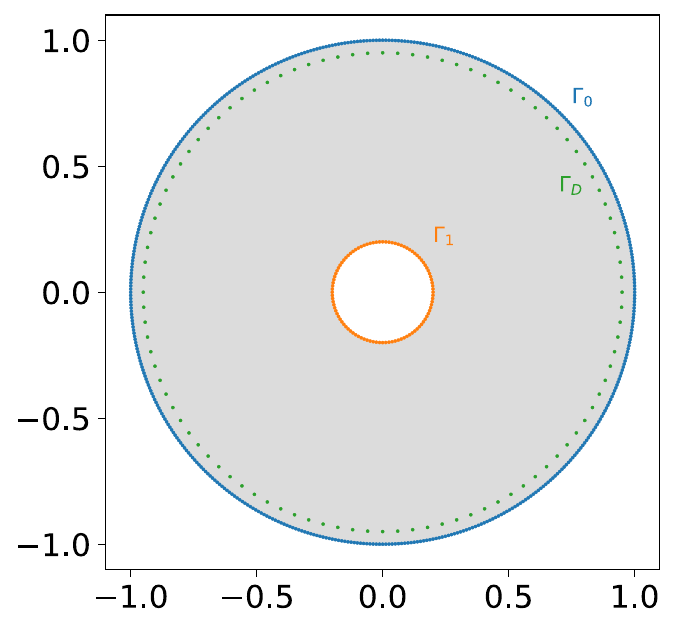}
\label{Fig6b}}
\subfigure[]{
\includegraphics[scale=0.45]{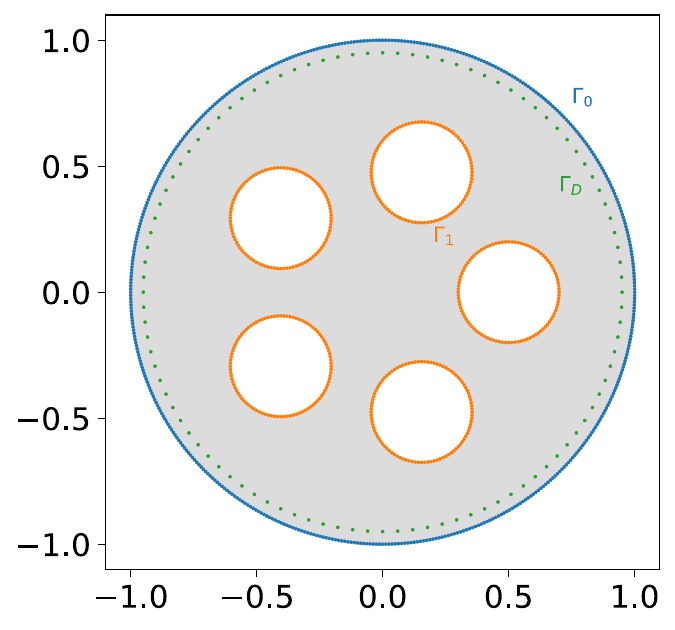}
\label{Fig6c}}
\subfigure[]{
\includegraphics[scale=0.45]{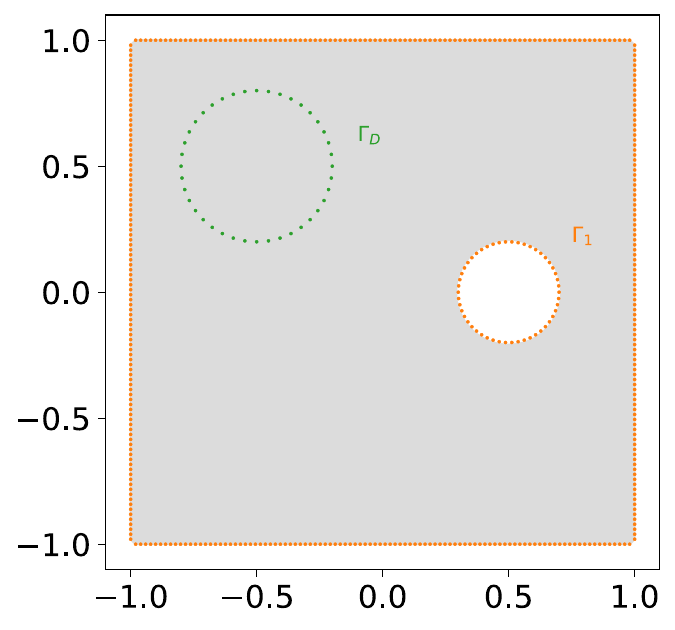}
\label{Fig6d}}
\vspace{-10pt}
\caption[Figure]{\footnotesize Schematic diagram of \subref{Fig1}~\Cref{ex:ex01}, \subref{Fig6a}~\Cref{ex:ex02}, \subref{Fig6b}~\Cref{ex:ex03}, \subref{Fig6c}~\Cref{ex:ex04}, and \subref{Fig6d}~\Cref{ex:ex05}, and the corresponding measurement points on the boundary $\big( \mathbf{x}_i^0 \big)_{i= \overline{1,M_0}} \subset \Gamma_0$ (blue dots) and in the domain $\big( \mathbf{x}_i^D \big)_{i= \overline{1,M_D}} \subset \Gamma_D \subset D$ (green dots), and points on the under-specified boundary $\big( \mathbf{x}^1_i \big)_{i = \overline{1, M_1}} \subset \Gamma_1$ (red dots).}
\label{Fig1-all}
\end{figure}

\begin{figure}[H]
\centering
\subfigure[]{
\includegraphics[scale=0.6]{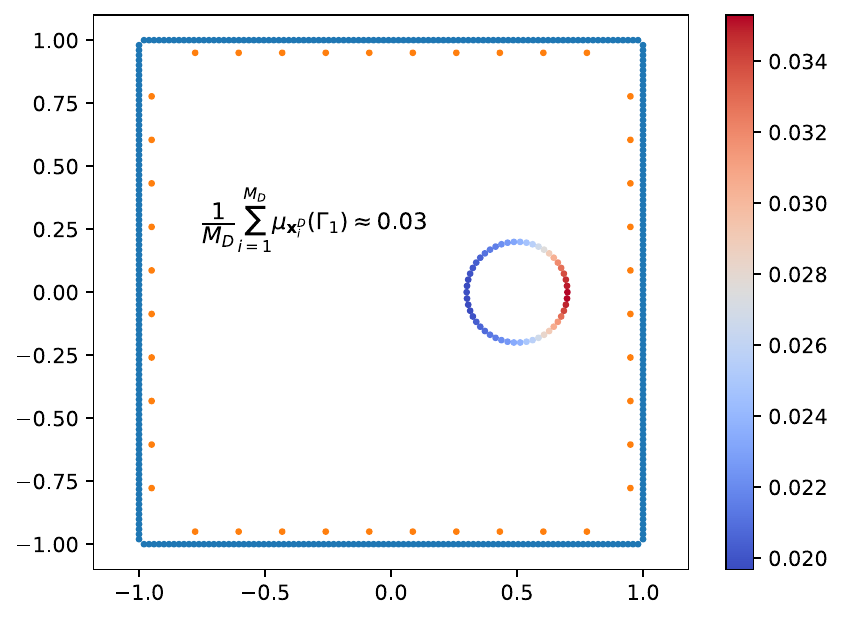}
\label{Fig1a}}
\subfigure[]{
\includegraphics[scale=0.6]{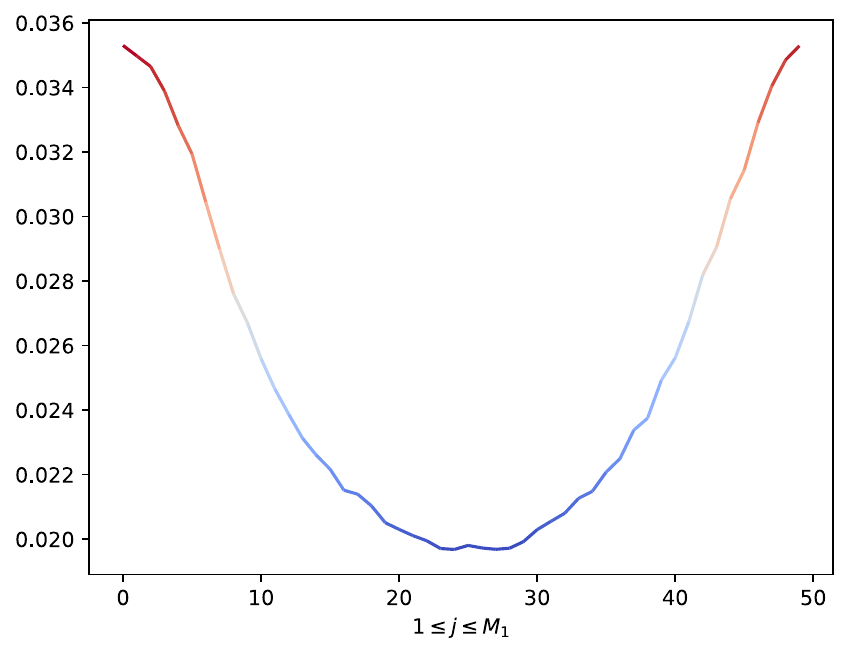}
\label{Fig1b}}
\subfigure[]{
\includegraphics[scale=0.6]{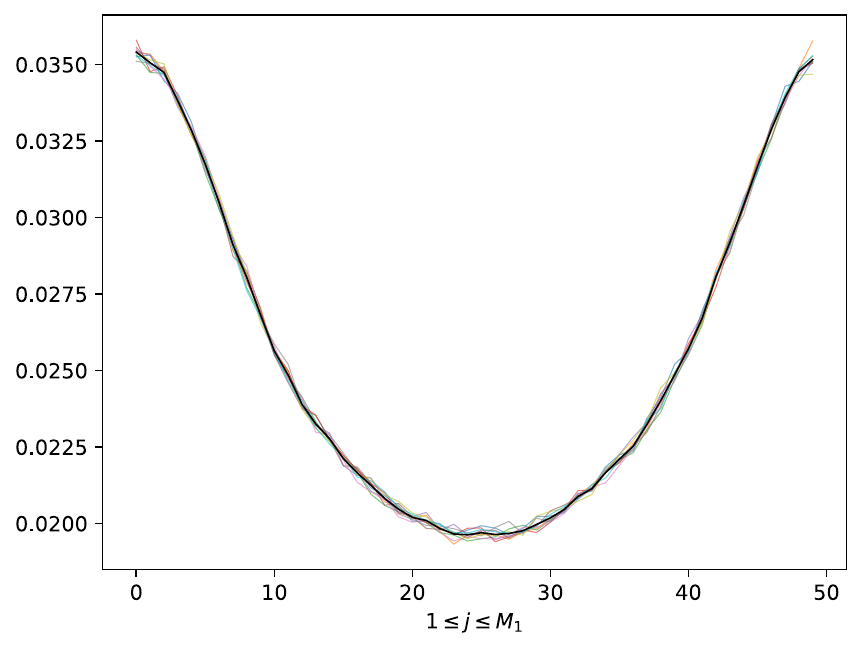}
\label{Fig1c}}
\subfigure[]{
\includegraphics[scale=0.6]{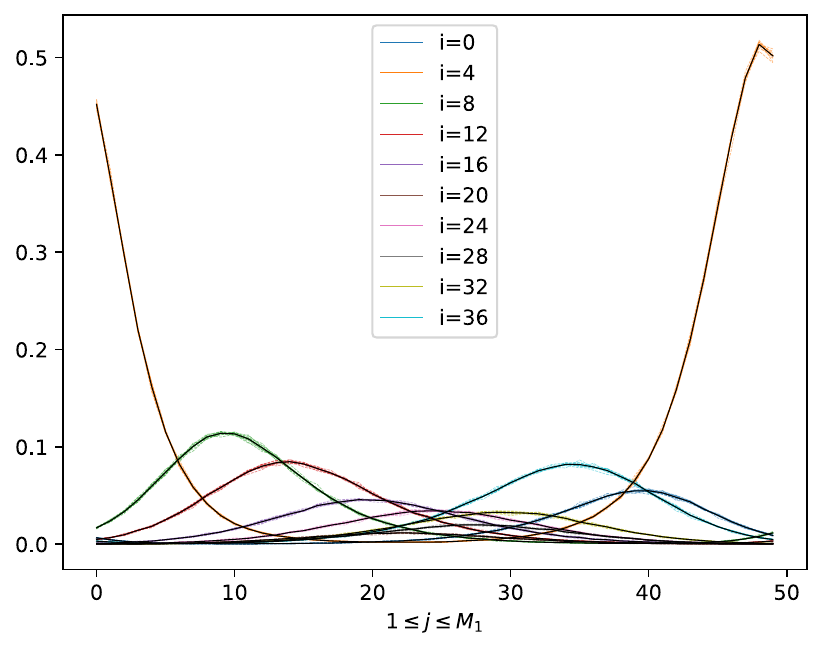}
\label{Fig1d}}
\vspace{-10pt}
\caption[Figure]{\footnotesize \Cref{ex:ex01}: \subref{Fig1a}--\subref{Fig1b}~One sample, and \subref{Fig1c}~Ten i.i.d. samples of $\dfrac{1}{M_D} \displaystyle \sum_{i=1}^{M_D} \rho_{\mathbf{x}_i^D, \omega^1, \varepsilon, N}\big( \mathbf{x}_j^1 \big)$, retrieved on $\Gamma_1$, and their empirical mean (black continuous line). 
\subref{Fig1d}~Ten i.i.d. samples of the elliptic density estimators $\rho_{\mathbf{x}_i^D, \omega^1, \varepsilon, N}\big( \mathbf{x}_j^1 \big)$ and their empirical mean (black continuous line), for $j \in \big\{ 0, 4, 8, \ldots, 36 \big\}$.}
\label{Fig01}
\end{figure}

\begin{figure}[H]
\centering
\subfigure[]{
\includegraphics[scale=0.6]{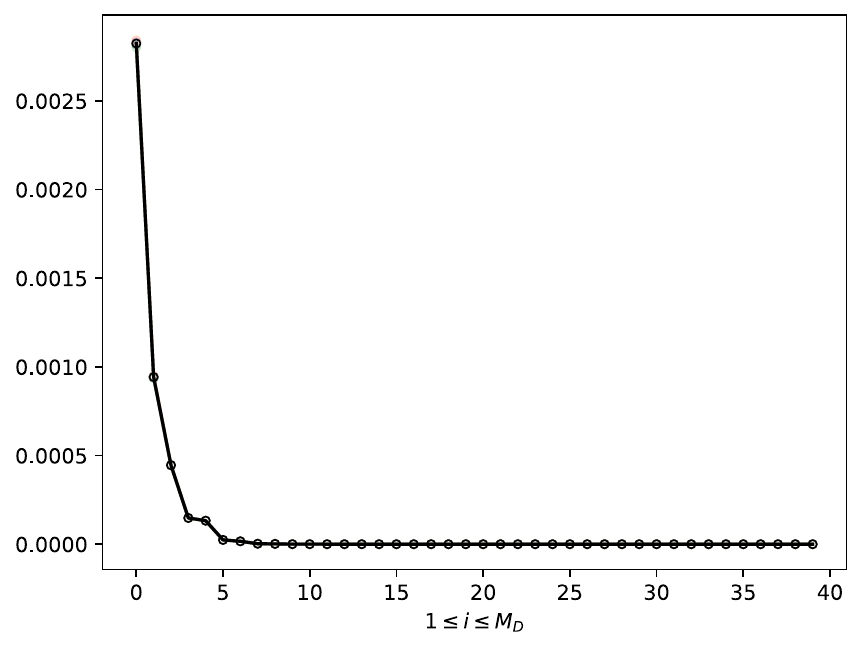}
\label{Fig2a}}
\subfigure[]{
\includegraphics[scale=0.6]{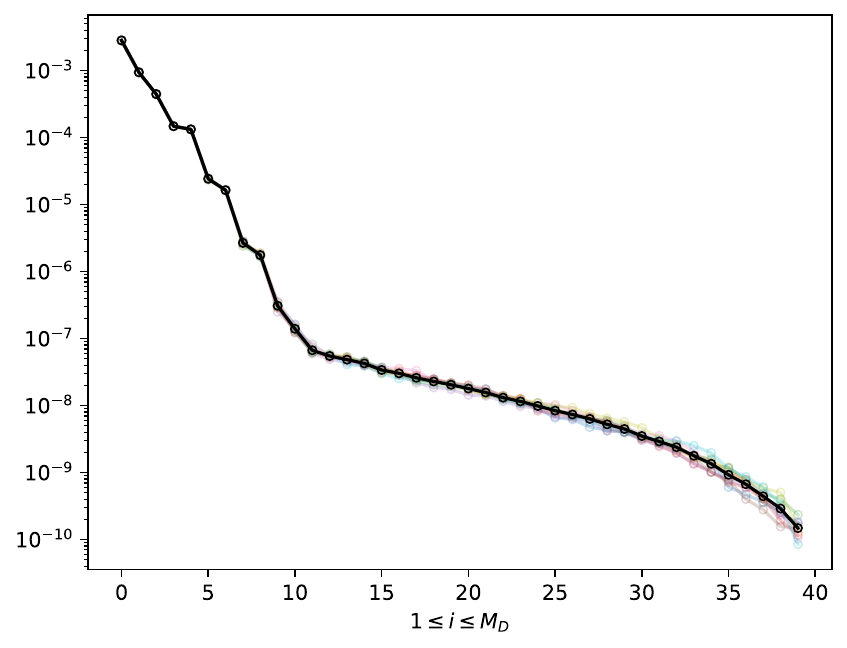}
\label{Fig2b}}
\subfigure[]{
\includegraphics[scale=0.6]{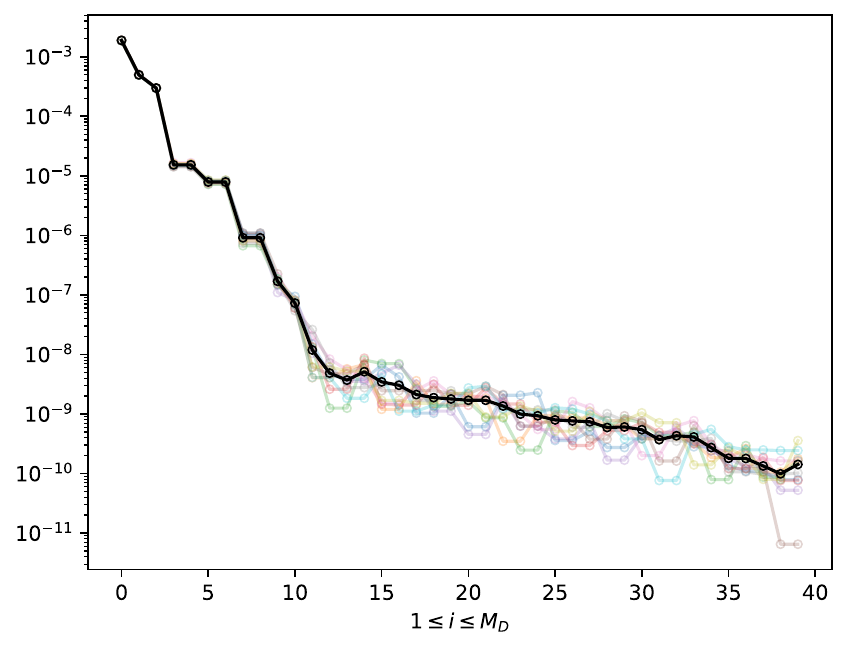}
\label{Fig2c}}
\subfigure[]{
\includegraphics[scale=0.6]{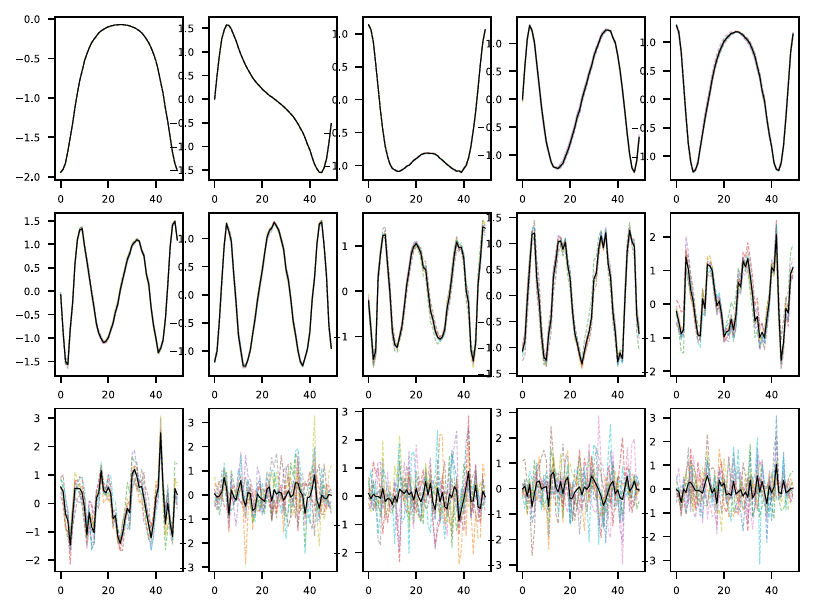}
\label{Fig4b}}
\vspace{-10pt}
\caption[Figure]{\footnotesize \Cref{ex:ex01}: Ten i.i.d. samples of the eigenvalues $\widetilde{\lambda}^{\boldsymbol{\nu}}_i$, $i = \overline{1,M_D}$, of the Monte-Carlo estimator $\boldsymbol{\Lambda}^{\boldsymbol{\nu}}_{\omega^1, \varepsilon, N}$ and their empirical mean (black continuous line), represented on \subref{Fig2a}~a linear scale, and \subref{Fig2b}~a semi-logarithmic $y-$axis scale. 
\subref{Fig2c}~The gaps of the eigenvalues $\widetilde{\lambda}^{\boldsymbol{\nu}}_i$, $i = \overline{1,M_D}$, represented on a semi-logarithmic $y-$axis scale.
\subref{Fig4b}~Ten i.i.d. samples of the eigenvector estimators $\widetilde{u}_\ell$, $\ell = \overline{1,15}$, retrieved on $\Gamma_1$, and their empirical mean (black continuous line).}
\label{Fig02}
\end{figure}

\clearpage

\begin{figure}[H]
\centering
\subfigure[]{
\includegraphics[scale=0.6]{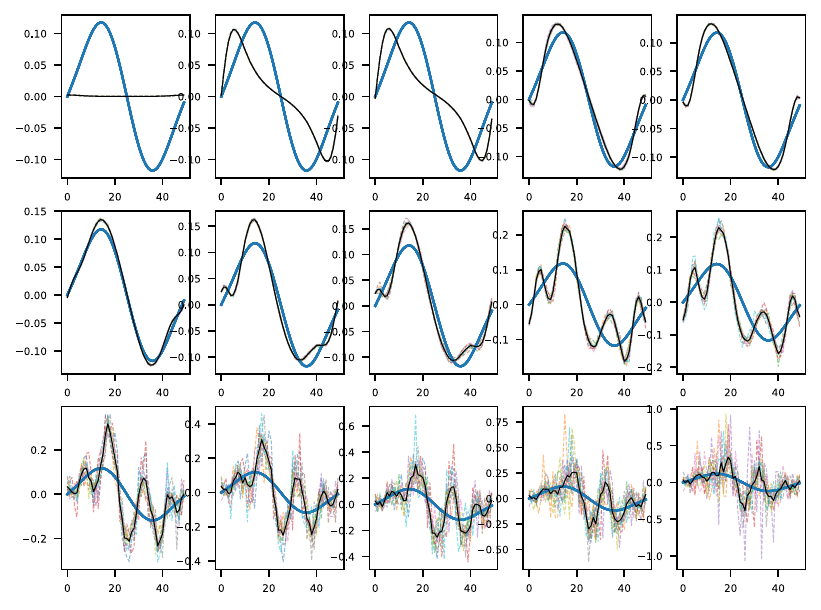}
\label{Fig5a}}
\hspace{5pt}
\subfigure[]{
\includegraphics[scale=0.6]{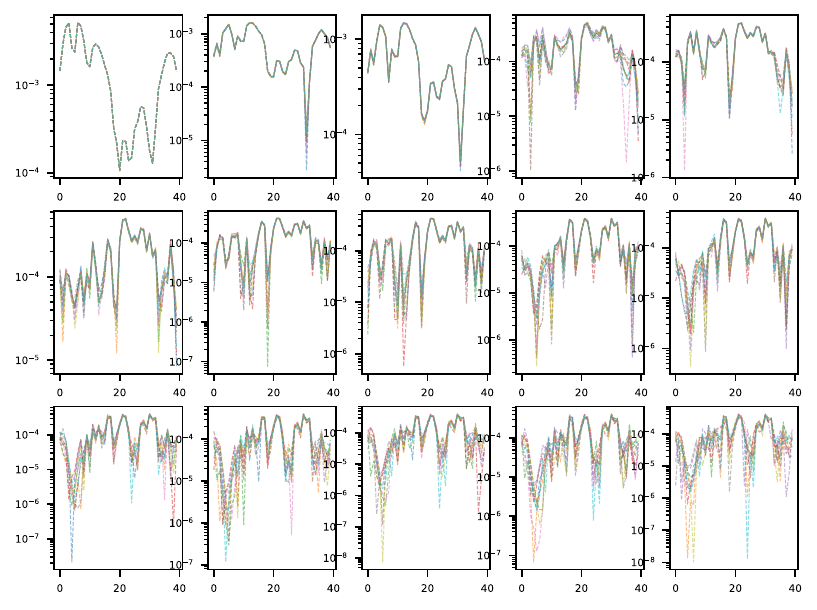}
\label{Fig5d}}
\vspace{-10pt}
\caption[Figure]{\footnotesize \Cref{ex:ex01}: \subref{Fig5a}~Ten i.i.d. samples of the Monte Carlo estimator $u^{(r)}_{\omega^1, \varepsilon, N}$, $r = \overline{1,15}$, at $\big( \mathbf{x}_j^1 \big)_{j = \overline{1, M_1}} \subset \Gamma_1$, their empirical mean (black continuous line), and the exact solution $u^{\rm (ex)}\big|_{\Gamma_1}$ given by \eqref{eq:2D-sol1} (blue continuous line), and \subref{Fig5d}~the absolute errors $\big\vert u_i^D - \widehat{u}_i^{D,r} \big\vert$, $i = \overline{1, M_D}$, $r = \overline{1, 15}$.}
\label{Fig5ad}
\end{figure}
\vspace{-20pt}

\begin{figure}[H]
\centering
\subfigure[]{
\includegraphics[scale=0.6]{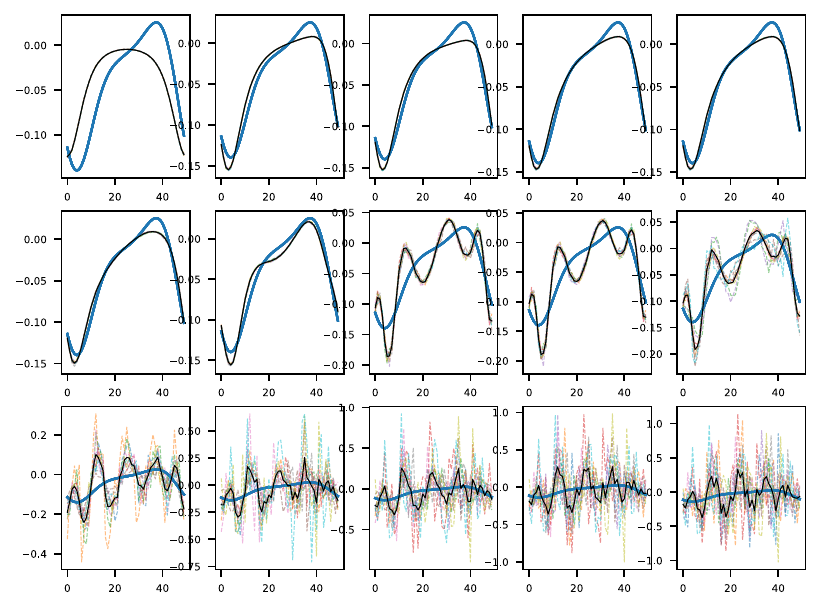}
\label{Fig5b}}
\hspace{5pt}
\subfigure[]{
\includegraphics[scale=0.6]{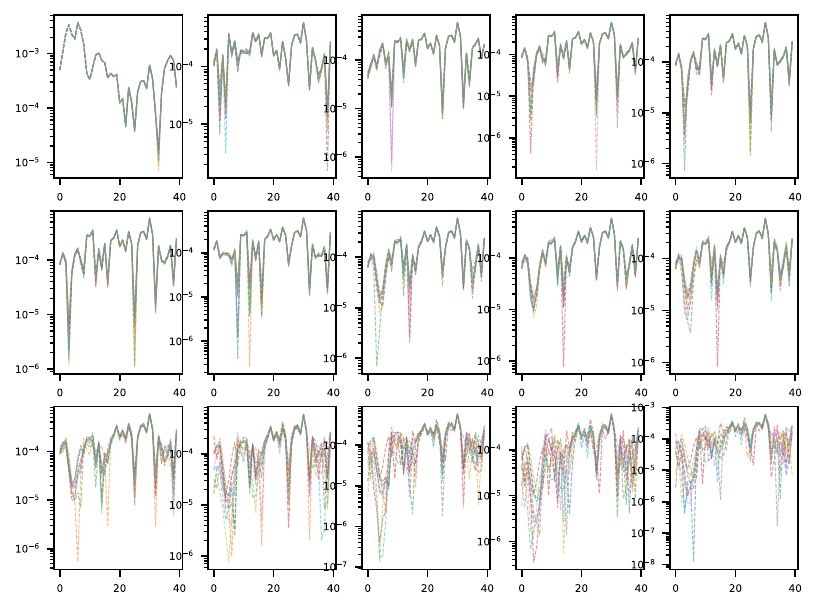}
\label{Fig5e}}
\vspace{-10pt}
\caption[Figure]{\footnotesize \Cref{ex:ex01}: \subref{Fig5b}~Ten i.i.d. samples of the Monte Carlo estimator $u^{(r)}_{\omega^1, \varepsilon, N}$, $r = \overline{1,15}$, at $\big( \mathbf{x}_j^1 \big)_{j = \overline{1, M_1}} \subset \Gamma_1$, their empirical mean (black continuous line), and the exact solution $u^{\rm (ex)}\big|_{\Gamma_1}$ given by \eqref{eq:2D-sol2} (blue continuous line), and \subref{Fig5e}~the absolute errors $\big\vert u_i^D - \widehat{u}_i^{D,r} \big\vert$, $i = \overline{1, M_D}$, $r = \overline{1, 15}$.}
\label{Fig5be}
\end{figure}
\vspace{-20pt}

\begin{figure}[H]
\centering
\subfigure[]{
\includegraphics[scale=0.6]{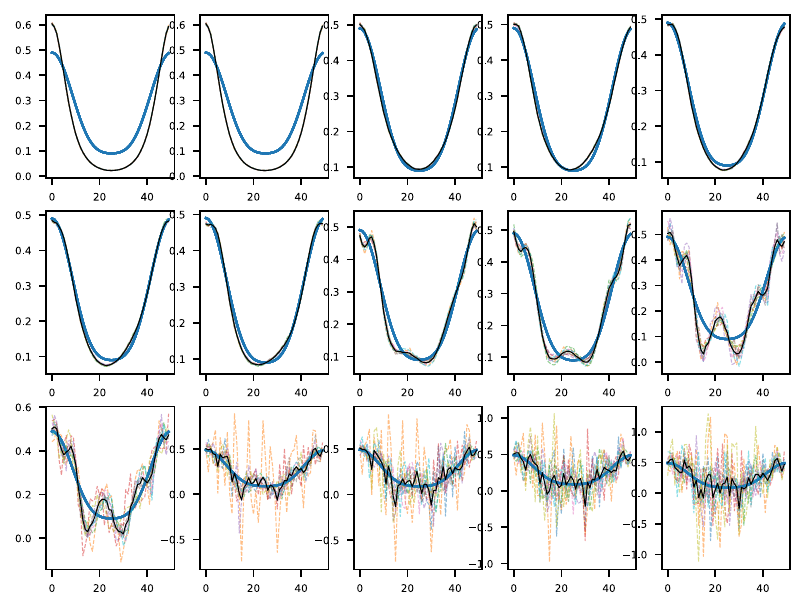}
\label{Fig5c}}
\hspace{5pt}
\subfigure[]{
\includegraphics[scale=0.6]{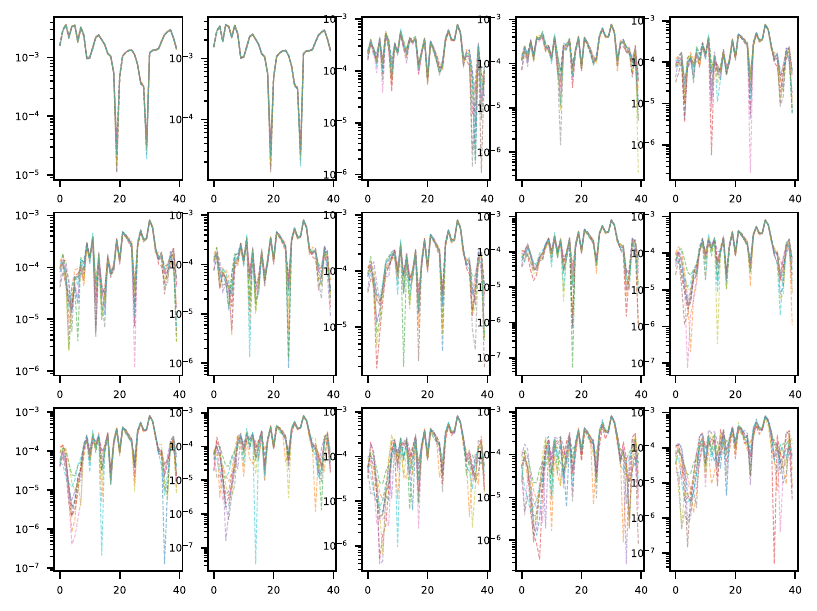}
\label{Fig5f}}
\vspace{-10pt}
\caption[Figure]{\footnotesize \Cref{ex:ex01}: \subref{Fig5c}~Ten i.i.d. samples of the Monte Carlo estimator $u^{(r)}_{\omega^1, \varepsilon, N}$, $r = \overline{1,15}$, at $\big( \mathbf{x}_j^1 \big)_{j = \overline{1, M_1}} \subset \Gamma_1$, their empirical mean (black continuous line), and the exact solution $u^{\rm (ex)}\big|_{\Gamma_1}$ given by \eqref{eq:2D-sol3} (blue continuous line), and \subref{Fig5f}~the absolute errors $\big\vert u_i^D - \widehat{u}_i^{D,r} \big\vert$, $i = \overline{1, M_D}$, $r = \overline{1, 15}$.}
\label{Fig5cf}
\end{figure}

\begin{figure}[H]
\centering
\subfigure[]{
\includegraphics[scale=0.6]{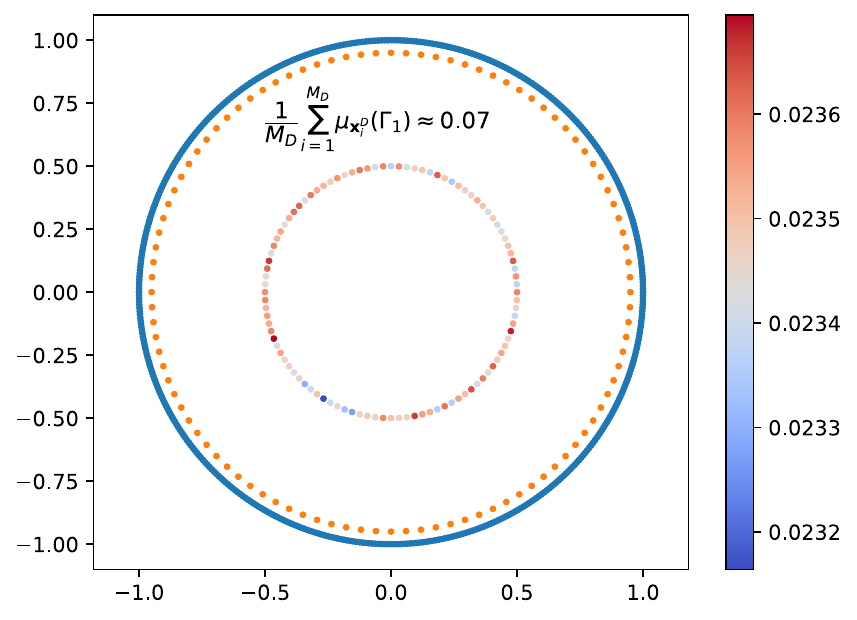}
\label{Fig7a}}
\subfigure[]{
\includegraphics[scale=0.55]{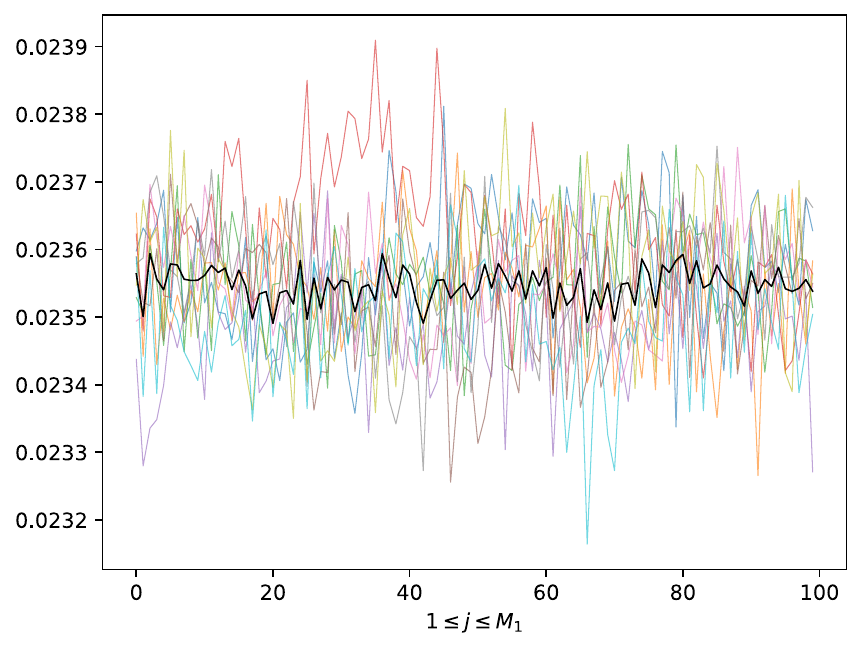}
\label{Fig7b}}
\subfigure[]{
\includegraphics[scale=0.55]{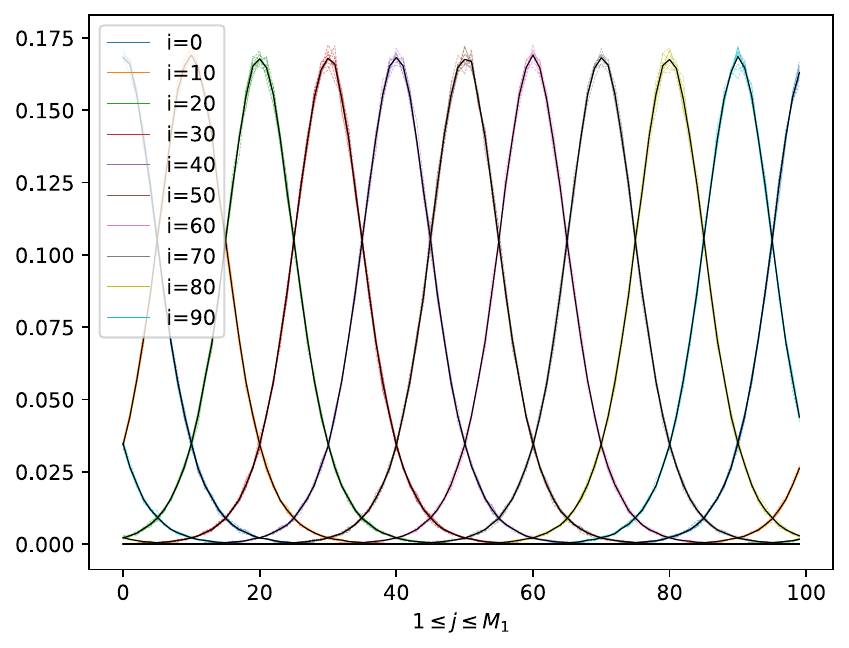}
\label{Fig7c}}
\subfigure[]{
\includegraphics[scale=0.55]{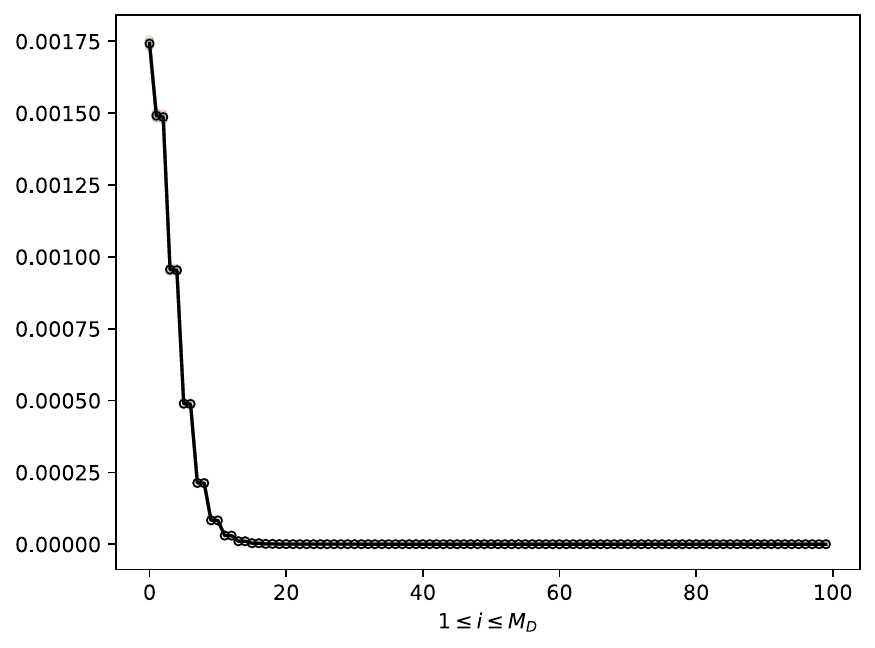}
\label{Fig7d}}
\subfigure[]{
\includegraphics[scale=0.55]{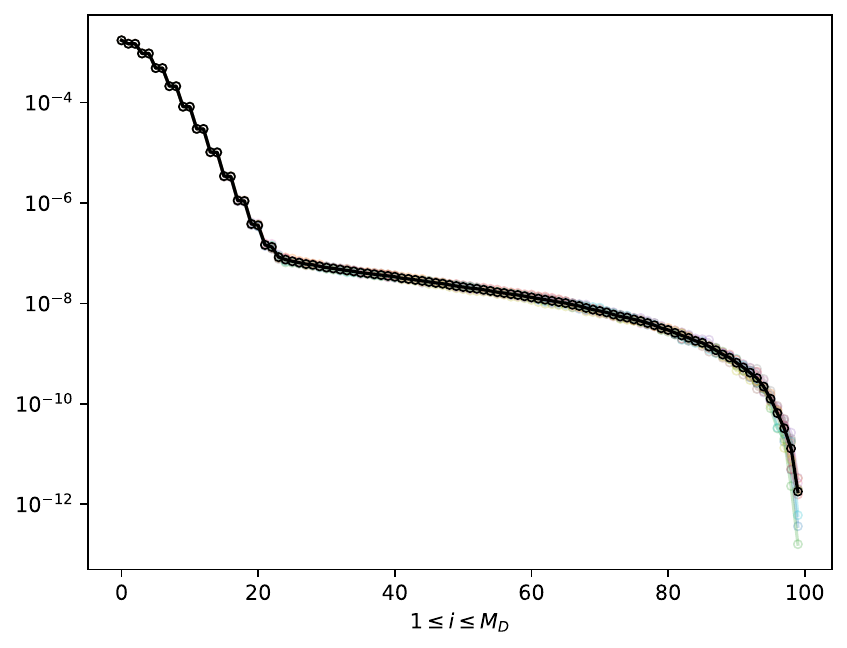}
\label{Fig7e}}
\subfigure[]{
\includegraphics[scale=0.6]{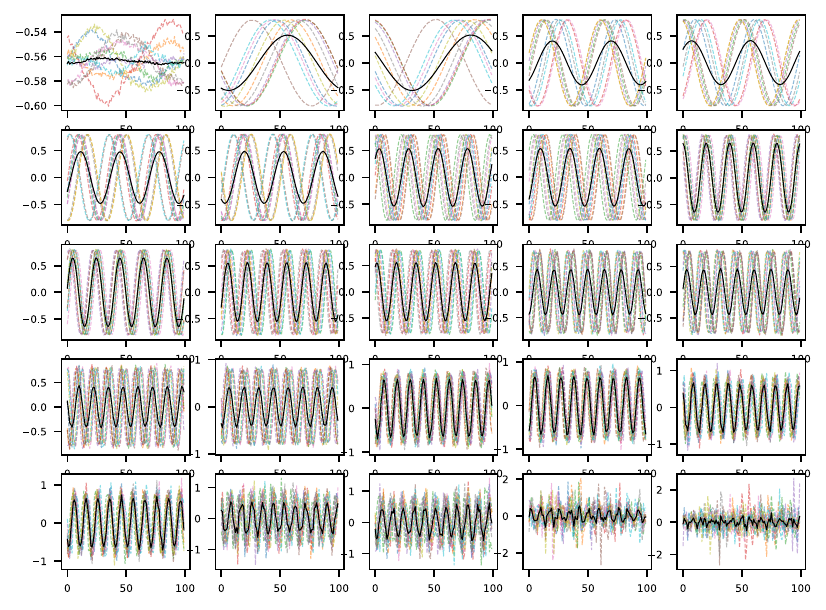}
\label{Fig7f}}
\vspace{-10pt}
\caption[Figure]{\footnotesize \Cref{ex:ex02}: \subref{Fig7a}~One sample, and \subref{Fig7b}~ten i.i.d. samples of $\dfrac{1}{M_D} \displaystyle \sum_{i=1}^{M_D} \rho_{\mathbf{x}_i^D, \omega^1, \varepsilon, N}\big( \mathbf{x}_j^1 \big)$, retrieved on $\Gamma_1$, and their empirical mean (black continuous line). 
\subref{Fig7c}~Ten i.i.d. samples of the elliptic density estimators $\rho_{\mathbf{x}_i^D, \omega^1, \varepsilon, N}\big( \mathbf{x}_j^1 \big)$ and their empirical mean (black continuous line), for $j \in \big\{ 0, 10, 20, \ldots, 90 \big\}$. 
Ten i.i.d. samples of the eigenvalues $\widetilde{\lambda}^{\boldsymbol{\nu}}_i$, $i = \overline{1,M_D}$, of the Monte Carlo estimator $\boldsymbol{\Lambda}^{\boldsymbol{\nu}}_{\omega^1, \varepsilon, N}$ and their empirical mean (black continuous line), represented on \subref{Fig7d}~a linear scale, and \subref{Fig7e}~a semi-logarithmic $y-$axis scale. 
\subref{Fig7f}~Ten i.i.d. samples of the approximate eigenfunction estimator $\widetilde{u}_\ell$, $\ell = \overline{1,25}$, and their empirical mean (black continuous line). 
}
\label{Fig7}
\end{figure}
\vspace{-10pt}
\begin{figure}[H]
\centering
\subfigure[]{
\includegraphics[scale=0.6]{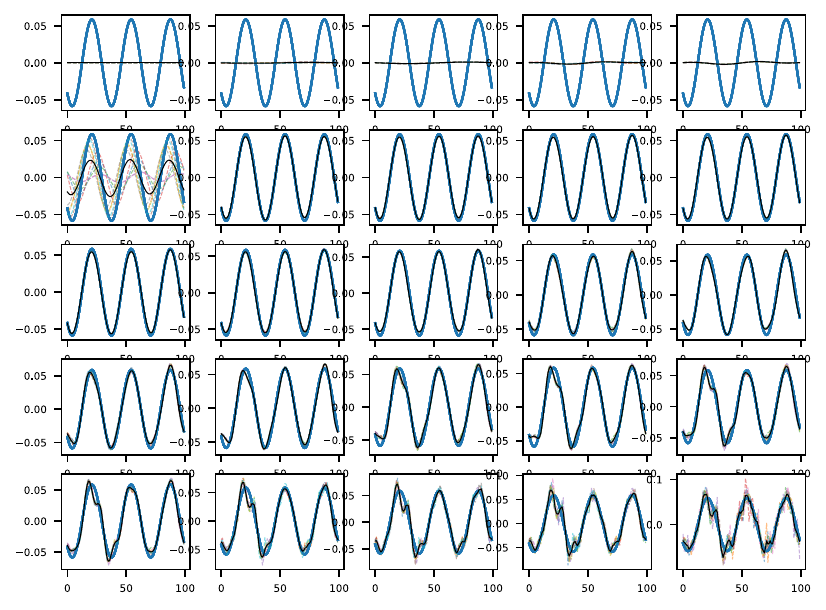}
\label{Fig8b}}
\hspace{5pt}
\subfigure[]{
\includegraphics[scale=0.6]{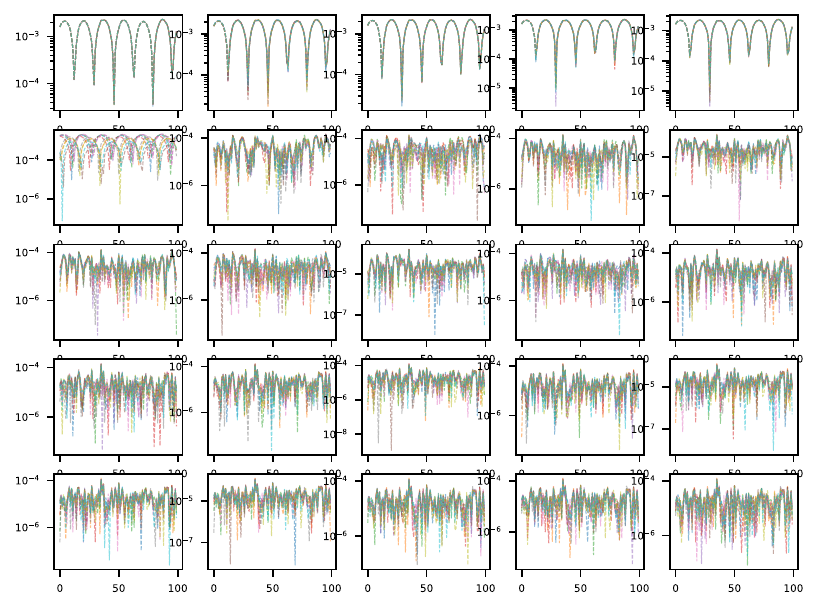}
\label{Fig8e}}
\vspace{-10pt}
\caption[Figure]{\footnotesize \Cref{ex:ex02}: \subref{Fig8b}~Ten i.i.d. samples of the Monte Carlo estimator $u^{(r)}_{\omega^1, \varepsilon, N}$, $r = \overline{1,25}$, at $\big( \mathbf{x}_j^1 \big)_{j = \overline{1, M_1}} \subset \Gamma_1$, their empirical mean (black continuous line), and the exact solution $u^{\rm (ex)}\big|_{\Gamma_1}$ given by \eqref{eq:2D-sol2} (blue continuous line), and \subref{Fig8e}~the absolute errors $\big\vert u_i^D - \widehat{u}_i^{D,r} \big\vert$, $i = \overline{1, M_D}$, $r = \overline{1, 25}$.}
\label{Fig8be}
\end{figure}
\vspace{-20pt}
\begin{figure}[H]
\centering
\subfigure[]{
\includegraphics[scale=0.6]{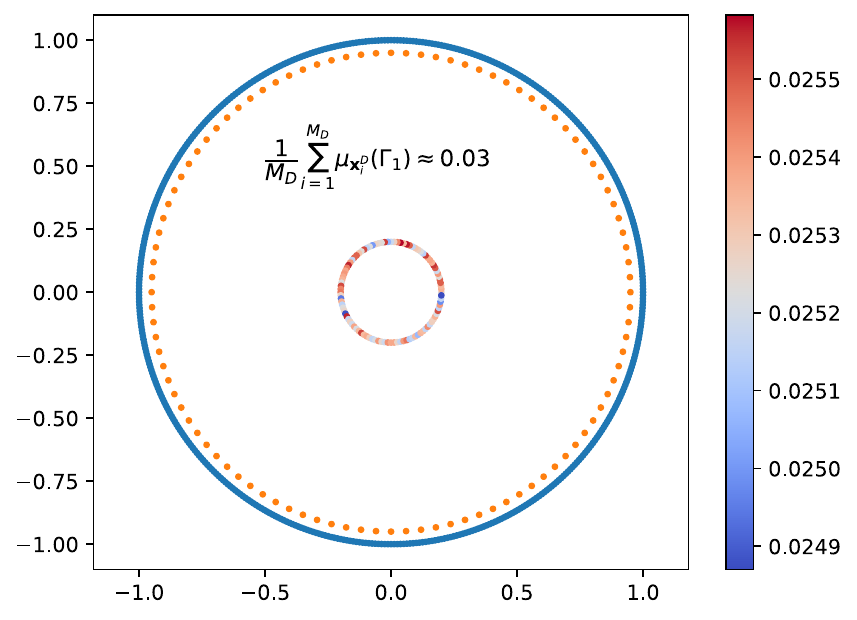}
\label{Fig9a}}
\subfigure[]{
\includegraphics[scale=0.55]{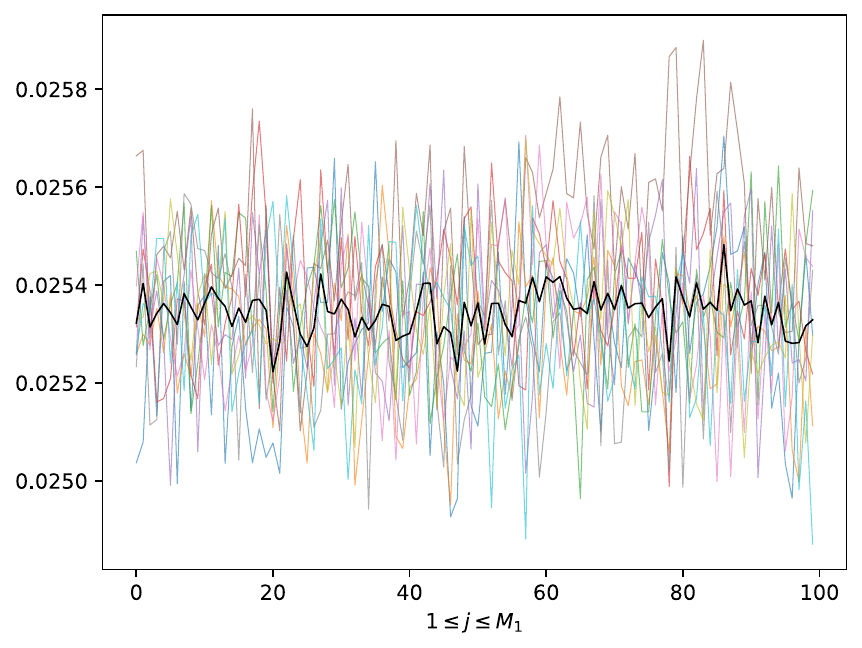}
\label{Fig9b}}
\subfigure[]{
\includegraphics[scale=0.55]{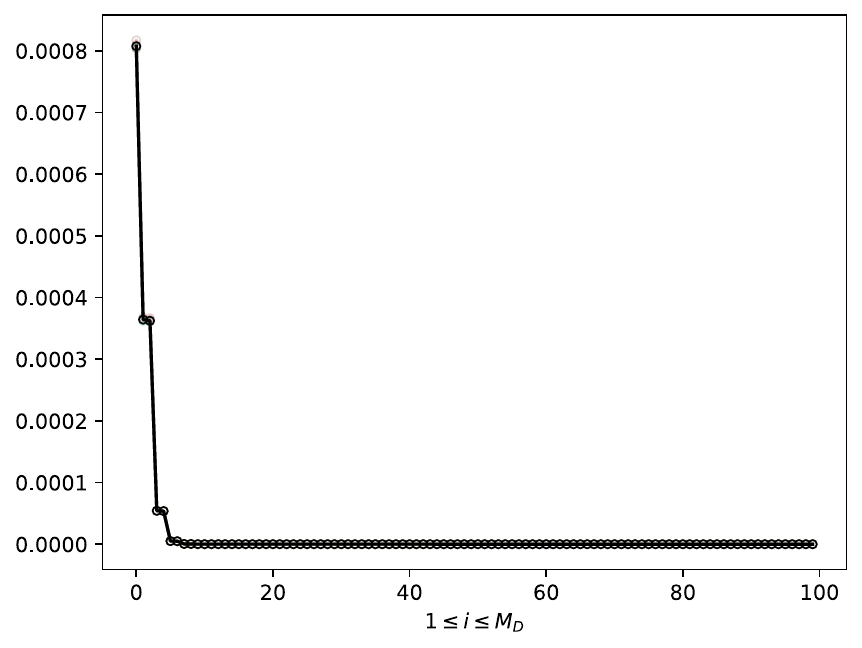}
\label{Fig9d}}
\subfigure[]{
\includegraphics[scale=0.6]{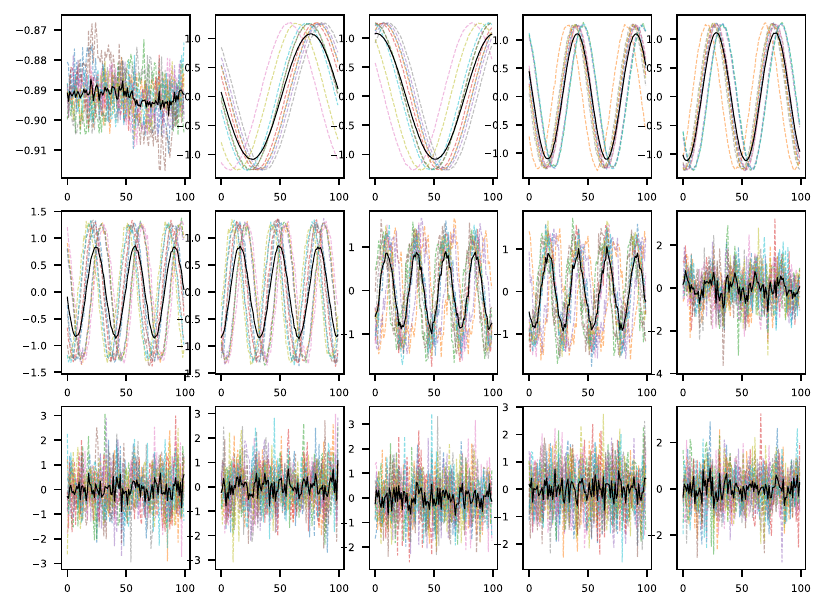}
\label{Fig9f}}
\vspace{-10pt}
\caption[Figure]{\footnotesize \Cref{ex:ex03}: \subref{Fig9a}~One sample, and \subref{Fig9b}~$10$ i.i.d. samples of $\dfrac{1}{M_D} \displaystyle \sum_{i=1}^{M_D} \rho_{\mathbf{x}_i^D, \omega^1, \varepsilon, N}\big( \mathbf{x}_j^1 \big)$, retrieved on $\Gamma_1$, and their empirical mean (black continuous line). 
\subref{Fig9b}~Ten i.i.d. samples of the elliptic density estimators $\rho_{\mathbf{x}_i^D, \omega^1, \varepsilon, N}\big( \mathbf{x}_j^1 \big)$ and their empirical mean (black continuous line), for $j \in \big\{ 0, 10, 20, \ldots, 90 \big\}$
\subref{Fig9d}~Ten i.i.d. samples of the eigenvalues $\widetilde{\lambda}^{\boldsymbol{\nu}}_i$, $i = \overline{1,M_D}$, of the Monte Carlo estimator $\boldsymbol{\Lambda}^{\boldsymbol{\nu}}_{\omega^1, \varepsilon, N}$ and their empirical mean (black continuous line), represented on a linear scale.
\subref{Fig9f}~Ten i.i.d. samples of the approximate eigenfunction estimator $\widetilde{u}_\ell$, $\ell = \overline{1,15}$, and their empirical mean (black continuous line). 
}
\label{Fig9}
\end{figure}
\vspace{-10pt}

\begin{figure}[H]
\centering
\subfigure[]{
\includegraphics[scale=0.6]{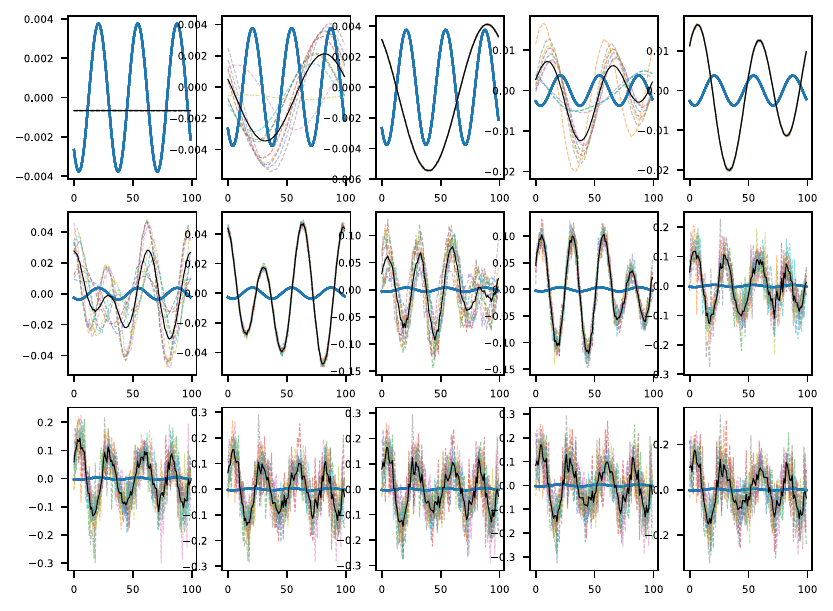}
\label{Fig10b}}
\subfigure[]{
\includegraphics[scale=0.6]{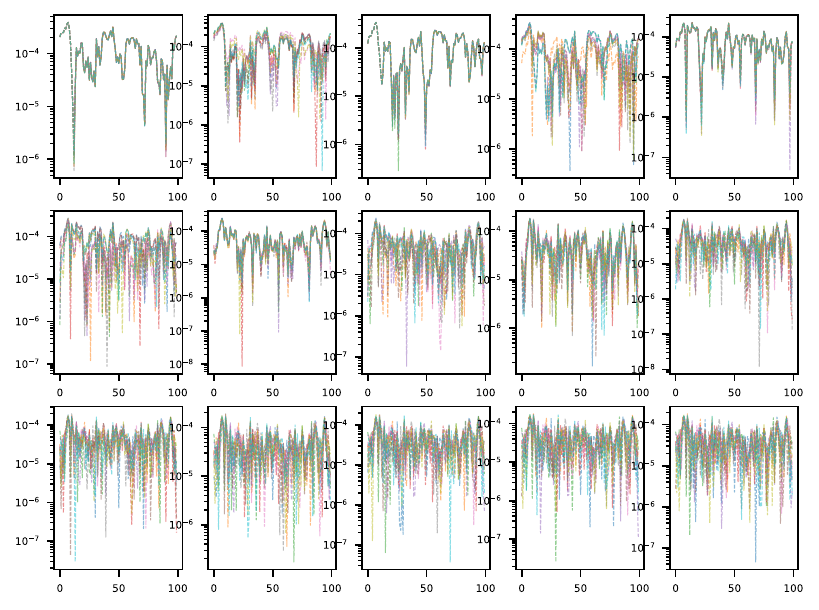}
\label{Fig10e}}
\vspace{-10pt}
\footnotesize \caption[Figure]{\footnotesize \Cref{ex:ex03}: \subref{Fig10b}~Ten i.i.d. samples of the Monte Carlo estimator $u^{(r)}_{\omega^1, \varepsilon, N}$, $r = \overline{1,15}$, at $\big( \mathbf{x}_j^1 \big)_{j = \overline{1, M_1}} \subset \Gamma_1$, their empirical mean (black continuous line), and the exact solution $u^{\rm (ex)}\big|_{\Gamma_1}$ given by \eqref{eq:2D-sol2} (blue continuous line), and \subref{Fig10e}~the absolute errors $\big\vert u_i^D - \widehat{u}_i^{D,r} \big\vert$, $i = \overline{1,M_D}$, $r = \overline{1,15}$.}
\label{Fig10be}
\end{figure}
\vspace{-20pt}

\begin{figure}[H]
\centering
\subfigure[]{
\includegraphics[scale=0.6]{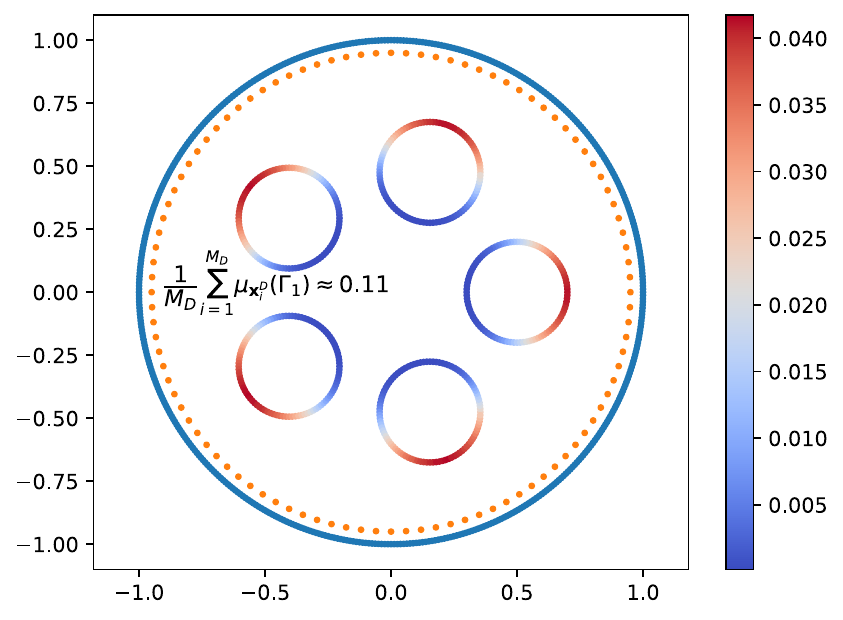}
\label{Fig11a}}
\subfigure[]{
\includegraphics[scale=0.55]{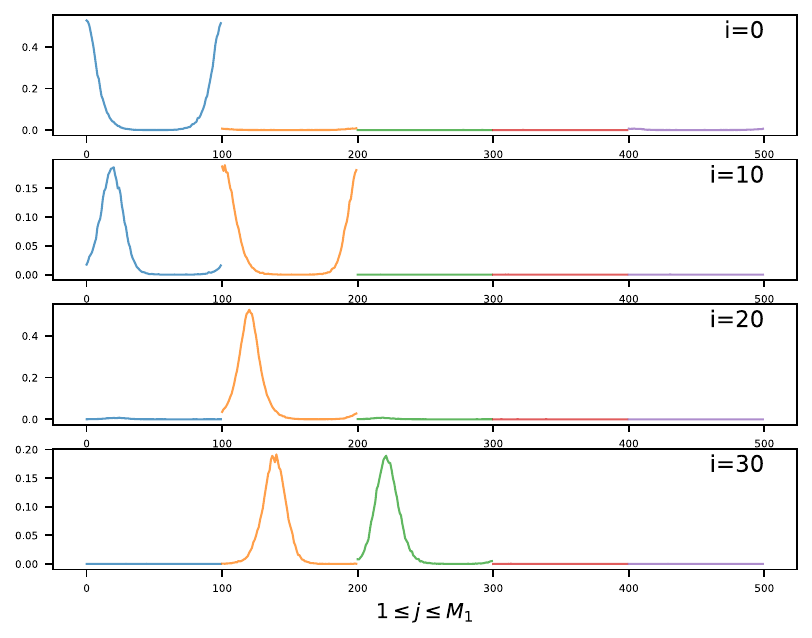}
\label{Fig11c}}
\subfigure[]{
\includegraphics[scale=0.55]{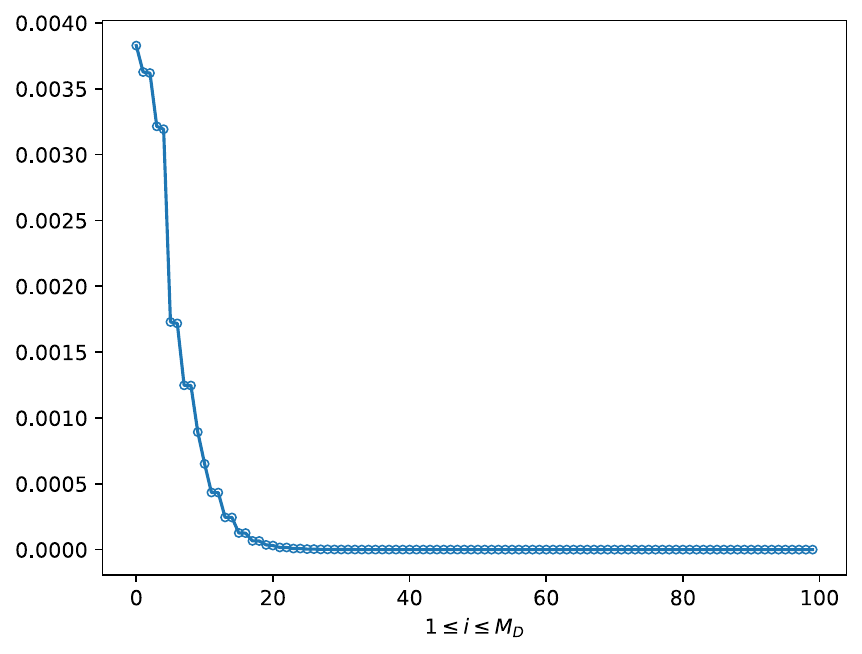}
\label{Fig11d}}
\subfigure[]{
\includegraphics[scale=0.6]{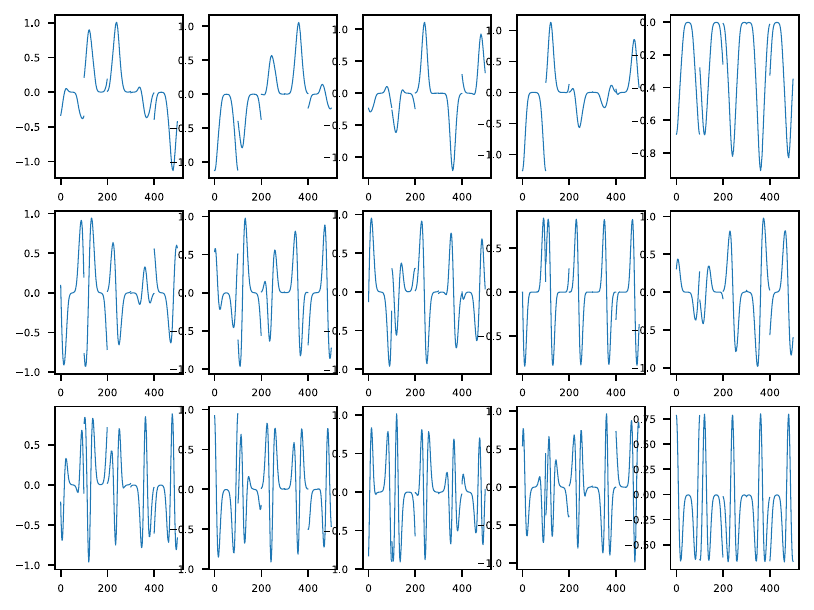}
\label{Fig11f}}
\vspace{-10pt}
\caption[Figure]{\footnotesize \Cref{ex:ex04}: \subref{Fig11a}~One sample of $\dfrac{1}{M_D} \displaystyle \sum_{i=1}^{M_D} \rho_{\mathbf{x}_i^D, \omega^1, \varepsilon, N}\big( \mathbf{x}_j^1 \big)$, retrieved on $\Gamma_1$. 
\subref{Fig11c}~One sample of the elliptic density estimators $\rho_{\mathbf{x}_i^D, \omega^1, \varepsilon, N}\big( \mathbf{x}_j^1 \big)$, for $j \in \big\{ 0, 10, 20, 30 \big\}$. 
\subref{Fig11d}~One sample of the eigenvalues $\widetilde{\lambda}^{\boldsymbol{\nu}}_i$, $i = \overline{1,M_D}$, of the Monte Carlo estimator $\boldsymbol{\Lambda}^{\boldsymbol{\nu}}_{\omega^1, \varepsilon, N}$, represented on a linear scale.
\subref{Fig11f}~One sample of the approximate eigenfunction estimator $\widetilde{u}_\ell$, $\ell = \overline{1,15}$. 
}
\label{Fig11}
\end{figure}
\vspace{-20pt}

\begin{figure}[H]
\centering
\subfigure[]{
\includegraphics[scale=0.6]{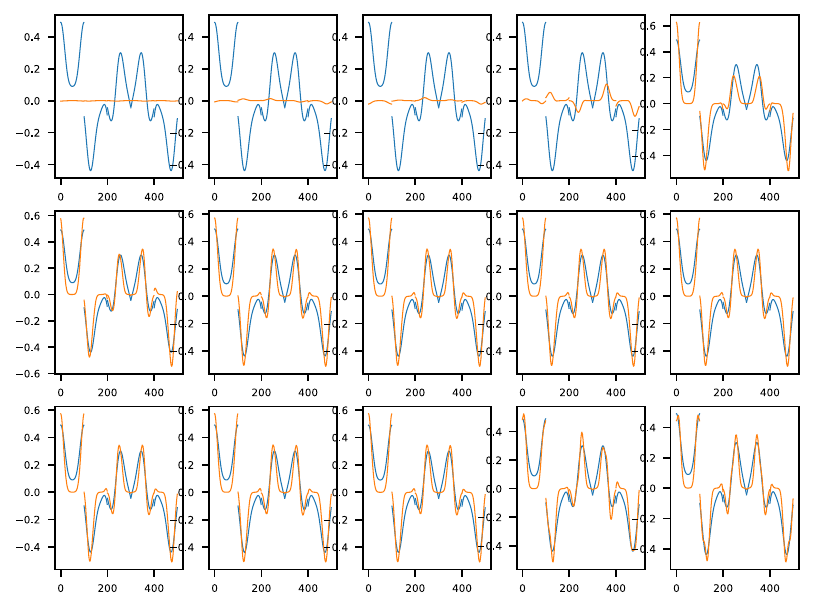}
\label{Fig12c}}
\subfigure[]{
\includegraphics[scale=0.6]{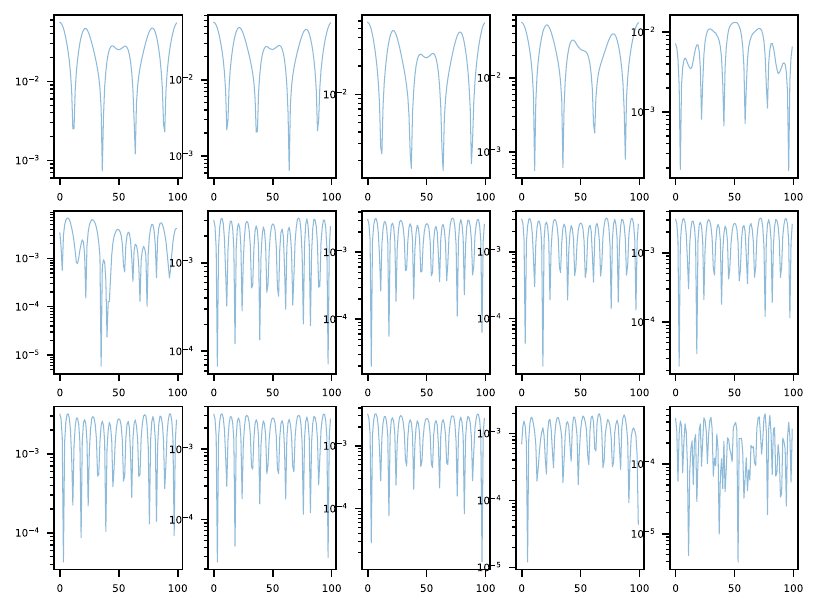}
\label{Fig12f}}
\vspace{-10pt}
\caption[Figure]{\footnotesize \Cref{ex:ex04}: \subref{Fig12c}~One sample of the Monte Carlo estimator $u^{(r)}_{\omega^1, \varepsilon, N}$, $r = \overline{1,15}$, at $\big( \mathbf{x}_j^1 \big)_{j = \overline{1, M_1}} \subset \Gamma_1$, and the exact solution $u^{\rm (ex)}\big|_{\Gamma_1}$ given by \eqref{eq:2D-sol3} (blue continuous line), and \subref{Fig12f}~the absolute errors $\big\vert u_i^D - \widehat{u}_i^{D,r} \big\vert$, $i = \overline{1,M_D}$, $r = \overline{1,15}$.}
\label{Fig12cf}
\end{figure}
\vspace{-20pt}

\begin{figure}[H]
\centering
\subfigure[]{
\includegraphics[scale=0.6]{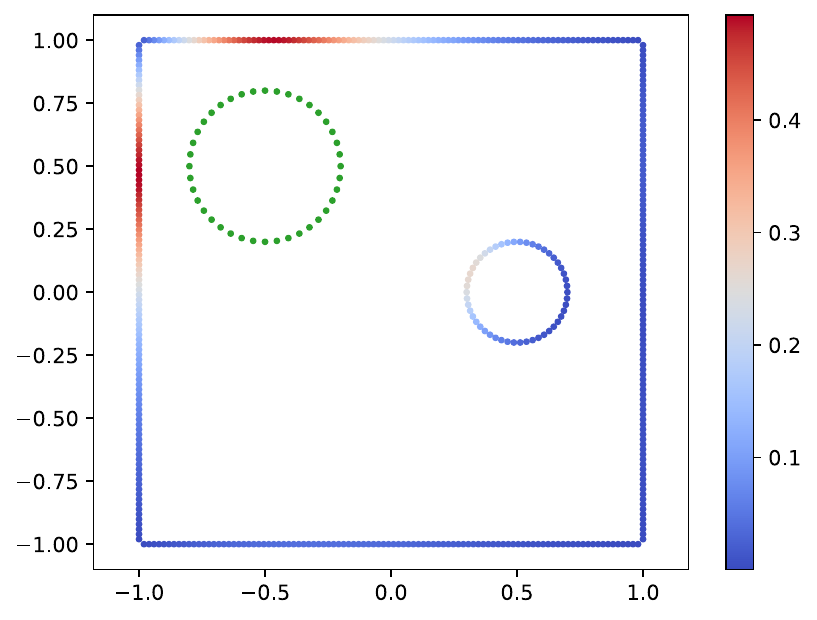}
\label{Fig13a}}
\subfigure[]{
\includegraphics[scale=0.55]{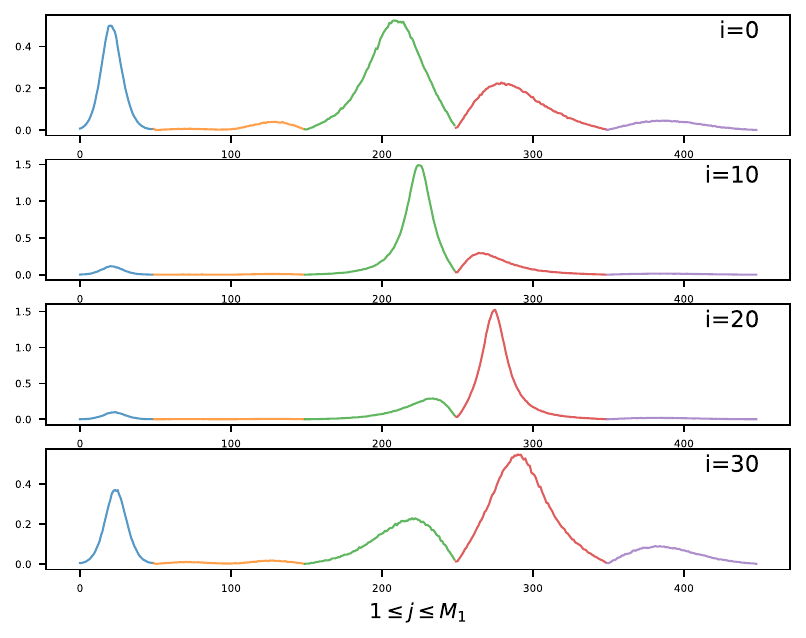}
\label{Fig13c}}
\subfigure[]{
\includegraphics[scale=0.55]{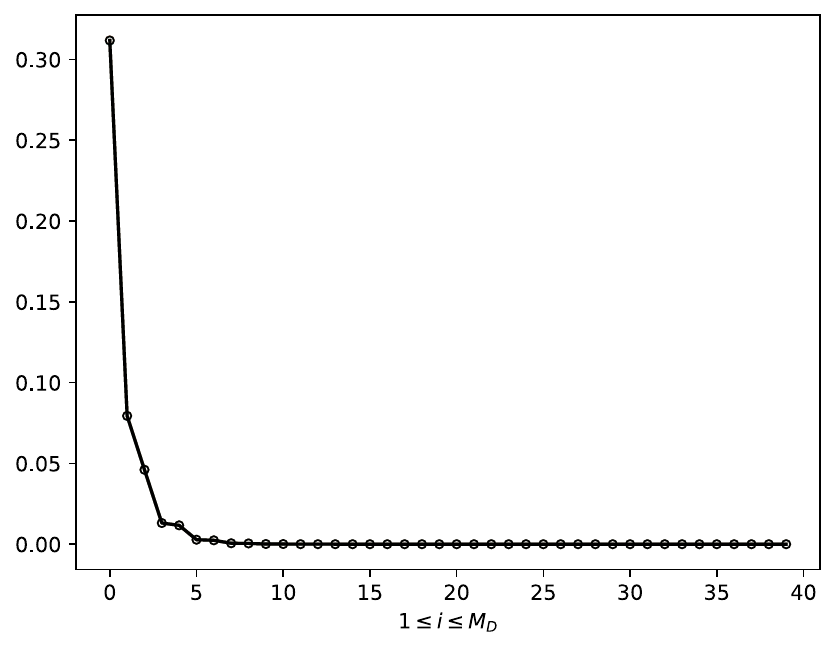}
\label{Fig13d}}
\subfigure[]{
\includegraphics[scale=0.6]{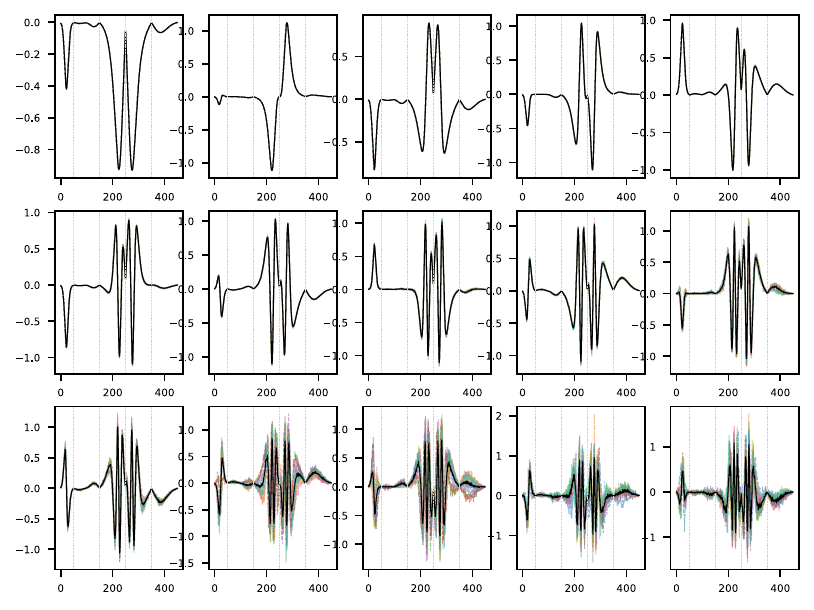}
\label{Fig13f}}
\vspace{-10pt}
\caption[Figure]{\footnotesize \Cref{ex:ex05}: \subref{Fig13a}~One sample of $\dfrac{1}{M_D} \displaystyle \sum_{i=1}^{M_D} \rho_{\mathbf{x}_i^D, \omega^1, \varepsilon, N}\big( \mathbf{x}_j^1 \big)$, retrieved on $\Gamma_1$. 
\subref{Fig13c}~One sample of the elliptic density estimators $\rho_{\mathbf{x}_i^D, \omega^1, \varepsilon, N}\big( \mathbf{x}_j^1 \big)$, for $j \in \big\{ 0, 10, 20, 30 \big\}$. 
\subref{Fig13d}~Ten i.i.d. samples of the eigenvalues $\widetilde{\lambda}^{\boldsymbol{\nu}}_i$, $i = \overline{1,M_D}$, of the Monte Carlo estimator $\boldsymbol{\Lambda}^{\boldsymbol{\nu}}_{\omega^1, \varepsilon, N}$, represented on a linear scale. 
\subref{Fig13f}~Ten i.i.d. samples of the approximate eigenfunction estimator $\widetilde{u}_\ell$, $\ell = \overline{1,15}$, and their empirical mean. 
}
\label{Fig13}
\end{figure}
\vspace{-10pt}

\begin{figure}[H]
\centering
\subfigure[]{
\includegraphics[scale=0.6]{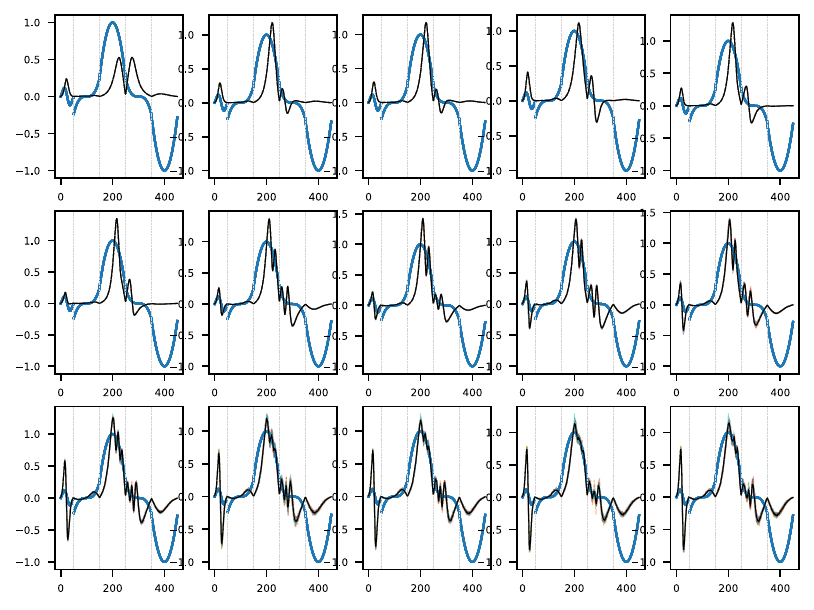}
\label{Fig14a}}
\subfigure[]{
\includegraphics[scale=0.6]{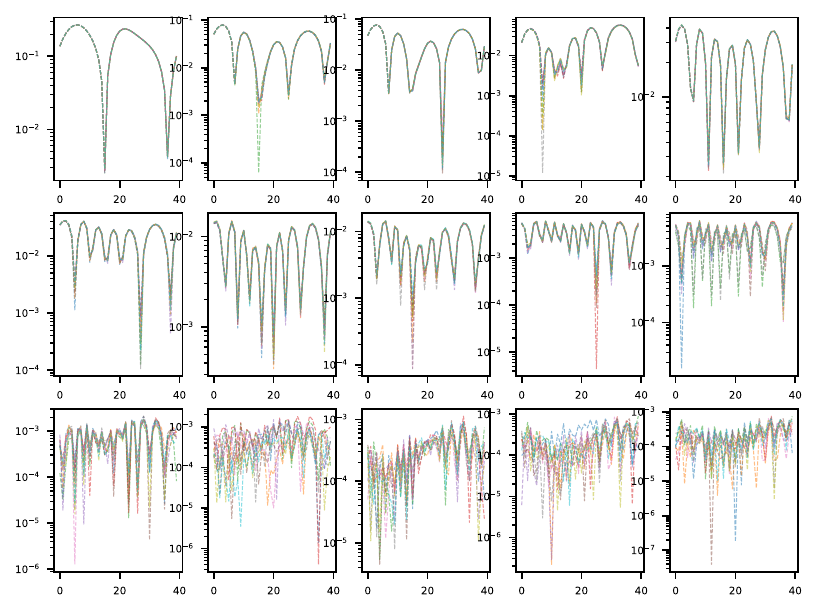}
\label{Fig14d}}
\vspace{-10pt}
\caption[Figure]{\footnotesize \Cref{ex:ex05}: \subref{Fig14a}~Ten i.i.d. samples of the Monte Carlo estimator $u^{(r)}_{\omega^1, \varepsilon, N}$, $r = \overline{1,15}$, at $\big( \mathbf{x}_j^1 \big)_{j = \overline{1, M_1}} \subset \Gamma_1$, their empirical mean (black continuous line), and the exact solution $u^{\rm (ex)}\big|_{\Gamma_1}$ given by \eqref{eq:2D-sol1} (blue continuous line), and \subref{Fig14d}~the absolute errors $\big\vert u_i^D - \widehat{u}_i^{D,r} \big\vert$, $i = \overline{1,M_D}$, $r = \overline{1,15}$.}
\label{Fig14ad}
\end{figure}
\vspace{-20pt}

\begin{figure}[H]
\centering
\subfigure[]{
\includegraphics[scale=0.6]{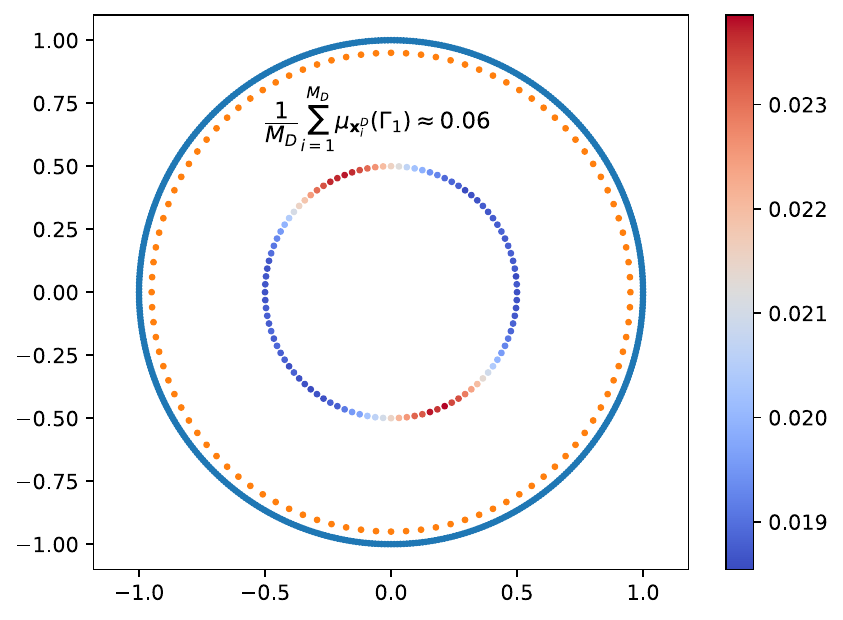}
\label{Fig15a}}
\subfigure[]{
\includegraphics[scale=0.55]{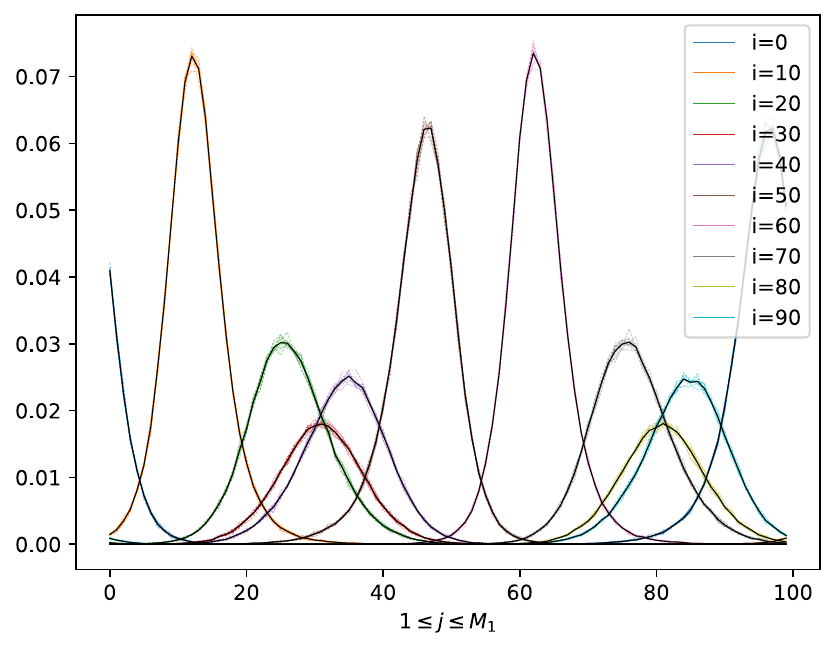}
\label{Fig15c}}
\subfigure[]{
\includegraphics[scale=0.55]{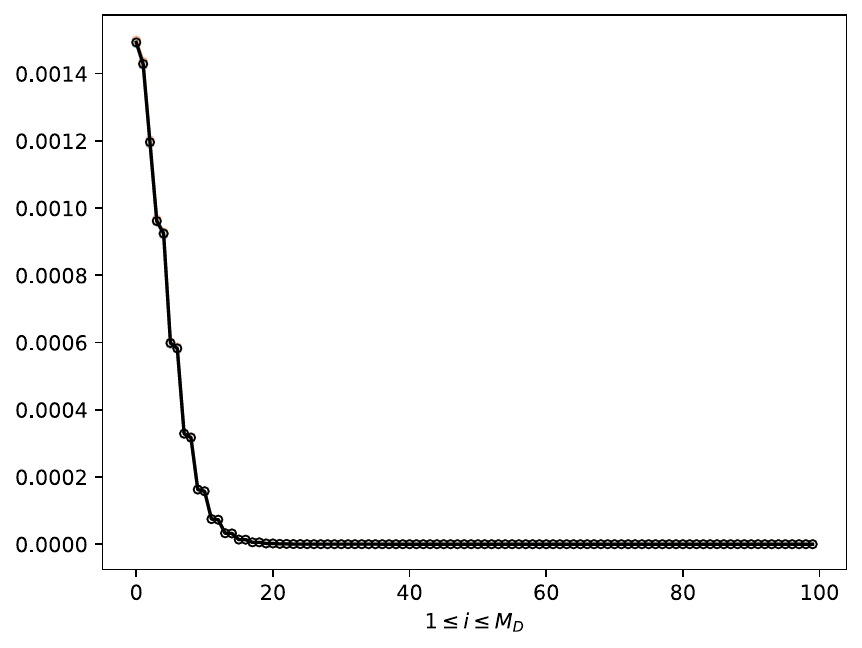}
\label{Fig15d}}
\subfigure[]{
\includegraphics[scale=0.6]{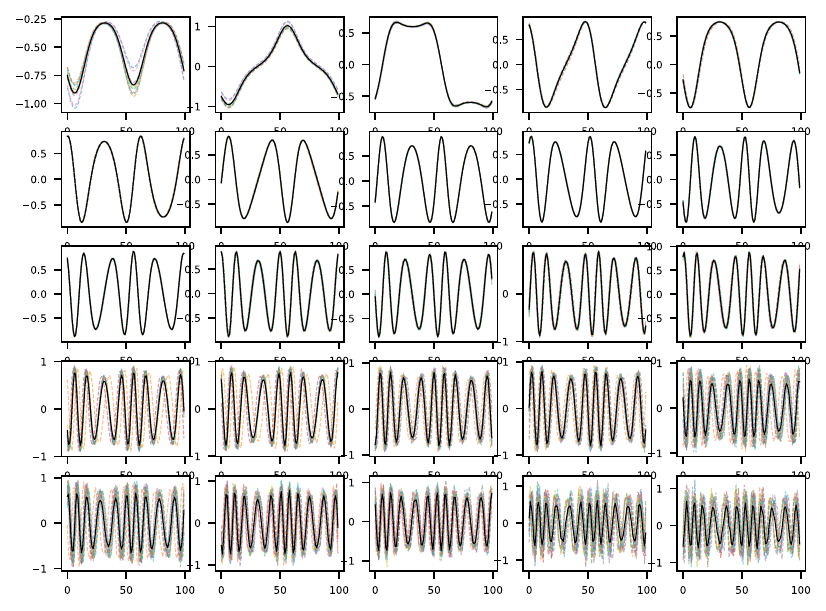}
\label{Fig15f}}
\vspace{-10pt}
\caption[Figure]{\footnotesize \Cref{ex:ex02}: \subref{Fig15a}~One sample of $\dfrac{1}{M_D} \displaystyle \sum_{i=1}^{M_D} \rho_{\mathbf{x}_i^D, \omega^1, \varepsilon, N}\big( \mathbf{x}_j^1 \big)$, retrieved on $\Gamma_1$. 
\subref{Fig15c}~Ten i.i.d. samples of the elliptic density estimators $\rho_{\mathbf{x}_i^D, \omega^1, \varepsilon, N}\big( \mathbf{x}_j^1 \big)$ and their empirical mean (black continuous line), for $j \in \big\{ 0, 10, 20, \ldots, 90 \big\}$. 
\subref{Fig15d}~Ten i.i.d. samples of the eigenvalues $\widetilde{\lambda}^{\boldsymbol{\nu}}_i$, $i = \overline{1,M_D}$, of the Monte Carlo estimator $\boldsymbol{\Lambda}^{\boldsymbol{\nu}}_{\omega^1, \varepsilon, N}$ and their empirical mean (black continuous line), represented on a linear scale. 
\subref{Fig15f}~Ten i.i.d. samples of the approximate eigenfunction estimator $\widetilde{u}_\ell$, $\ell = \overline{1,25}$, and their empirical mean (black continuous line). 
}
\label{Fig15}
\end{figure}
\vspace{-10pt}

\begin{figure}[H]
\centering
\subfigure[]{
\includegraphics[scale=0.55]{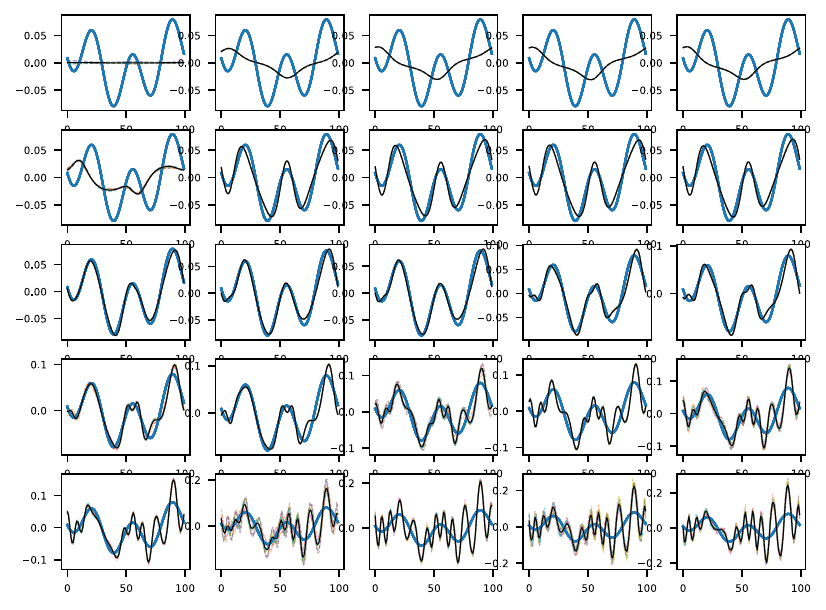}
\label{Fig16a}}
\subfigure[]{
\includegraphics[scale=0.55]{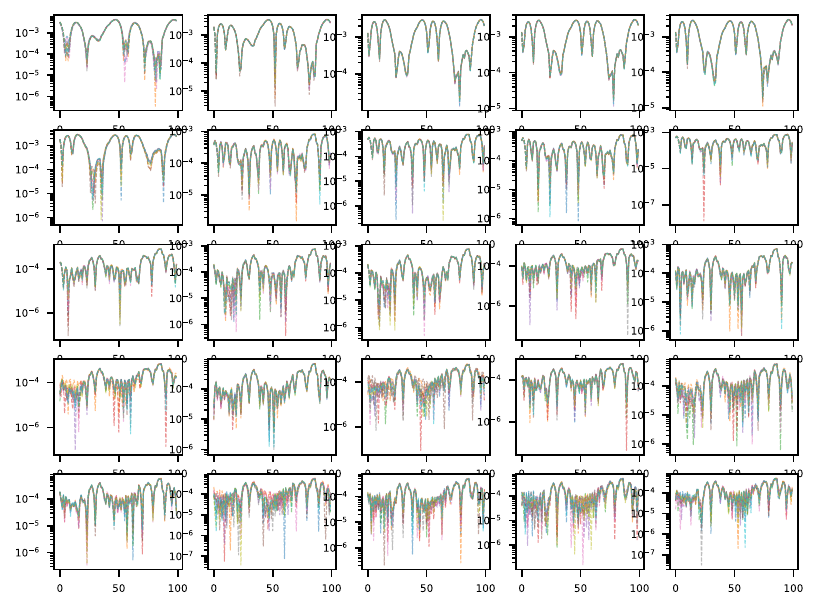}
\label{Fig16b}}
\vspace{-10pt}
\caption[Figure]{\footnotesize \Cref{ex:ex02}: \subref{Fig16a}~Ten i.i.d. samples of the Monte Carlo estimator $u^{(r)}_{\omega^1, \varepsilon, N}$, $r = \overline{1,25}$, at $\big( \mathbf{x}_j^1 \big)_{j = \overline{1, M_1}} \subset \Gamma_1$, their empirical mean (black continuous line), and the exact solution $u^{\rm (ex)}\big|_{\Gamma_1}$ given by \eqref{eq:2D-sol4} (blue continuous line), and \subref{Fig16b}~the absolute errors $\big\vert u_i^D - \widehat{u}_i^{D,r} \big\vert$, $i = \overline{1, M_D}$, $r = \overline{1, 25}$.}
\label{Fig16}
\end{figure}

\begin{figure}[H]
\centering
\subfigure[]{
\includegraphics[scale=0.6]{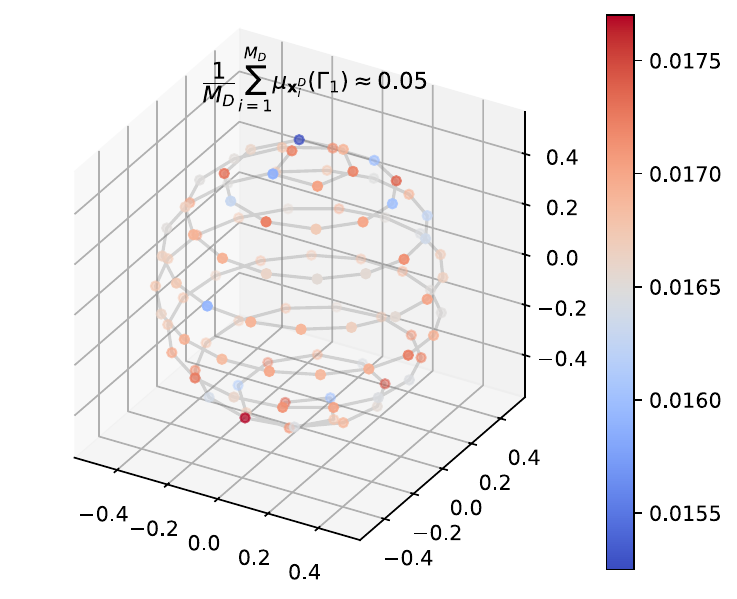}
\label{Fig17a}}
\subfigure[]{
\includegraphics[scale=0.55]{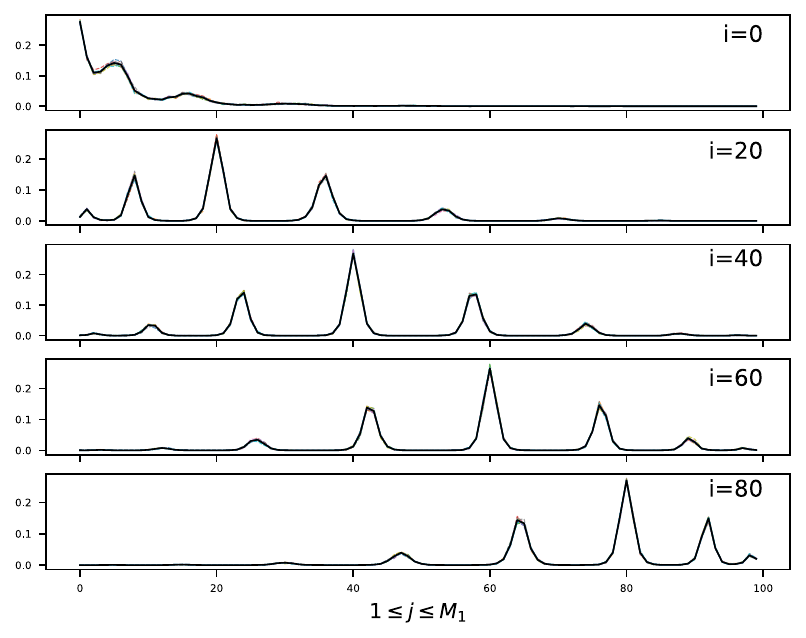}
\label{Fig17c}}
\subfigure[]{
\includegraphics[scale=0.55]{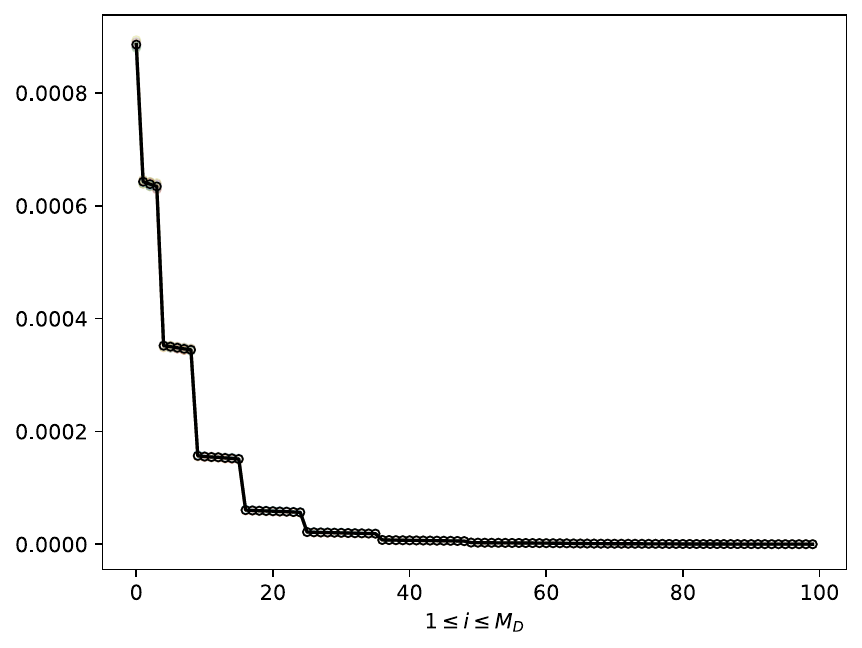}
\label{Fig17d}}
\subfigure[]{
\includegraphics[scale=0.6]{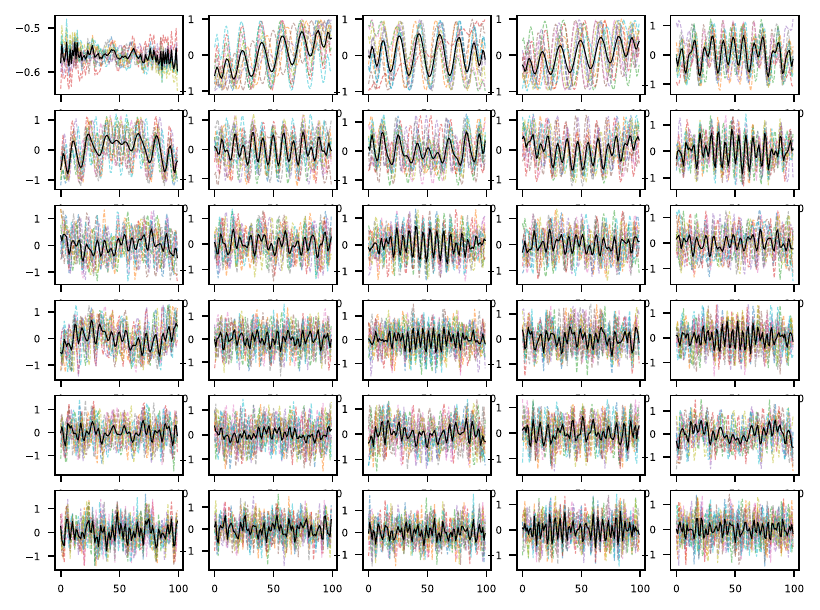}
\label{Fig17f}}
\vspace{-10pt}
\caption[Figure]{\footnotesize \Cref{ex:ex06}: \subref{Fig17a}~One sample 
of $\dfrac{1}{M_D} \displaystyle \sum_{i=1}^{M_D} \rho_{\mathbf{x}_i^D, \omega^1, \varepsilon, N}\big( \mathbf{x}_j^1 \big)$, retrieved on $\Gamma_1$.
\subref{Fig17c}~Ten i.i.d. samples of the elliptic density estimators $\rho_{\mathbf{x}_i^D, \omega^1, \varepsilon, N}\big( \mathbf{x}_j^1 \big)$ and their empirical mean (black continuous line), for $j \in \big\{ 0, 20, 40, 60, 80 \big\}$. 
\subref{Fig17d}~Ten i.i.d. samples of the eigenvalues $\widetilde{\lambda}^{\boldsymbol{\nu}}_i$, $i = \overline{1,M_D}$, of the Monte Carlo estimator $\boldsymbol{\Lambda}^{\boldsymbol{\nu}}_{\omega^1, \varepsilon, N}$ and their empirical mean (black continuous line), represented on a linear scale. 
\subref{Fig17f}~Ten i.i.d. samples of the approximate eigenfunction estimator $\widetilde{u}_\ell$, $\ell = \overline{1,30}$, and their empirical mean (black continuous line).
}
\label{Fig17}
\end{figure}
\vspace{-10pt}

\begin{figure}[H]
\centering
\subfigure[]{
\includegraphics[scale=0.6]{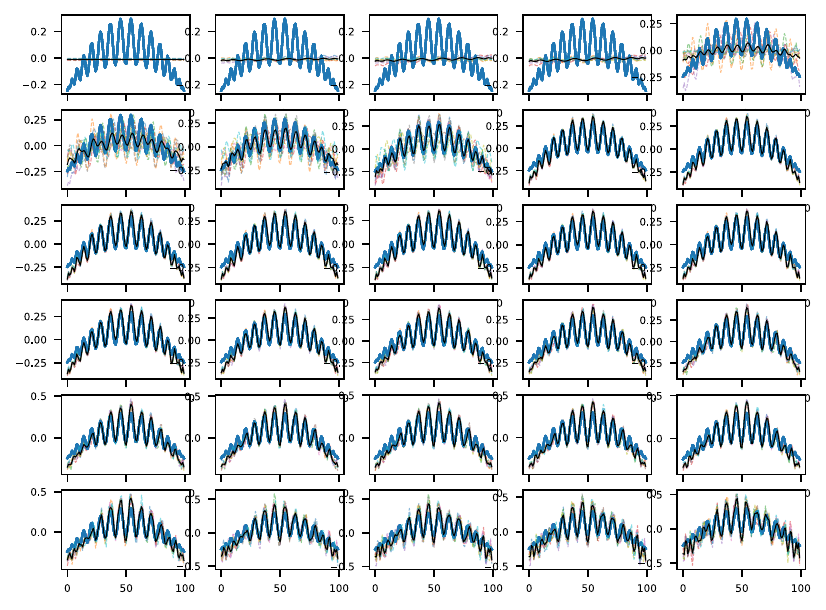}
\label{Fig18a}}
\subfigure[]{
\includegraphics[scale=0.6]{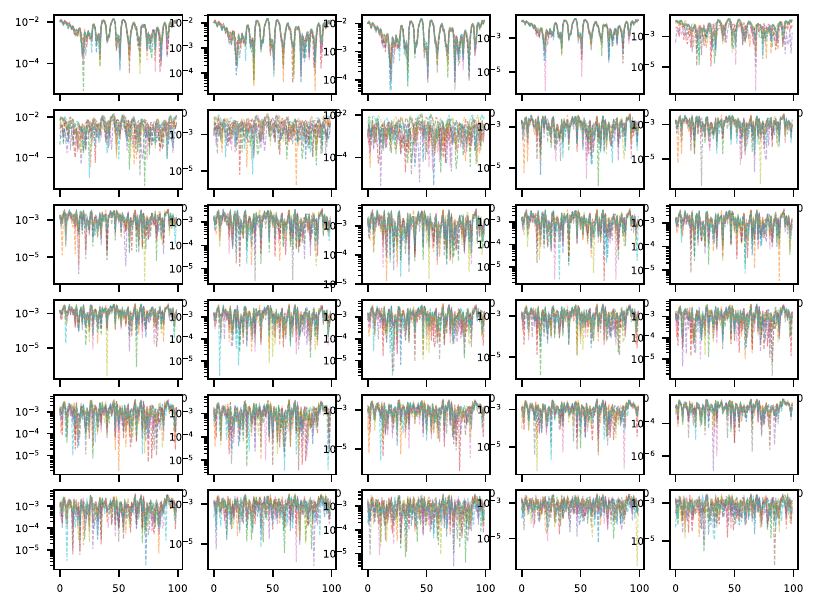}
\label{Fig18d}}
\vspace{-10pt}
\caption[Figure]{\footnotesize \Cref{ex:ex06}: \subref{Fig18a}~Ten i.i.d. samples of the Monte Carlo estimator $u^{(r)}_{\omega^1, \varepsilon, N}$, $r = \overline{1,30}$, at $\big( \mathbf{x}_j^1 \big)_{j = \overline{1, M_1}} \subset \Gamma_1$, their empirical mean (black continuous line), and the exact solution $u^{\rm (ex)}\big|_{\Gamma_1}$ given by \eqref{eq:3D-sol1} (blue continuous line), and \subref{Fig18d}~the absolute errors $\big\vert u_i^D - \widehat{u}_i^{D,r} \big\vert$, $i = \overline{1,M_D}$, $r = \overline{1,30}$.}
\label{Fig18ad}
\end{figure}
\vspace{-20pt}

\begin{figure}[H]
\centering
\subfigure[]{
\includegraphics[scale=0.6]{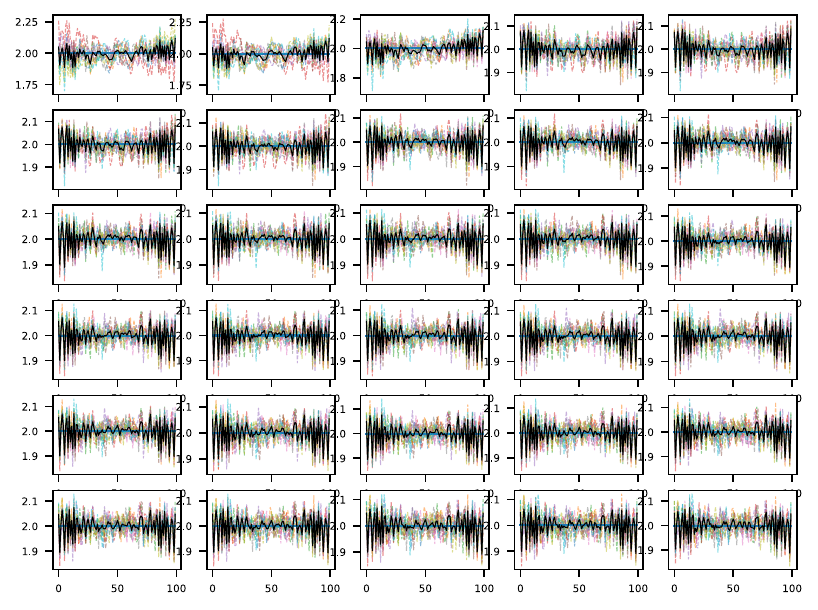}
\label{Fig18b}}
\subfigure[]{
\includegraphics[scale=0.6]{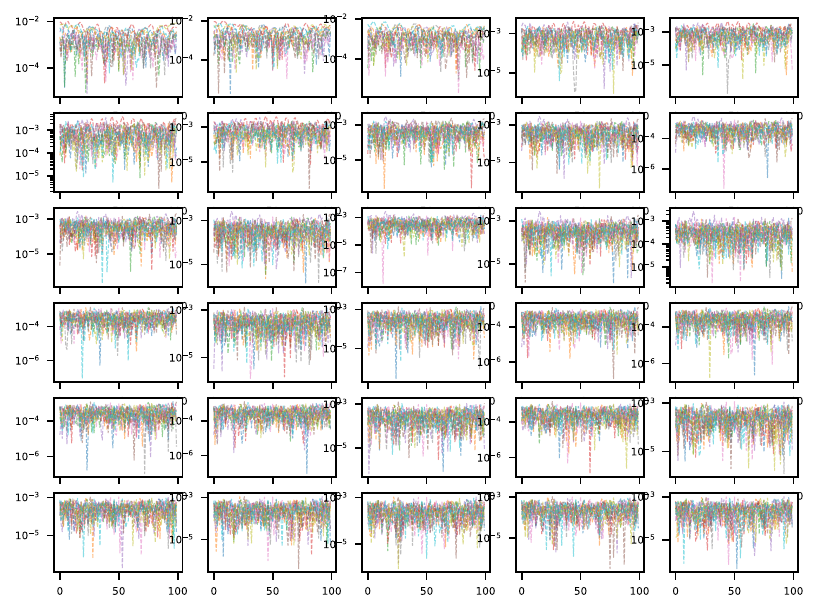}
\label{Fig18e}}
\vspace{-10pt}
\caption[Figure]{\footnotesize \Cref{ex:ex06}: \subref{Fig18b}~Ten i.i.d. samples of the Monte Carlo estimator $u^{(r)}_{\omega^1, \varepsilon, N}$, $r = \overline{1,30}$, at $\big( \mathbf{x}_j^1 \big)_{j = \overline{1, M_1}} \subset \Gamma_1$, their empirical mean (black continuous line), and the exact solution $u^{\rm (ex)}\big|_{\Gamma_1}$ given by \eqref{eq:3D-sol2} (blue continuous line), and \subref{Fig18e}~the absolute errors $\big\vert u_i^D - \widehat{u}_i^{D,r} \big\vert$, $i = \overline{1,M_D}$, $r = \overline{1,30}$.}
\label{Fig18be}
\end{figure}
\vspace{-20pt}

\begin{figure}[H]
\centering
\subfigure[]{
\includegraphics[scale=0.6]{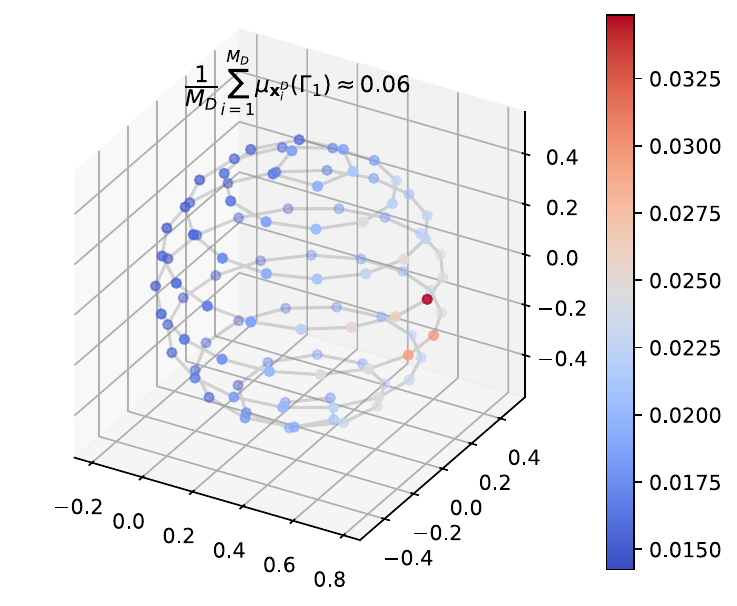}
\label{Fig19a}}
\subfigure[]{
\includegraphics[scale=0.55]{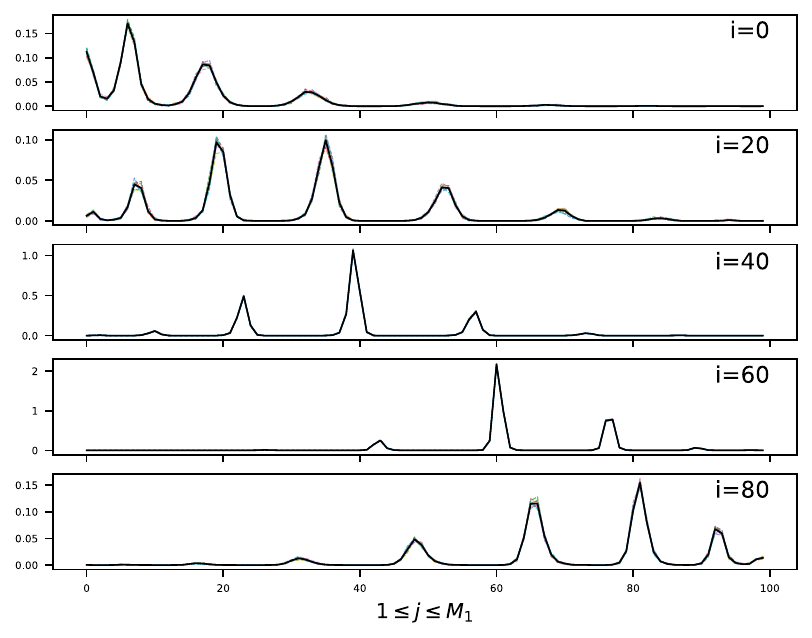}
\label{Fig19c}}
\subfigure[]{
\includegraphics[scale=0.55]{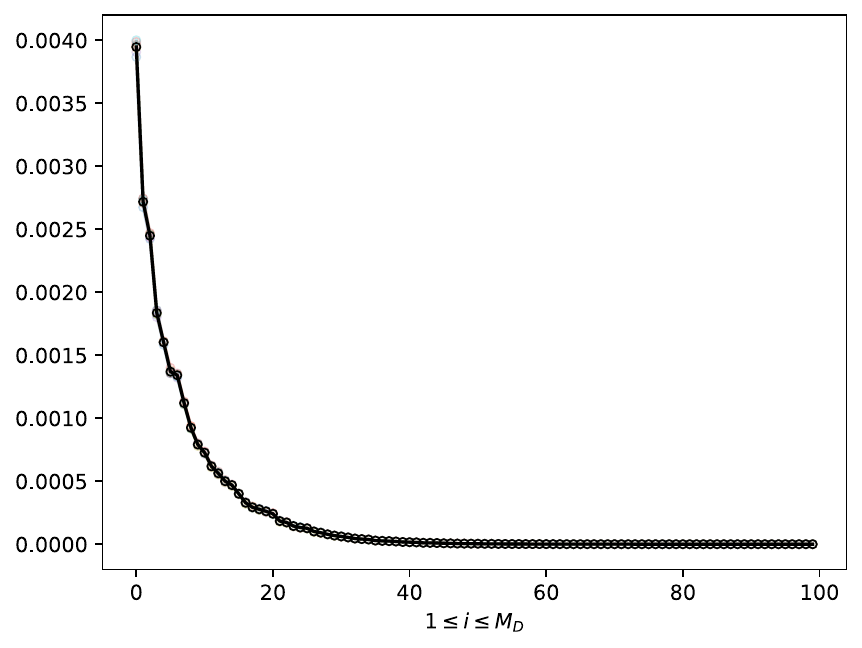}
\label{Fig19d}}
\subfigure[]{
\includegraphics[scale=0.6]{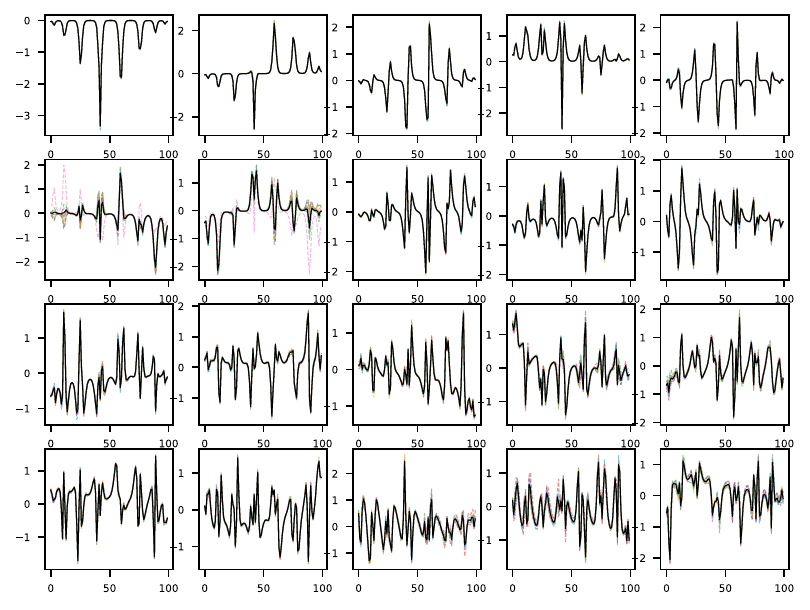}
\label{Fig19f}}
\vspace{-10pt}
\caption[Figure]{\footnotesize \Cref{ex:ex07}: \subref{Fig19a}~One sample of $\dfrac{1}{M_D} \displaystyle \sum_{i=1}^{M_D} \rho_{\mathbf{x}_i^D, \omega^1, \varepsilon, N}\big( \mathbf{x}_j^1 \big)$, retrieved on $\Gamma_1$. 
\subref{Fig19c}~Ten i.i.d. samples of the elliptic density estimators $\rho_{\mathbf{x}_i^D, \omega^1, \varepsilon, N}\big( \mathbf{x}_j^1 \big)$ and their empirical mean (black continuous line), for $j  \in \big\{ 0, 20, 40, 60, 80 \big\}$. 
\subref{Fig19d}~Ten i.i.d. samples of the eigenvalues $\widetilde{\lambda}^{\boldsymbol{\nu}}_i$, $i = \overline{1,M_D}$, of the Monte-Carlo estimator $\boldsymbol{\Lambda}^{\boldsymbol{\nu}}_{\omega^1, \varepsilon, N}$ and their empirical mean (black continuous line), represented on a linear scale. 
\subref{Fig19f}~Ten i.i.d. samples of the approximate eigenfunction estimator $\widetilde{u}_\ell$, $\ell = \overline{1,20}$, and their empirical mean (black continuous line).
}
\label{Fig19}
\end{figure}
\vspace{-10pt}
\begin{figure}[H]
\centering
\subfigure[]{
\includegraphics[scale=0.6]{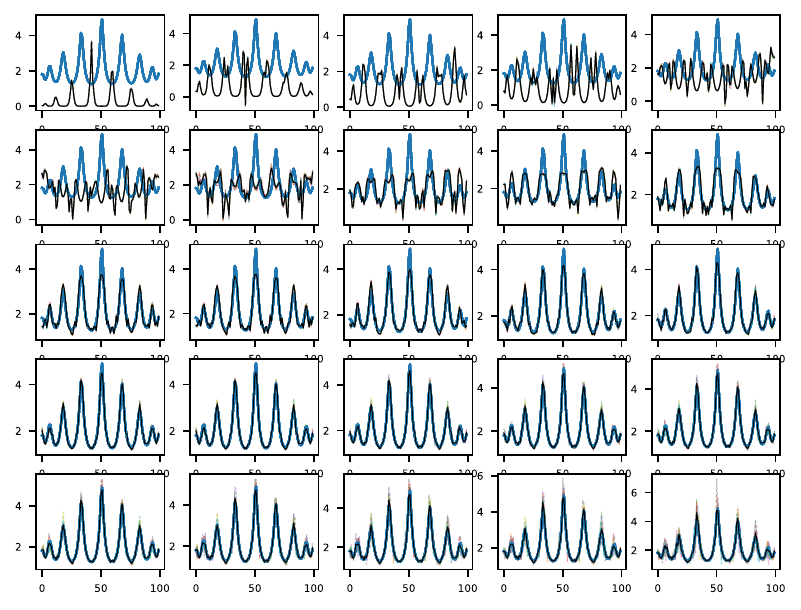}
\label{Fig20b}}
\subfigure[]{
\includegraphics[scale=0.6]{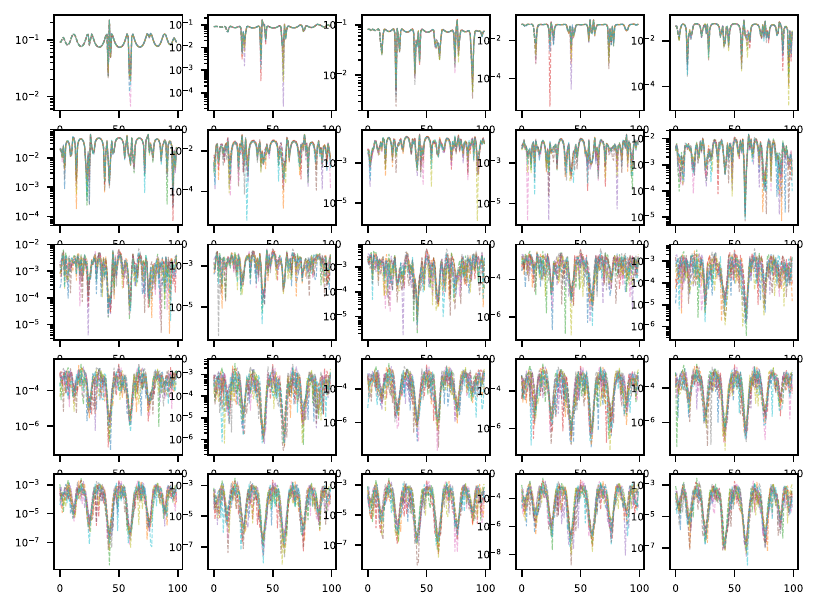}
\label{Fig20e}}
\vspace{-10pt}
\caption[Figure]{\footnotesize \Cref{ex:ex07}: \subref{Fig20b}~Ten i.i.d. samples of the Monte Carlo estimator $u^{(r)}_{\omega^1, \varepsilon, N}$, $r \in \{1, 3, \ldots, 75 \}$, at $\big( \mathbf{x}_j^1 \big)_{j = \overline{1, M_1}} \subset \Gamma_1$, their empirical mean (black continuous line), and the exact solution $u^{\rm (ex)}\big|_{\Gamma_1}$ given by \eqref{eq:3D-sol2} (blue continuous line), and \subref{Fig20e}~the absolute errors $\big\vert u_i^D - \widehat{u}_i^{D,r} \big\vert$, $i = \overline{1,M_D}$, $r \in \{1, 3, \ldots, 75 \}$.}
\label{Fig20be}
\end{figure}
\vspace{-20pt}

\end{document}